\documentclass[11pt]{aart}

\usepackage[letterpaper, hmargin=1.1in, top=1.1in, bottom=1.1in, footskip=0.5in]{geometry}

\usepackage{titlesec}
\titleformat{\section}[block]{\filcenter\normalfont\bfseries\large}{\thesection.}{.5em}{}\titlespacing*{\section}{0pt}{2\baselineskip}{1\baselineskip}
\titleformat{\subsection}[runin]{\normalfont\bfseries}{\thesubsection.}{.4em}{}[.]\titlespacing{\subsection}{0pt}{2ex plus .1ex minus .2ex}{.8em}
\titleformat{\subsubsection}[runin]{\normalfont\itshape}{\thesubsubsection.}{.3em}{}[.]\titlespacing{\subsubsection}{0pt}{1ex plus .1ex minus .2ex}{.5em}
\titleformat{\paragraph}[runin]{\normalfont\itshape}{\theparagraph.}{.3em}{}[.]\titlespacing{\paragraph}{0pt}{1ex plus .1ex minus .2ex}{.5em}

\usepackage[labelfont=bf,font=small,labelsep=period]{caption}
\usepackage[T1]{fontenc}
\usepackage{lmodern}

\let\originalleft\left
\let\originalright\right
\renewcommand{\left}{\mathopen{}\mathclose\bgroup\originalleft}
\renewcommand{\right}{\aftergroup\egroup\originalright}

\newcommand{\msp}[1][]{\mathrel{\phantomsection{#1}}{}}
\usepackage{tensor}

\usepackage{amsmath}
\usepackage{amssymb}
\usepackage{amsfonts}
\usepackage{latexsym}
\usepackage{amsthm}
\usepackage{amsxtra}
\usepackage{amscd}
\usepackage{bbm}
\usepackage{mathrsfs}
\usepackage{bm}
\usepackage{mathtools} 

\usepackage{graphicx, color}

\definecolor{darkred}{rgb}{0.9,0,0.3}
\definecolor{darkblue}{rgb}{0,0.3,0.9}
\definecolor{orange}{rgb}{0.98, 0.6, 0.01}
\definecolor{darkgreen}{rgb}{0, 0.65, 0.1}

\definecolor{vdarkred}{rgb}{0.6,0,0.2}
\definecolor{vdarkblue}{rgb}{0,0.2,0.6}
\usepackage[colorlinks, linkcolor=vdarkblue,citecolor=vdarkred]{hyperref}

\usepackage[nottoc,notlof,notlot]{tocbibind}
\usepackage{cite} 

\numberwithin{equation}{section}
\numberwithin{figure}{section}

\theoremstyle{plain} 
\newtheorem{theorem}{Theorem}[section]
\newtheorem*{theorem*}{Theorem}
\newtheorem{lemma}[theorem]{Lemma}
\newtheorem*{lemma*}{Lemma}

\newtheorem*{corollary*}{Corollary}
\newtheorem{proposition}[theorem]{Proposition}
\newtheorem*{proposition*}{Proposition}

\newtheorem*{conjecture*}{Conjecture}

\theoremstyle{definition} 
\newtheorem{definition}[theorem]{Definition}
\newtheorem*{definition*}{Definition}
\newtheorem{example}[theorem]{Example}
\newtheorem*{example*}{Example}
\newtheorem{remark}[theorem]{Remark}
\newtheorem*{remark*}{Remark}

\newtheorem*{assumption*}{Assumption}

\newcommand{\f}[1]{\boldsymbol{\mathrm{#1}}} 
\renewcommand{\r}{\mathrm}  
\newcommand{\bb}{\mathbb} 
\renewcommand{\cal}{\mathcal}

\newcommand{\ul}[1]{\underline{#1} \!\,} 
\newcommand{\ol}[1]{\overline{#1} \!\,} 

\newcommand{\txt}[1]{\text{\rm{#1}}}

\newcommand{\E}{\mathbb{E}}

\newcommand{\R}{\mathbb{R}}
\newcommand{\C}{\mathbb{C}}
\newcommand{\N}{\mathbb{N}}

\newcommand{\cov}{\mathrm{Cov}}

\newcommand{\ii}{\mathrm{i}}
\newcommand{\dd}{\mathrm{d}}
\newcommand{\col}{\vcentcolon}
\newcommand*{\deq}{\mathrel{\vcenter{\baselineskip0.65ex \lineskiplimit0pt \hbox{.}\hbox{.}}}=}
\newcommand*{\eqd}{=\mathrel{\vcenter{\baselineskip0.65ex \lineskiplimit0pt \hbox{.}\hbox{.}}}}
\newcommand{\eqdist}{\overset{\txt{d}}{=}}

\renewcommand{\leq}{\leqslant}
\renewcommand{\le}{\leqslant}
\renewcommand{\geq}{\geqslant}
\renewcommand{\ge}{\geqslant}
\renewcommand{\epsilon}{\varepsilon}
\newcommand{\OO}{\r O}
\newcommand{\oo}{\r o}

\newcommand{\ceil}[1]  {\lceil  {#1} \rceil}
\renewcommand{\prescript}[2]{\smash{\tensor[^{#1}]#2{}}}

\newcommand{\pb}[1]{\bigl({#1}\bigr)}

\newcommand{\pbb}[1]{\biggl({#1}\biggr)}
\newcommand{\pBB}[1]{\Biggl({#1}\Biggr)}

\newcommand{\abs}[1]{\lvert #1 \rvert}

\newcommand{\norm}[1]{\lVert #1 \rVert}

\newcommand{\avg}[1]{\langle #1 \rangle}

\DeclareMathOperator{\tr}{Tr}
\DeclareMathOperator{\var}{Var}
\DeclareMathOperator{\supp}{supp}
\DeclareMathOperator{\re}{Re}
\DeclareMathOperator{\im}{Im}

\DeclareMathOperator{\sgn}{sgn}

\title{Mesoscopic eigenvalue density correlations of Wigner matrices}

\author{Yukun He \and Antti Knowles}

\begin{document}

\maketitle

\begin{abstract}
We investigate to what extent the microscopic Wigner-Gaudin-Mehta-Dyson (WGMD) (or sine kernel) statistics of random matrix theory remain valid on mesoscopic scales. To that end, we compute the connected two-point spectral correlation function of a Wigner matrix at two mesoscopically separated points. In the mesoscopic regime, density correlations are much weaker than in the microscopic regime. Our result is an explicit formula for the two-point function. This formula implies that the WGMD statistics are valid to leading order on all mesoscopic scales, that in the real symmetric case there are subleading corrections matching precisely the WGMD statistics, while in the complex Hermitian case these subleading corrections are absent. We also uncover non-universal subleading correlations, which dominate over the universal ones beyond a certain intermediate mesoscopic scale. The proof is based on a hierarchy of Schwinger-Dyson equations for a sufficiently large class of polynomials in the entries of the Green function. The hierarchy is indexed by a tree, whose depth is controlled using stopping rules. A key ingredient in the derivation of the stopping rules is a new estimate on the density of states, which we prove to have bounded derivatives of all order on all mesoscopic scales.
\end{abstract}

\section{Introduction}

\subsection{Local spectral statistics of Wigner matrices}
Let $\lambda_1, \dots, \lambda_N$ be the eigenvalues of an $N \times N$ Wigner matrix $H$ -- a Hermitian random matrix with independent upper-triangular entries with zero expectation and constant variance.
We normalize $H$ so that as $N \to \infty$ its spectrum converges to the interval $[-2,2]$, and therefore its typical eigenvalue spacing is of order $N^{-1}$. In this paper we study the eigenvalue density-density correlations at two different energies. The eigenvalue density is given by the eigenvalue process $\sum_i \delta_{\lambda_i}$, which we analyse using its correlation functions\footnote{As customary, to simplify notation somewhat, we regard $\rho_1$ and $\rho_2$ as functions (densities with respect to Lebesgue measure) although the right-hand sides of \eqref{def_rho} are measures. They will always be integrated against continuous test functions, justifying such an abuse of notation.}
\begin{equation} \label{def_rho}
\rho_1(x) \deq \E \sum_i \delta(x - \lambda_i)\,, \qquad \rho_2(x,y) \deq \E \sum_{i \neq j} \delta(x - \lambda_i) \delta(y - \lambda_j).
\end{equation}
The central object of our analysis is the connected (or truncated) two-point function
\begin{equation*}
p(x,y) \deq \rho_2(x,y) - \rho_1(x) \rho_1(y)\,.
\end{equation*}
It measures eigenvalue density-density correlations at the energies $x$ and $y$. The behaviour of $p$ on small energy scales $x - y$ as $N \to \infty$ is one of the central questions of random matrix theory, ever since the seminal works of Dyson, Gaudin, and Mehta \cite{Gau,Meh,Dys}.

It is well known that the expected normalized eigenvalue density $\frac{1}{N} \rho_1(E)$ at energy $E$ is asymptotically given by the semicircle density
\begin{equation*}
\varrho_E \deq \frac{1}{2\pi}\sqrt{(4-E^2)_+}\,.
\end{equation*}
Hence, the mean eigenvalue spacing of the process $\sum_i \delta_{\lambda_i}$ at energy $E$ is $\frac{1}{N \varrho_E}$. In order to analyse the eigenvalue density correlations at small energy scales, we choose a reference energy $E \in (-2,2)$ and replace the eigenvalue process $\sum_i \delta_{\lambda_i}$ with the rescaled process $\sum_i \delta_{N \varrho_E(\lambda_i - E)}$ on the microscopic scale around energy $E$, whose mean eigenvalue spacing is one. The connected two-point function of the rescaled process $\sum_i \delta_{N \varrho_E(\lambda_i - E)}$, denoted by $p_E$, may be expressed in terms of $p$ as
\begin{equation} \label{1.3}
p_E(u,v) = \frac{1}{(N\varrho_E)^2} \, p\bigg(E+\frac{u}{N\varrho_E},E+\frac{v}{N\varrho_E}\bigg)\,.
\end{equation}

The asymptotic behaviour of $p_E$ was analysed in the works of Dyson, Gaudin, and Mehta \cite{Gau,Meh,Dys} for Gaussian $H$. They proved that
\begin{equation} \label{p_conv_GUE}
\lim_{N \to \infty} p_E(u,v) = Y_\beta(u - v)
\end{equation}
weakly in $u,v \in \R$, where $Y_\beta$ is given in the complex Hermitian case ($H = \mathrm{GUE}$, $\beta = 2$) by
\begin{equation} \label{Y_complex}
Y_2(u) \deq - s(u)^2
\end{equation}
and in the real symmetric case ($H = \mathrm{GOE}$, $\beta = 1$) by
\begin{equation} \label{Y_real}
Y_1(u) \deq - s'(u) \int_u^\infty s(v) \, \dd v - s(u)^2\,.
\end{equation}
Here
\begin{equation*}
s(u) \deq \frac{\sin (\pi u)}{\pi u}
\end{equation*}
is the \emph{sine kernel}. In the real symmetric case, $Y_1$ from \eqref{Y_real} has the asymptotic expression \cite[Equation (7.2.45)]{Meh}
\begin{equation} \label{Y_expansion}
Y_1(u) = -\frac{1}{\pi^2 u^2} + \frac{1 + \cos^2 (\pi u)}{\pi^4 u^4} + \OO\pbb{\frac{1}{u^6}}
\end{equation}
for $u \to \infty$.

In fact, the weak convergence of $p_E$ to $Y_\beta$ holds not only for the Gaussian matrices GUE and GOE but also for arbitrary Wigner matrices. This universality statement is commonly referred to as the Wigner-Gaudin-Mehta-Dyson conjecture, and it was proved in a recent series of works \cite{Agg16, ESY1, ESY2, ESY3, ESY4,ESY5,EPRSY, EPRSTY, ESYY, EYY1, EYY2, EYY3, TV1,TV2, EKYY1, EKYY2, BEYY14}.

\subsection{Results} \label{sec:outline_results}
In this paper we investigate the eigenvalue density-density correlation between energies $u$ and $v$ that are much further apart than the microscopic scale considered above. Thus, we analyse $p_E(u,v)$ on \emph{mesoscopic scales}, where $u - v$ is much larger than the microscopic scale $1$.  On mesoscopic scales, the correlations are much weaker than on microscopic scales, as can be easily guessed from the right-hand side of \eqref{p_conv_GUE} which behaves like $(u - v)^{-2}$ for $u - v \gg 1$. We investigate to what extent the asymptotic behaviour $p_E(u,v) \approx Y_\beta(u - v)$, valid on the microscopic energy scale, remains correct on mesoscopic energy scales.

As explained around \eqref{def_rho}, $p_E(u,v)$ is in general a measure and not a function, and therefore has to be integrated with respect to continuous test functions to yield a number, just as the convergence \eqref{p_conv_GUE} is in terms of the weak convergence of measures. (Pointwise convergence of $p_E(u,v)$ is in general false and even meaningless.) Thus, we test $p_E$ against some test functions $f$ and $g$ and consider the limiting behaviour of
\begin{equation} \label{p_E_convolved}
\int \dd \tilde u \, \dd \tilde v \, p_E(\tilde u, \tilde v) \, f(u - \tilde u) \, g(v - \tilde v)
\end{equation}
instead. We therefore have another choice of scale -- that of the averaging test functions $f$ and $g$ (where scale means e.g.\ the diameter of their support).
From the expressions \eqref{Y_complex} and \eqref{Y_real} it is clear that $p_E(u,v)$ oscillates on the microscopic energy scale $u - v \asymp 1$, which leads to rapid oscillations on the mesoscopic scale $u - v \gg 1$ (see Figure \ref{fig:Y}). In order to obtain a well-defined limiting behaviour of $p_E$ on mesoscopic scales, we therefore get rid of the oscillations by choosing the scale of the averaging functions $f$ and $g$ to be mesoscopic. Thus, we have two energy scales: the energy separation $u - v$, denoted by $2 N \omega$, and the averaging scale, denoted by $N \eta$. See Figure \ref{fig:spectrum}. One can think of $N \eta$ as being slightly larger than $1$, although other window scales are of interest too.

\begin{figure}[!ht]
\begin{center}
{\small 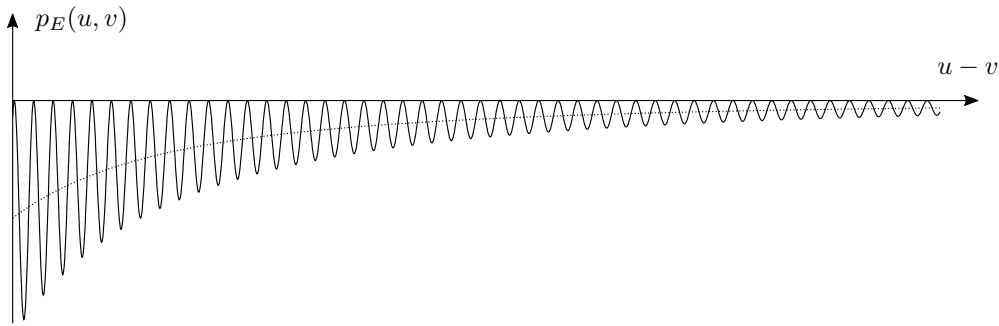}
\end{center}
\caption{A plot of $p_E(u,v)$ as a function of $u - v$ on mesoscopic scales $u - v \gg 1$ (solid), along with the function where the microscopic oscillations have been averaged out over a small energy window (dotted). \label{fig:Y}}
\end{figure}

\begin{figure}[!ht]
\begin{center}
{\small 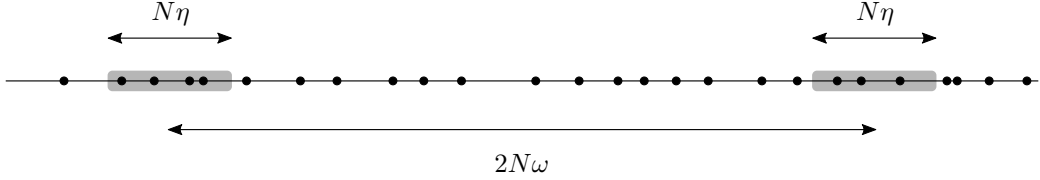}
\end{center}
\caption{The rescaled eigenvalue process $\sum_i \delta_{N \varrho_E(\lambda_i - E)}$ at some energy $E$, where each atom is represented using a dot. The typical spacing of the dots is $1$. The spectral widows of width $N \eta$ located around $N \omega$ and $-N \omega$ respectively are drawn in grey. \label{fig:spectrum}}
\end{figure}

Next, we state our main result. We consider the following class of random matrices.

\begin{definition} \label{def:wigner}
A \emph{Wigner matrix} is a Hermitian $N\times N$ matrix $H=H^{*} \in \bb C^{N \times N}$ whose entries $H_{ij}$ satisfy the following conditions.
\begin{enumerate}
	\item The upper-triangular entries ($H_{ij} \col 1 \le i \le j\le N$) are independent.
	\item We have $\bb E H_{ij}=0$ for all $i,j$.
	\item For each $p \in \bb N$ there exists a constant $C_p$ such that $\bb E|\sqrt{N}H_{ij}|^p \le C_p$ for all $i,j,N$.  
\end{enumerate}
We distinguish\footnote{We remark that the condition on the variances of the diagonal entries is imposed for convenience, and it can be easily relaxed; see Appendix \ref{sec:diag_general} for details.} the real symmetric case ($\beta=1$), where $H_{ij} \in \bb R$ and $\bb E|\sqrt{N}H_{ij}|^2=1+\delta_{ij}$ for all $i,j$, and the complex Hermitian case ($\beta=2$), where $\bb E|\sqrt{N}H_{ij}|^2=1$ for all $i,j$, and $\bb EH^2_{ij}=0$ for all $i\ne j$.
\end{definition}

We use $C^{\infty}_{c}(\bb R)$ to denote the set of real-valued smooth functions with compact support.

\begin{theorem} \label{mainthm} \label{thm:main}
Fix $0 < \tau < 1$. Let $E \in [-2+\tau,2-\tau]$ and abbreviate $\kappa_E \deq 2\pi \varrho_E=\sqrt{4-E^2}$. Let $f,g \in C^{\infty}_{c}(\bb R)$, and fix $M\geq1$ such that $\supp f$, $\supp g \subset [-M,M]$. Let $\eta,\omega$ satisfy $N^{-1+\tau}\leq \eta$, $3M\eta \le \omega\le \tau/3$. 
For any function $f$ on $\bb R$ we denote
\begin{equation*} 
f_{\pm}(u)\deq \frac{1}{N\eta }f\bigg(\frac{u\mp N\omega }{N\eta }\bigg)\,.
\end{equation*}
Let $p_E(u,v)$ be defined as in \eqref{1.3}, with $H$ as in Definition \ref{def:wigner}. We have
\begin{equation*}
\int p_{E}(u,v) f_{+}(u)g_{-}(v)  \,\mathrm{d}u\,\mathrm{d}v= \int \pb{\Upsilon_{E,\beta}(u,v) + \cal E(u,v)} \,f_{+}(u)g_{-}(v)\,\mathrm{d}u\,\mathrm{d}v\,,
\end{equation*}
where for the real symmetric case ($\beta=1$), 
	\begin{multline} \label{2.6}
		\Upsilon_{E,1}(u,v)\deq -\frac{1}{\pi^2(u-v)^2}+\frac{3}{2\pi^4(u-v)^4}
		\\
		+\frac{1}{N^2\kappa_E^2}\pbb{F_1(u,v)+F_2(u,v)\sum_{i,j}\cal C_4(H_{ij})+F_3(u,v)\sum_{i}\cal C_3({H_{ii})}}\,,
	\end{multline}
	and for the complex Hermitian case ($\beta=2$),
	\begin{equation} \label{2.7}
		\Upsilon_{E,2}(u,v)\deq-\frac{1}{2\pi^2(u-v)^2}
		+\frac{1}{N^2\kappa_E^2}\pbb{\frac{1}{2}F_1(u,v)+F_2(u,v)\sum_{i,j}\cal C_{2,2}(H_{ij})+F_3(u,v)\sum_{i}\cal C_3({H_{ii})}}\,.
	\end{equation}
	Here, $\cal C_3, \cal C_4, \cal C_{2,2}$ denote cumulants (see \eqref{cumulant_real} and \eqref{cumulant_complex} below), $F_1(u,v),...,F_3(u,v)$ are elementary bounded functions defined in \eqref{FFF} below, and 
\begin{equation} \label{def_calE}
\cal E(u,v)= \OO \pbb{\frac{1}{(u-v)^5}+\frac{1}{N^2(u-v)}}
\end{equation}
is an error term, where the implicit constants in $\OO(\cdot )$ only depend on $f,g,\tau$, and $C_p$ from Definition \ref{def:wigner}.	
\end{theorem}

More informally, Theorem \ref{thm:main} states that for a fixed energy $E \in (-2,2)$ and mesoscopic scales $1 \ll N \eta \leq N \omega \leq N$, the density-density correlation function $p_E(u,v)$ at the energies $u = N \omega$ and $v = - N \omega$, averaged in $u$ and $v$ over windows of size $N \eta$, is equal to $\Upsilon_{E, \beta}(u,v)$ up to the small error term $\cal E(u,v)$.

The first term of $\Upsilon_{E,2}(u,v)$ is precisely obtained by averaging out the microscopic fluctuations of $Y_2(u - v)$ from \eqref{Y_complex}. Similarly, the first two terms of $\Upsilon_{E,1}(u,v)$ are precisely obtained by averaging out the main terms of the expansion \eqref{Y_expansion} for $Y_1(u - v)$. The remaining terms of $\Upsilon_{E,1}(u,v)$ and $\Upsilon_{E,2}(u,v)$ are non-universal deviations (depending on the high-order cumulants of $H$) from the microscopic Wigner-Gaudin-Mehta-Dyson statistics appearing on mesoscopic scales. Such deviations have been previously observed in the macroscopic linear statistics for Wigner matrices, where $\omega$ and $\eta$ are both of order $1$ \cite{KKP95, KKP96,LP}. See Figure \ref{fig:graph} for a plot of the various terms of $\Upsilon_{E,1}(u,v)$.

\begin{figure}[!ht]
\begin{center}
{\small 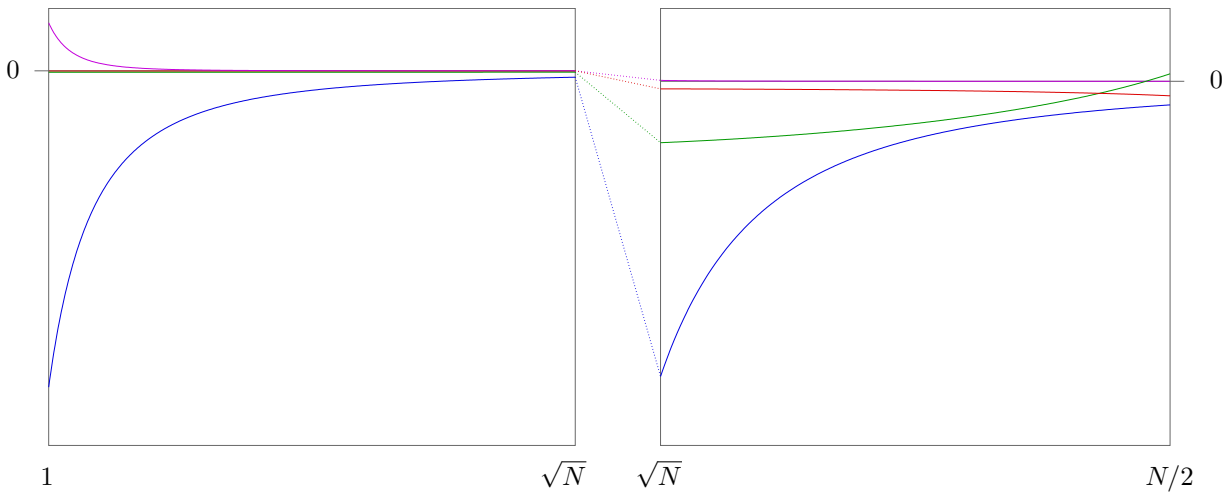}
\end{center}
\caption{A plot of the three terms of $\Upsilon_{E,1}(u,v)$ for $N = 50$. We choose $E = 0$ and $v = 0$, and plot the three terms of $\Upsilon_{E,1}(u,v)$ as a function of $u \in [1,N/2]$. For clarity, we split the plot into two pieces at $u = \sqrt{N}$. Blue: first term. Purple: second term. Red: third term for i.i.d.\ Gaussian entries. Green: third term for i.i.d.\ Bernoulli entries. The blue and and purple curves are universal, while the red and green ones are not. The plot ranges from $u = 1$ (microcscopic scale) to $u = N/2$ (macroscopic scale). On the microscopic scale, the blue and purple curves (arising from the sine kernel) are dominant and of the same order. On the intermediate mesoscopic scale $u = \sqrt{N}$, the purple, red, and green curves are of the same order. On the macroscopic scale, the blue, red, and green curves are of the same order. The non-universality of the third term of $\Upsilon_{E,1}(u,v)$ is apparent in the discrepancy between the red and green curves. The blue curve is dominant on all mesoscopic scales, but not on the microscopic ($u = 1$) and macroscopic ($u = N/2$) scales. \label{fig:graph}} 
\end{figure}

In fact, for the extreme (doubly) macroscopic case $N \eta \asymp N \omega \asymp N$, our main result (captured by $\Upsilon_{E,\beta}(u,v)$) reproduces the covariance derived in \cite{LP}. Indeed, it was proved in \cite{LP} that (for instance) for $\beta = 1$ and i.i.d.\ entries we have
\begin{multline} \label{macrosopic_variance}
\var \sum_i f(\lambda_i) = \frac{1}{2 \pi^2}\int_{-2}^2 \dd x \int_{-2}^2 \dd y \, \pbb{\frac{f(x) - f(y)}{x - y}}^2 \frac{4 - xy}{\sqrt{4 - x^2} \sqrt{4 - y^2}}
\\
+ \frac{N^2 \cal C_4(H_{12})}{2\pi^2} \pbb{\int_{-2}^2 \dd x \, f(x) \, \frac{2 - x^2}{\sqrt{4 - x^2}}}^2\,,
\end{multline}
which, by polarization and a short calculation, is easily seen (see Appendix \ref{sec:linstat}) to match with the corresponding quantity obtained from the leading order of $\Upsilon_{E,1}(u,v)$ on the macroscopic scale. 
Thus, Theorem \ref{thm:main} encompasses the macroscopic regime \eqref{macrosopic_variance} and, at the opposite end of smallest mesoscopic scales, it is consistent with the microscopic regime \eqref{p_conv_GUE} and \eqref{Y_real}. In this sense, it relates the microscopic and macroscopic regimes as endpoints of a continuous mesoscopic landscape.

We make some observations on the expressions $\Upsilon_{E,\beta}(u,v)$.
\begin{enumerate}
\item
To leading order, the Wigner-Gaudin-Mehta-Dyson correlations remain valid on all mesoscopic scales $N \omega \ll N$, failing only at the macroscopic scale $\omega \asymp 1$. (The first term of $\Upsilon_{E,2}(u,v)$, respectively of $\Upsilon_{E,1}(u,v)$, is the leading term if and only if $\omega \ll 1$. If $\omega \asymp 1$ then the second term of $\Upsilon_{E,2}(u,v)$, respectively the third term of $\Upsilon_{E,1}(u,v)$, is also of leading order.)
\item
For $\beta = 1$, the subleading corrections inherent to the universal Wigner-Gaudin-Mehta-Dyson statistics from \eqref{Y_expansion} are valid up to mesoscopic scales $N \omega \ll \sqrt{N}$. For $N \omega \gg \sqrt{N}$, the dominant subleading corrections arise from the third term of $\Upsilon_{E,1}(u,v)$, and they depend on the fourth cumulants of $H$. These non-universal corrections become of leading order at macroscopic scales $N \omega \asymp N$.
\item
Our result is completely insensitive to size of the spectral window $N \eta$ over which the function $p_E$ is averaged, provided that $N \eta \gg 1$. In particular, even at the macroscopic scale $N \omega \asymp N$, our result on the covariance is stronger than the previous macroscopic results \cite{KKP95, KKP96,LP}, where also $N \eta \asymp N$, whereas we admit any $\eta$ satisfying $N \eta \gg 1$. For instance, our result covers the density-density correlations of $N^{\oo(1)}$ eigenvalues located at opposite ends of the bulk spectrum\footnote{Here $\oo(1)$ means for any fixed (small) $\varepsilon>0$.}.
\item
The scaling factor $N \varrho_E$ of the rescaled point process (see \eqref{1.3}) is chosen so that its limiting behaviour does not depend on $E$ on the microscopic scale. Accordingly, the first term of $\Upsilon_{E,2}(u,v)$ and the first two terms of $\Upsilon_{E,1}(u,v)$ do not depend on $E$. All other terms of $\Upsilon_{E,1}(u,v)$ and $\Upsilon_{E,2}(u,v)$, however, carry a factor $\varrho_E^{-2}$, which has the interpretation that these terms are naturally associated with the macroscopic scale, where a rescaling by $\varrho_E$ makes no sense.
\item
Our method can be extended to yield further terms in the $(u - v)^{-2}$-expansion in $\Upsilon_{E,1}(u,v)$, with a corresponding higher power of $(u - v)^{-1}$ in $\cal E(u,v)$, which match the corresponding further terms in the expansion of \eqref{Y_real}, as in \eqref{Y_expansion}. For the sake of brevity we refrain from doing so in this paper.
\end{enumerate}

It is instructive to rewrite $p_E$ as follows. The density-density correlation function $p_E(u,v)$ at the energies $u$ and $v$, averaged in $u$ and $v$ over windows of size $N \eta$, can be expressed directly in terms of the eigenvalues of $H$ as
\begin{equation} \label{cov_intro}
\frac{1}{(N \varrho_E)^2} \, \cov \pBB{\sum_i f^{\eta/\varrho_E}\pbb{E + \frac{u}{N \varrho_E} - \lambda_i} \,,\, \sum_j f^{\eta/\varrho_E}\pbb{E + \frac{v}{N \varrho_E} - \lambda_j}}\,,
\end{equation}
where $f^\epsilon(x) \deq \frac{1}{\epsilon} f\pb{\frac{x}{\epsilon}}$ is an approximate delta function on scale $\epsilon$ with a compactly supported positive test function $f$ having integral one. Note that the term $\sum_i f^{\eta/\varrho_E}\pb{E + \frac{u}{N \varrho_E} - \lambda_i}$ has variance of order $\eta^{-2}$ (see e.g.\ \cite{HK}), much larger (by a factor up to $N^2$) than the covariance in \eqref{cov_intro} which is of order $\omega^{-2} = \pb{\frac{u - v}{2N}}^{-2}$. This illustrates a key difficulty in deriving expressions of the form $\Upsilon_{E,\beta}(u,v)$ for \eqref{cov_intro}, which entails computing the covariance of strongly fluctuating but weakly correlated random variables. Indeed, as explained in Section \ref{sec:sketch} below, even to obtain the leading order terms of $\Upsilon_{E,\beta}(u,v)$, we have to compute the covariance, of order $1/(N \omega)$, of two random variables whose variances are each of order $1/(N \eta)$. These two quantities may differ by a factor of up to $N^{1 - \oo(1)}$. We remark that in the regime $\omega \leq 3 M \eta$ not covered by Theorem \ref{thm:main}, the covariance is of the same order as the variances and the problem becomes much easier; it was addressed in \cite{HK}.

We conclude this subsection with previous results on the mesoscopic spectral statistics of Wigner matrices. The study of mesoscopic density-density eigenvalue correlations was initiated in \cite{Kho1} where the authors derived the first term of $\Upsilon_{E,1}(u,v)$, $-\frac{1}{\pi^2(u-v)^2}$, in the case when $H$ is GOE with the test function $f(x) = \frac{1}{\pi} \frac{1}{1 + x^2}$. In \cite{Kho2}, the same result was obtained for real symmetric Wigner matrices under the condition $\omega \gg \eta \gg N^{-1/8}$ and sufficiently many odd moments of the matrix entries vanish.

In \cite{G2005}, the eigenvalue distribution of GUE matrices (later extended to GOE matrices in \cite{O2010}) on mesoscopic scales was studied from a different point of view. Essentially, instead of computing correlations of eigenvalue \emph{densities}, \cite{G2005} addresses correlations of eigenvalue \emph{locations}. In order to summarize the results of \cite{G2005}, we order the eigenvalues $\lambda_1 \leq \lambda_2 \leq \cdots \leq \lambda_N$ and introduce the quantiles $\gamma_i$, $i = 1, \dots, N$, of the semicircle law through $i / N = \int_{-2}^{\gamma_i} \varrho_x \, \dd x$. Defining the normalized eigenvalues $\tilde \lambda_i \deq \frac{\pi \varrho_{\gamma_i}N(\lambda_i-\gamma_i)}{\sqrt{\log N}}$, it is proved in \cite{G2005} that $\tilde \lambda_i$ has asymptotically a normal distribution with variance $1$, and two normalized eigenvalues $\tilde \lambda_i$ and $\tilde \lambda_j$ are asymptotically jointly normal with covariance $\gamma$, where $\gamma = 1 - \log_N(j - i)$. Thus, the \emph{locations} of two mesoscopically separated eigenvalues have a covariance of the same order as their individual variances. This is in stark contrast to two mesoscopically separated \emph{densities}, whose correlations decrease rapidly with the separation. 

Heuristically, the strong correlations between eigenvalue locations can be attributed to the entire spectrum of the random matrix fluctuating as a semi-rigid jelly on the scale $\sqrt{ \log N}/N$. This global collective fluctuation of the eigenvalues is on a much larger scale than the typical eigenvalue spacing, $1/N$. Note that, unlike eigenvalue locations, the eigenvalue density is only weakly sensitive to a global shift of eigenvalues. We conclude that our results and those of \cite{G2005, O2010} pertain to different phenomena and are essentially independent. Although the density-density correlations of the form \eqref{cov_intro} are in principle completely characterized by the joint law of locations of pairs of mesoscopically separated eigenvalues, the precision required to see even the leading order of $\Upsilon_{E,\beta}(u,v)$ is far beyond that obtained in \cite{G2005, O2010}: the leading order of the density-density correlations arises from subleading position-position correlations. We refer to Appendix \ref{sec:Gustavsson} for a more detailed comparison of our results to \cite{G2005,O2010}.

\subsection*{Conventions}
We regard $N$ as our fundamental large parameter. Any quantities that are not explicitly constant or fixed may depend on $N$; we almost always omit the argument $N$ from our notation. Sometimes we use $c$ to denote a small positive constant. If the implicit constant in the notation $\OO(\cdot)$ depends on a parameter $\alpha$, we sometimes indicate this explicitly by writing $\OO_\alpha(\cdot)$.

\section{Outline of the proof} \label{sec:sketch}
As a starting point, we use an approximate Cauchy formula to reduce the problem from the case of general test functions to the case of the Green function $f(x)=(x-\ii)^{-1}$ (see Lemma \ref{lem5.1} and \eqref{n5.3}). The main work, therefore, is to prove the corresponding result for Green function, summarized in Theorem \ref{thm_resolvent} below. Let $G(z)\deq (H-z)^{-1}$ with $\im z \ne 0$ be the Green function, and denote by
\begin{equation*}
\ul{G}(z)\deq \frac{1}{N}\tr G(z)
\end{equation*}
the Stieltjes transform of the empirical spectral measure of $H$. The following result gives a precise computation of the covariance $\cov (\ul{G}(z_1),\ul{G}(z_2^*))$ for two mesoscopically separated spectral parameters $z_1$ and $z_2$.

\begin{theorem} \label{thm_resolvent}
Fix $\tau>0$. Let $E_1,E_2 \in [-2+\tau,2-\tau]$, $N^{-1+\tau} \leq \eta \leq \omega \leq \tau/3$, where $\omega\deq E_2-E_1>0$. Define $z_1\deq E_1+\mathrm{i}\eta$, $z_2 \deq E_2+\mathrm{i}\eta$. For the real symmetric case ($\beta=1$), we have 
		\begin{equation}	\label{2.4i}
	\begin{aligned}	
	\cov \big(\ul{G}(z_1), \ul{G}(z_2^*) \big) =&-\frac{2}{N^2(z_1-z_2^{*})^2}+\frac{f_1(z_1,z_2^*)}{N^3(z_1-z_2^{*})^3}+\frac{12}{N^4(z_1-z_2^{*})^4\kappa_E^2}
	\\&
	+\frac{1}{N^2}\Big(f_{2+}(z_1,z_2^*)+f_3(z_1,z_2^*)\sum_{i,j}\cal C_4(H_{ij})+f_4(z_1,z_2^*)\sum_{i}\cal C_3(H_{ii})\Big)\\
	&+\frac{\ii }{N^3(z_1-z_2^*)^2}\Big(-\frac{E}{\kappa_E^3}+V(E)\sum_{i,j}\cal C_4(H_{ij})\Big)+\cal E_1(\omega)\,,
	\end{aligned}
	\end{equation}
and for the complex Hermitian case ($\beta=2$), we have
	\begin{equation}	 \label{2.4ii}
	\begin{aligned}	
	&\,\cov \big(\ul{G}(z_1), \ul{G}(z_2^*) \big)\\ =&-\frac{1}{N^2(z_1-z_2^{*})^2}+\frac{1}{N^2}\Big(\,\frac{1}{2}f_{2+}(z_1,z_2^*)+f_3(z_1,z_2^*)\sum_{i,j}\cal C_{2,2}(H_{ij})+f_4(z_1,z_2^*)\sum_{i}\cal C_3(H_{ii})\Big)\\
	&+\frac{\ii V(E) }{2N^3(z_1-z_2^*)^2}\sum_{i,j}\cal C_{2,2}(H_{ij})+\cal E_1(\omega)\,.
	\end{aligned}
	\end{equation}
	Here $f_1(z_1,z_2^*),\dots,f_4(z_1,z_2^*)$, $V(E)$ are elementary bounded functions defined in \eqref{488}-\eqref{491} below, $V(E) \in \bb R$, and
	\begin{equation} \label{cal E_1}
	\cal E_1(\omega)=\OO\Big(\frac{1}{N^5\omega^5}+\frac{1}{N^3\omega}\Big)\,.
	\end{equation} 
	The implicit constants in $\OO(\cdot )$ only depend on $\tau$ and $C_p$ from Definition \ref{def:wigner}.

\begin{remark}
Analogously to Theorem \ref{thm_resolvent}, we also compute $\cov \big(\ul{G}(z_1), \ul{G}(z_2) \big)$. Indeed, we shall show that it is much smaller than $\cov \big(\ul{G}(z_1), \ul{G}(z_2^*) \big)$, in the sense that $\cov \big(\ul{G}(z_1), \ul{G}(z_2) \big)\asymp1/N^2$ while  $\cov \big(\ul{G}(z_1), \ul{G}(z_2^*) \big)\asymp 1/(N^2\omega^2)$. The precise statement is given in Proposition \ref{prop6.1} below.
\end{remark}
\end{theorem}

In the remainder of this section, we explain how to prove Theorem \ref{thm_resolvent}.
Define $\langle X \rangle \deq X-\bb E X$ for a random variable $X$ with finite expectation. Abbreviate $G=(H-z_1)^{-1}$, $F=(H-z_2)^{-1}$. The main work thus lies in computing $\bb E \langle \ul{G} \rangle \langle \ul{F^*}\rangle=\cov (\ul{G}(z_1),\ul{G}(z_2^*))$.

We give a sketch for the real symmetric case ($\beta=1$).  A fundamental source of difficulties throughout the proof is that we are computing the covariance of weakly correlated random variables, whose covariance is much smaller than their variances. More precisely, we compute the covariance of $G(z_1)$ and $G(z_2)$, which is of order $(N\omega)^{-2}$, while $G(z_1)$ and $G(z_2)$ each has a variance of order $(N\eta)^{-2}$. These two orders may differ up to a factor of $N^{2 - \oo(1)}$.

The basic idea behind the computation is to use resolvent identity $zG=HG-I$ to extract the centred random variables $H_{ij}$ from the expression, and then use the cumulant expansion formula (see Lemma \ref{lem:cumulant_expansion} below) 
\begin{equation} \label{simplecumulant}
\bb E f(h)h=\bb E|h|^2 \cdot \bb E f'(h)+\cdots
\end{equation}
for a centred random variable $h$. Unlike previous works (e.g.\ \cite{HK},\cite{LS1},\cite{HKR},\cite{EKS},\cite{HLY},\cite{H}), our central object $\bb E \langle \ul{G} \rangle \langle \ul{F^*}\rangle$ requires recursive cumulant expansions in order to control the error terms. The first expansion by \eqref{simplecumulant} leads to the Schwinger-Dyson (or self-consistent) equation
\begin{equation} \label{preeqn}
\bb E \langle \ul{G} \rangle \langle \ul{F^*} \rangle=\frac{1}{-z_1-2\bb E \ul{G}}\bigg(\frac{2}{N^2}\bb E \ul{GF^{*2}}+\frac{1}{N}\bb E \langle\ul{G^2}\rangle \langle \ul{F^*}\rangle+\bb E \langle \ul{G} \rangle^2 \langle \ul{F^*} \rangle+\cal W_1\bigg)\,,
\end{equation}
where $\cal W_1$ corresponds to the higher cumulant terms from the cumulant expansion (see \eqref{4.11} below for more details). The first term on right-hand side of \eqref{preeqn} is the leading term, and it can be shown of order $1/(N^2\omega^2)$. However, it is difficult to control the error terms in \eqref{preeqn}. For example, simply applying the local semicircle law (Proposition \ref{refthm1} below) to the third term on right-hand side of \eqref{preeqn}, we have
\begin{equation} \label{wwww}
\frac{1}{-z_1-2\bb E\ul{G}}\cdot \bb E \langle\ul{G}\rangle^2 \langle \ul{F^*}\rangle=\OO\Big(\frac{1}{N^3\eta^3}\Big)\,,
\end{equation}
which is in general not smaller than $1/(N^2\omega^2)$. This is because the large deviation estimate from the local semicircle law cannot make use of the fact that  $\langle\ul{G}\rangle^2$ and $\langle \ul{F^*} \rangle$ are weakly correlated. This difficulty is actually expected from the nature of the methodology. Trivially by the local semicircle law we have
$$
\bb E \langle \ul{G} \rangle \langle \ul{F^*}\rangle = \OO\Big(\frac{1}{N^2\eta^2}\Big)\,,
$$
and a comparison of the above and \eqref{wwww} shows that the bound from large deviation results is improved by a factor of $1/(N\eta)$ when we apply one cumulant expansion. The improvement $1/(N\eta)$ is consistent with previous works in random matrix theory (for example in the standard bootstrap argument for proving the local semicircle law, one expansion by the Schur's complement or cumulant expansion will lower the bound of $|G_{ij}-\delta_{ij}m|$ from 1 to $1/\sqrt{N\eta}$). The source of the factor $1/(N\eta)$ is the fundamental Ward identity, and it is the best we can get from any kind of single expansion. (We refer for instance to \cite[Section 3]{BK16} for a discussion of the usefulness of the Ward identity in random matrix theory.) In our context, in order to reveal the true size of the error terms in \eqref{preeqn}, we need to expand them recursively, until we obtain error bounds much smaller than $1 / (N^2 \omega^2)$. This procedure generates terms that either can be explicitly computed or are small enough owing to an accumulation of factors of $1/(N\eta)$.

The computation of $\bb E \langle \ul{G}\rangle \langle \ul{F^*}\rangle$ consists of three steps. In a first step, we show that
\begin{equation} \label{Ppeiqi}
\bb E \ul{G^m}=\OO(1)
\end{equation}
for any fixed $m\in \bb N$ (see Proposition \ref{lem4.4} below). This is a sharp and nontrivial bound (by the local semicircle law we could only get $\bb E\ul{G^m}=\OO(1/\eta^{m-1})$). As an interpretation of \eqref{Ppeiqi}, it implies that the normalized one-point function $\rho(x)\deq N^{-1}\sum_i\bb E \delta(x-\lambda_i)$, when smoothed out on any mesoscopic scale, has uniformly bounded derivatives of all order; this can be seen from the calculation
\begin{multline*}
\im \bb E \ul{G^m}=\im \int \frac{\rho(x) }{(x-z)^m}\, \dd x
= (-1)^{m-1}\im \int \frac{\rho^{(m-1)}(x)}{x-z}\, \dd x
\\
= (-1)^{m-1}\int \rho^{(m-1)}(x)\cdot \frac{\eta}{(x-E)^2+\eta^2} \, \dd x = (-1)^{m - 1} (\rho^{(m - 1)} * \theta^\eta)(E)\,,
\end{multline*}
with the Cauchy  kernel $\theta^\eta(x) = \frac{\eta}{\eta^2 + x^2}$ on scale $\eta$. Note that this estimate is optimal for $\eta\geq N^{-1+\oo(1)}$ and wrong on the microscopic scale $\eta = 1/N$. Indeed, even for the GUE, the derivative of the density of states diverges with $N$ on microscopic scales. The heuristic behind why \eqref{Ppeiqi} is nevertheless true is that the lack of boundedness of the derivatives of the density of states on microscopic scales arises from microscopic oscillations of the form $\frac{1}{N} r(E) \cos(N t(E))$ in the density of states where $r,t$ are smooth functions that do not depend on $N$. On mesoscopic scales, these oscillations are washed out and the resulting density has bounded derivatives of all order.

Now we have (see \eqref{Gm} for details)
\begin{equation} \label{P3.4}
\bb E \ul{G^m}=\frac{1}{-z_1-2\bb E \ul{G}}\bigg(\bb E \ul{G^{m-1}}+\sum_{a=1}^{m}\bb E \langle \ul{G^a}\rangle \langle\ul{G^{m+1-a}}\rangle+\sum_{a=2}^{m-1}\bb E \ul{G^a}\,\bb E\ul{G^{m+1-a}}+\frac{m}{N}\,\bb E \ul{G^{m+1}}+\cal W_2\bigg)\,,
\end{equation}
where $\cal W_2$ corresponds to the higher cumulant terms (see \eqref{Gm} below for more details). To be able to expand recursively, we need to have closed formulas for the terms in \eqref{P3.4}. Note that the first, third and fourth terms on the right-hand side of \eqref{P3.4} can already be re-expanded using \eqref{P3.4}. To deal with the second term on the right-hand side of \eqref{P3.4}, we abbreviate $\bb E Q_s\deq \bb E \langle \ul{G^{\delta_1}}\rangle\cdots \langle \ul{G^{\delta_s}}\rangle $ for $s \geq 2$ and $\delta_1,...,\delta_s \in \bb N_{+}$. By \eqref{preeqn} we have
\begin{equation} \label{P3.5}
\begin{aligned}
\bb E Q_{s}&=\frac{1}{-z_1-2\bb E \ul{G}} \bigg(\bb E Q_{s-1} \langle \ul{G^{\delta_s-1}} \rangle+
\sum_{a=0}^{\delta_s-1} \bb E Q_{s-1} \langle \ul{G^{\delta_s-a}} \rangle \langle \ul{G^{a+1}} \rangle-\sum_{a=0}^{\delta_s-1} \bb E Q_{s-1} \bb E \langle \ul{G^{\delta_s-a}} \rangle \langle \ul{G^{a+1}} \rangle\\
&\ \ +\frac{\delta_s}{N}\bb E Q_{s-1} \langle \ul{G^{\delta_s+1}} \rangle + 2\sum_{a=1}^{\delta_s-1}Q_{s-1}\langle \ul{G^{\delta_s-a}}\rangle \bb E \ul{G^{a+1}}+\frac{2}{N^2}\sum_{p=1}^{s-1}\delta_p\bb E Q_{s-1}/\langle\ul{G^{\delta_p}}\rangle\bb E \ul{G^{\delta_p+\delta_s+1}}\\
&\ \ +\frac{2}{N^2}\sum_{p=1}^{s-1}\delta_p\bb E Q_{s-1}/\langle\ul{G^{\delta_p}}\rangle\langle \ul{G^{\delta_p+\delta_s+1}}\rangle+\cal W_3\bigg)\,,
\end{aligned}
\end{equation}
where $\bb E Q_{s-1}\deq \bb E \langle \ul{G^{\delta_1}}\rangle\cdots \langle \ul{G^{\delta_{s-1}}}\rangle $, and $\cal W_3$ corresponds to the higher cumulant terms. Ignoring $\cal W_2$ and $W_3$, \eqref{P3.4} and \eqref{P3.5} now give close expansions, namely for any term we get from \eqref{P3.4} or \eqref{P3.5}, it can again be expanded using either of them. In terms of graphs, one can think of the process as a rooted tree, where $\bb E \ul{G^m}$ is the root, and the terms on right-hand side of \eqref{P3.4} are vertices in generation 1, and further expansions give the vertices in generation 2. We stop expanding a vertex when it has bound $\OO(1)$. By a careful combinatorial argument and Lemma \ref{prop4.4}, we show that the tree has finite depth. In this way we generate a locally finite tree with finite depth, which proves \eqref{Ppeiqi}. The general case with $\cal W_2$ and $\cal W_3$ will involve expansions which are slightly more complicated than \eqref{P3.4} and \eqref{P3.5}. The details of the tree structure are given at the end of Section \ref{sec:4.1} below.

In a second step, we use \eqref{Ppeiqi} in addition to \eqref{simplecumulant} to get a good bound for expectations of monomials of $\langle \ul{G}\rangle$, $\langle \ul{F^*}\rangle$, $G_{ij}$ and $F^*_{ij}$ (precise statements are given in Proposition \ref{prop4.3} below). Our result reveals the interaction between $\langle \ul{G}\rangle$ and $\langle \ul{F^*}\rangle$, which is much deeper than the fluctuation of $\langle \ul{G} \rangle$. In particular, it implies
\[
\bb E \langle \ul{G} \rangle \langle \ul{F^*}\rangle =\OO\Big(\frac{1}{N^2\omega^2}\Big)\,,
\]
which is sharp, and
\[
\frac{1}{-z_1-2\bb E\ul{G}}\cdot \bb E \langle\ul{G}\rangle^2 \langle \ul{F^*}\rangle=\OO\Big(\frac{1}{N^3\omega^3}\Big) \ll \frac{1}{N^2\omega^2}\,.
\] 
which solves the issue in \eqref{wwww}. The proof uses similar recursive expansion as in the first step whose details we omit.

In a third step, we prove Theorem \ref{thm_resolvent} by going back to \eqref{preeqn} and expanding each term on the right-hand side by hand to reveal all the explicit terms on right-hand side of \eqref{2.4i}. Equipped with the bounds proved in the second step, we show in Lemma \ref{cov} below that 
\[
\bb E \langle \ul{G} \rangle \langle \ul{F^*} \rangle=-\frac{2}{N^2(z_1-z_2^*)^2}+\OO\Big(\frac{1}{N^{2+\varepsilon}\omega^2}\Big)
\] 
for some small $\epsilon>0$. Afterwards, a small number of expansions (see Lemmas \ref{lem4.2} -- \ref{lem4.7} below) leads to other main terms in right-hand side of \eqref{2.4i}. For example, in Lemma \ref{lem44} below we show that the second term on right-hand side of \eqref{preeqn} can be expanded and leads to the term
$\frac{12}{N^4(z_1-z_2^{*})^4\kappa_E^2}$
in \eqref{2.4i}.  As an additional note, this latter term corresponds to the correction term $\frac{3}{2\pi^4(u-v)^4}$ in \eqref{2.6}, which does not appear in the complex Hermitian case \eqref{2.6}. The reason is simple from the above argument, since in the complex Hermitian case we have
\[
\bb E \langle \ul{G} \rangle \langle \ul{F^*} \rangle=\frac{1}{-z_1-2\bb E \ul{G}}\bigg(\frac{1}{N^2}\bb E \ul{GF^{*2}}+\bb E \langle \ul{G} \rangle^2 \langle \ul{F^*} \rangle+\cdots\bigg)\,,
\]
which does not contain the term $N^{-1}\bb E \langle\ul{G^2}\rangle \langle \ul{F^*}\rangle$ as in \eqref{preeqn}. More details are given in Section \ref{sec:4.4} below. 

The rest of this paper is organized as follows. In Section \ref{sec3} we introduce basic notations and tools, including estimates on the Green function from previous works; we also introduce the fundamental set of polynomials that we use throughout the proof, and record some of their basic properties. In Section \ref{sec4} we prove Theorem \ref{thm_resolvent} assuming a key estimate for a general class of polynomials of Green function, Proposition \ref{prop4.3}. In Section \ref{section5} we use Theorem \ref{thm_resolvent} and its analogue, Proposition \ref{prop6.1}, to prove our main result, Theorem \ref{mainthm}. In Section \ref{sec5} we prove Proposition \ref{prop4.3}. Finally, in Section \ref{sec7} we sketch the proof of Proposition \ref{prop6.1}. The appendices contain remarks on the connection of our results to several previous results mentioned in Section \ref{sec:outline_results} as well as a generalization on the distribution of the diagonal entries.

\section{Preliminaries}  \label{sec3}

Let $M$ be an $N \times N$ matrix. We use the notations $M_{ij}^n\equiv (M_{ij})^n$, $M^{*n}\equiv (M^{*})^n$, $M^{*}_{ij}\equiv (M^{*})_{ij} = \ol M_{ji}$, $\ul M \deq N^{-1} \tr M$. For $\im z \ne 0$, we define the Stieltjes transform $m$ of the semicircle density $\varrho$ by 
\begin{equation} \label{2.5}
m(z)\deq \int \frac{\varrho_x }{x-z}\,\dd x=\frac{1}{2\pi}\int_{-2}^{2}\frac{\sqrt{4-x^2}}{x-z}\dd x\,,
\end{equation}
so that
\[
m(z)=\frac{-z+\sqrt{z^2-4}}{2}\,.
\]
When taking the square root $\sqrt{z}$ of a complex number $z$, we always choose the branch cut at the negative real axis. This implies $\im \sqrt{z}>0$ whenever $\im z>0$, and $\im \sqrt{z}<0$ whenever $\im z<0$. We denote $\langle X \rangle \deq X-\bb E X$ for any random variable $X$ with finite expectation. Let $E \in [-2+\tau,2-\tau]$ be defined as in Theorems \ref{mainthm} and Theorem \ref{thm_resolvent}, and for the rest of the paper we abbreviate $\kappa\equiv \kappa_E\deq\sqrt{4-E^2}$.

If $h$ is a real-valued random variable with finite moments of all order, we denote by $\mathcal{C}_k(h)$ the $k$th cumulant of $h$, i.e.
\begin{equation} \label{cumulant_real}
\cal{C}_k(h) \deq  (-\mathrm{i})^k\cdot\big(\partial^k_\lambda \log \bb E e^{\mathrm{i}\lambda h} \big) \big{|}_{\lambda=0}\,.
\end{equation}
Accordingly, if $h$ is a complex-valued random variable with finite moments of all order, we denote by $\mathcal{C}_{p,q}(h)$ the $(p,q)$-cumulant of $h$, i.e.
\begin{equation} \label{cumulant_complex}
\cal{C}_{p,q}(h)\deq (-\ii)^{p+q} \cdot \left(\partial^{p}_s \partial^q_t \log \bb E e^{\mathrm{i}sh+\mathrm{i}t\bar{h}}\right) \bigg{|}_{s=t=0}\,.
\end{equation}

For a real symmetric Wigner matrix $H$, we shall use the generalized Stein Lemma \cite{Stein1981,Bar}, which can be viewed as a more precise and quantitative version of Stein's method. It was developed in the context of random matrix theory in \cite{KKP96, Kho1, Kho2, HK, HKR}. The proof of a slightly different version can be found in \cite{HKR}. 
\begin{lemma}[Cumulant expansion]\label{lem:cumulant_expansion}
Let $f\col\R\to\C$ be a smooth function, and denote by $f^{(k)}$ its $k$th derivative. Then, for every fixed $l \in\N$, we have 
\begin{equation}\label{eq:cumulant_expansion}
	\mathbb{E}\big[h\cdot f(h)\big]=\sum_{k=0}^{l}\frac{1}{k!}\mathcal{C}_{k+1}(h)\mathbb{E}[f^{(k)}(h)]+\cal R_{l+1},
\end{equation}	
provided that all expectations in \eqref{eq:cumulant_expansion} exist. Here $\cal R_{l+1}$ is a remainder term (depending on $f$ and $h$), such that for any $t>0$,
\begin{equation*} 
	\cal R_{l+1} = \OO(1) \cdot \bb E \big|h^{l+2}\cdot\mathbf{1}_{\{|h|>t\}}\big| \cdot \big\| f^{(l+1)}\big\|_{\infty} +\OO(1) \cdot \bb E |h|^{l+2} \cdot  \sup_{|x| \le t}\big|f^{(l+1)}(x)\big|\,.
\end{equation*}
\end{lemma}

We have a complex analogue of the above lemma, which will be used to deal with the complex Hermitian case.
\begin{lemma}[Complex cumulant expansion] \label{lem:5.1}
Let $f\col\bb C^2 \to \bb C$ be a smooth function, and we denote its holomorphic  derivatives by
\begin{equation*} 
f^{(p,q)}(z_1,z_2)\deq \partial^p_{z_1} \partial^q_{z_2} f(z_1,z_2)\,.
\end{equation*} 
Then for any fixed $l \in \bb N$, we have
\begin{equation} \label{5.16}
\bb E f(h,\bar{h})\bar{h}=\sum\limits_{p+q\leq l} \frac{1}{p!\,q!}\mathcal{C}_{p,q+1}(h)\bb E f^{(p,q)}(h,\bar{h}) + \cal R_{l+1}\,,
\end{equation}
assuming that all expectations in \eqref{5.16} exists, where $\cal R_{l+1}$ is a remainder term (depending on $f$ and $h$), such that for any $t>0$,
\begin{equation*}
\begin{aligned}
\cal R_{l+1}&=\OO(1)\cdot \bb E \big|h^{l+2}\cdot\mathbf{1}_{\{|h|>t\}}\big|\cdot \max\limits_{p+q=l+1}\big\| f^{(p,q)}(z,\bar{z})\big\|_{\infty} \\
&\msp +\OO(1) \cdot \bb E |h|^{l+2} \cdot \max\limits_{p+q=l+1}\big\| f^{(p,q)}(z,\bar{z})\cdot \mathbf{1}_{\{|z|\le t\}}\big\|_{\infty}\,.
\end{aligned}
\end{equation*}
\end{lemma}

The following result gives bounds on the cumulants of the entries of $H$.
\begin{lemma} \label{lem} 
\begin{enumerate}
	\item 
Let $H$ be a real Wigner matrix satisfying Definition \ref{def:wigner}. For every $i,j = 1, \dots, N$ and $k\ge 2$ we have
\begin{equation*}
\cal C_{k}(H_{ij})=\OO_k\big(N^{-k/2}\big)
\end{equation*}
and $\cal C_{1}(H_{ij})=0$.
\item 
Let $H$ be a complex Wigner matrix satisfying Definition \ref{def:wigner}. For every $i,j = 1, \dots, N$ and $p,q\in \N_+$ we have
\begin{equation*}
\cal C_{p,q}(H_{ij})=\OO_{p,q}\big(N^{-(p+q)/2}\big)
\end{equation*}
and $\cal C_{1,0}(H_{ij})=\cal C_{0,1}(H_{ij})=0$.
\end{enumerate}
\end{lemma}

We shall repeatedly us the following identity, whose proof is obvious.
\begin{lemma} [Resolvent identity] \label{resolvent}
For complex numbers $z_1$ and $z_2$ satisfying $z_1 \ne z_2$, $\im z_1\ne 0$, and $\im z_2 \ne 0$, we have
\begin{equation*} 
G(z_1)G(z_2)=\frac{G(z_1)-G(z_2)}{z_1-z_2}\,.
\end{equation*}
\end{lemma}

The following definition introduces a notion of a high-probability bound that is suited for our purposes. It was introduced (in a more general form) in \cite{EKYY4}.
\begin{definition}[Stochastic domination] \label{def:2.3} 
Let $$X=\pb{X^{(N)}(u)\col N \in \bb N, u \in U^{(N)}}\,,\qquad Y=\pb{Y^{(N)}(u)\col N \in \bb N, u \in U^{(N)}}$$ be two families of nonnegative random variables, where $U^{(N)}$ is a parameter set which may depend on $N$. We say that $X$ is stochastically dominated by $Y$, uniformly in $u$, if for all (small) $\varepsilon>0$ and (large) $D>0$ we have
\begin{equation*}
\sup\limits_{u \in U^{(N)}}	\bb P \left[ X^{(N)}(u) > N^{\varepsilon} Y^{(N)}(u) \right] \le N^{-D}
\end{equation*} 
for large enough $N \ge N_0(\varepsilon,D)$. If $X$ is stochastically dominated by $Y$, we use the notation $X \prec Y$. The stochastic domination will always be uniform in all parameters, such as $z$ and matrix indices, that are not explicitly stated to be constant.
\end{definition}

We now state the local semicircle law for Wigner matrices from \cite{EKYY4,EYY3}. For a recent survey of the local semicircle law, see \cite{BK16}, where the following version of the local semicircle law is stated.
\begin{proposition}[Local semicircle law] \label{refthm1}
Let $H$ be a Wigner matrix satisfying Definition \ref{def:wigner}, and define the spectral domain 
$$ 
{\bf S}  \deq  \{E+\mathrm{i}\eta\col |E| \le 10, 0 <  \eta \le 10 \}\,.
$$
Then we have the bounds
\begin{equation}  \label{afterall}
\max\limits_{i,j}|G_{ij}(z)-\delta_{ij}m(z)| \prec \sqrt{\frac{\im m(z)}{N\eta}}+
\frac{1}{N\eta}
\end{equation} 
and
\begin{equation} \label{3.4}
|\underline{G}(z)-m(z)| \prec \frac{1}{N\eta}\,,
\end{equation}
uniformly in $z =   
E+\mathrm{i}\eta \in {\bf S}$. 
\end{proposition}
The following lemma is a preliminary estimate on $G$. It provides a priori bounds on entries of powers of $G$ which are significantly better than those obtained by a direct application of the local semicircle law.
\begin{lemma}[Lemma 4.4,\cite{HK}] \label{prop4.4}
Let $H$ be a Wigner matrix satisfying Definition \ref{def:wigner}. Fix $\alpha \in [0,1)$ and $E \in (-2,2)$. Let $G\equiv G(z)=(H-z)^{-1}$, where $z\deq E+\mathrm{i}\eta$ and $\eta\deq N^{-\alpha}$. For any fixed $k \in \bb N_{+}$ we have
\begin{equation*} 
\big|\big\langle \ul{G^{k}} \big\rangle \big|\prec N^{(k-1)\alpha-(1-\alpha)}
\end{equation*}
as well as
\begin{equation*} 
\big|\big(G^k\big)_{ij}\big|\prec
\begin{cases}
N^{(k-1)\alpha} & \txt{if } i = j
\\
N^{(k-1)\alpha-(1-\alpha)/2} & \txt{if } i \neq j\,,
\end{cases}
\end{equation*}
uniformly in $i,j$.
\end{lemma}
We now introduce some additional notations that will be used frequently in the proof of Theorem \ref{thm_resolvent}. To motivate them, we note that the proof relies on a calculus of products of expectations of random variables of the type $G^{m}_{ij}$, $\ul{G^m}$, $\avg{\ul{G^m}}$ evaluated at $z_1$ and $z_2$, as well as their complex conjugates, e.g.
\begin{equation*}
a(z_1, z_2) \, N^{3/2} \, \E \ul{G^2}(z_1) \ul{G}(z_2^*)\,, \qquad
a_{i_1 i_2 i_3 i_4}(z_1, z_2)\, N^{-1/2} \, \bb E G_{i_1i_2}(z_1)G_{i_3i_4}(z_2^*)\E \avg{\ul{G^2}(z_2^*)}\,,
\end{equation*}
where $a(z_1, z_2)$ and $a_{i_1 i_2 i_3 i_4}(z_1, z_2)$ are uniformly bounded functions that may depend on $N$. It is convenient to classify such expressions depending on the exponent $t \in \R$ and on the number $n$ of indices $i_k$. Below, we introduce the notation $\cal P^{(n,t)}(\cal A)$ for the set of such expressions, where $\cal A$ is the set of matrices appearing in them, in the above examples $\cal A = \{G(z_1), G(z_2^*)\}$.

To that end, we define a set of formal monomials in a set of formal variables. Here the word \emph{formal} refers to the fact that these objects are purely symbolic and we do not assign any values to variables or monomials. The formal variables are constructed from a finite set of formal matrices $\cal A$ and the infinite set of formal indices $\{i_1, i_2, \dots\}$.
\begin{itemize}
\item
Set $\cal T(\cal A)\deq\{\ul{A^m}\col A \in \cal A, m \ge 1\}$.
\item
For $n\in \N$ denote by $\cal M^{(n)}(\cal A)$ the set of monomials with coefficient $1$ in the variables $A^m_{ x  y}$ and $\langle \ul{A^m}\rangle$, where $A \in \cal A$, $m \in \bb N_+$, and $ x,  y\in \{i_1,\ldots,i_n \}$.
\item
Let $\cal P^{(n,t)}(\cal A)$ be the set of monomials with coefficient $N^{t}$ in the variables $\bb E X$, where $t \in \bb R$, $X \in \cal T(\cal A) \cup \cal M^{(n)}(\cal A)$.
\item
Set $\cal P^{(n)}(\cal A)\deq \bigcup_{t}\cal P^{(n,t)}(\cal A)$ and $\cal P(\cal A)\deq \bigcup_{n}\cal P^{(n)}(\cal A)$.
\end{itemize}
Note that $ \tilde{\cal A} \subset \cal A$ implies $\cal P^{(n,t)}( \tilde{\cal A}) \subset \cal P^{(n,t)}(\cal A)$ for all $t$ and $n$.
Next, we define the following maps $\nu_1,\ldots,\nu_6\col \cal P(\cal A)\to \N$.
\begin{enumerate}
\item $\nu_{1} (P)$ = sum of $m-1$ of all $\ul {A^m}$ in $P$ with $A \in \cal A$.
\item $\nu_2(P)$ = sum of $m-1$ of all $\langle \ul {A^m} \rangle$ and $(A^m)_{ x  y}$ in $P$ with $A \in \cal A$.
\item $\nu_3(P)$ = number of $\ul{A^m}$ in $P$ with $m \ge 2$ and $A \in \cal A$.
\item $\nu_4(P)$ = number of $\langle \ul{A^m} \rangle$ in $P$ with $A \in \cal A$.
\item $\nu_5(P)$ = number of $k \in \N_+$ such that the index $i_k$ appears an odd number of times in $P$.
\item $\nu_6(P)$ = number of $k \in \N_+$ such that the index $i_k$ appears an even number of times in $P$ and $i_k$ appears in at least one $(A^m)_{ x  y}$ with $x \ne y$ and $A \in \cal A$.
\end{enumerate}

Next, we assign to each monomial $P\in \cal P^{(n,t)}(\cal A)$ a \emph{value} $P_{i_1 \dots i_n} \in \C$ as follows. Suppose that the set $\cal A$ consists of $N \times N$ random matrices. Then for any $n$-tuple $(i_1, \dots, i_n) \in \{1, \dots, N\}^n$ we define the number $P_{i_1 \dots i_n}$ as the one obtained by taking the formal expression $P$ and evaluating it with the laws of the matrices in $\cal A$ and the numerical values of $i_1, \dots, i_n$.

In the following arguments, the set $\cal A$ will consist of Green functions of $H$ for various values of the spectral parameter $z$, and the indies $i_1, \dots, i_n$ will be summed over.

We give an example to illustrate the above definitions.
\begin{example}
	Let $\cal A=\{A,B\}$ and set
	\begin{equation*}
	P\deq N^{t}\bb E \ul{A^3} \,\bb E\ul{B^{*4}} \,\bb E A^2_{i_1i_2} (B^2)_{i_3i_3}  (A^2)_{i_2i_4} \, \bb E A_{i_4i_1} B_{i_5i_6}\langle \ul{A^7} \rangle\,.
	\end{equation*}
	Clearly, $P \in \cal P^{(6,t)}(\cal A) \subset \cal P(\cal A)$. Then $\nu_1(P)=(3-1)+(4-1)=5$, $\nu_2(P)=(2-1)+(2-1)+(2-1)+(1-1)+(1-1)+(7-1)=9$, $\nu_3(P)=2$, and $\nu_4(P)=1$. We also see that $\nu_5(P)=2$, and the two corresponding indices are $i_5$ and $i_6$; $\nu_6(P)=3$, and the corresponding indices are $i_1, i_2, i_4$. 
	
	Furthermore, let $G$ be as in Lemma \ref{prop4.4} and we explicitly set $A=G$, $B=G^*$. We have, by Lemma \ref{prop4.4},
	\[
	\sum_{i_1,...,i_6} P_{i_1,...,i_6} \prec  N^4\cdot N^{t} \cdot N^{8\alpha}\cdot N^{7\alpha-1} \cdot \sum_{i_5,i_6}\bb E|G^{*}_{i_5i_6}| \prec N^{t+9/5+31\alpha/2}=N^{t+6+\alpha(\nu_1+\nu_2)-(1-\alpha)\nu_4-(1-\alpha)\nu_5/4}\,.
	\]
\end{example}

In the light of the above example, we have the following priori bound for general $P$. It is a relatively straightforward consequence of Lemma \ref{prop4.4}.
\begin{lemma} \label{lem:4.2}
Let us adopt the conditions of Theorem \ref{thm_resolvent}. Let $\cal A=\{G(z_1),G^{*}(z_1), G(z_2),G^{*}(z_2)\}$, and fix $P \in P^{(n,t)}(\cal A)$ for some $n \in \bb N$, $t \in \bb R$. Let $(a_{i_1\ldots i_n})_{1\le i_1,\ldots,i_n\le N}$ be a family of complex numbers that is uniformly bounded in $i_1,\ldots,i_n$. Set $\alpha \deq -\log_N \eta$. Then
\begin{equation} \label{3.14}
\sum_{i_1,\ldots, i_n}a_{i_1\ldots i_n}P_{i_1\ldots i_n} \prec N^{t+n+b_0(P)}\,,
\end{equation}
where $b_0(P) \deq \alpha(\nu_1(P)+\nu_2(P))-(1-\alpha)\nu_4(P)-(1-\alpha)\nu_5(P)/4$. Moreover, for $\nu_1(P)=\nu_2(P)=0$, we have
\begin{equation} \label{nodomination}
\sum_{i_1,\ldots, i_n}a_{i_1\ldots i_n}P_{i_1\ldots i_n}=\OO(N^{t+n})\,.
\end{equation}

\end{lemma}  
\begin{proof}
Let us first prove \eqref{3.14}. After renaming the indices $i_1, \dots, i_n$, we may assume that for each $k=1,2,...,\nu_5$ the total number of appearances of $i_k$ in all factors $(A^m)_{ x  y}$ of $P$ is odd, where $A \in \cal A$. It is easy to see that $\nu_5$ is even. We further assume that $i_1,\cdots,i_{\nu_5/2}$ appear in different factors $(A^m)_{xy}$, $x\not\equiv y$ in $P$. Then by the inclusion-exclusion principle, we split the summation in \eqref{3.14} into the ones that for any $a, b \in \{1,...,{\nu_5/2}\}$, $i_a$ and $i_b$ are either distinct or identical. For example, if $P=A_{jj}A_{kk}A_{ll}A_{mm}A_{jk}A_{lm}$ we see that $\nu_5(P)=4$ and $\nu_5(P)/2=2$. We set $i_1=j$ and $i_2=l$, so that the summation can be rearranged as
\[
\sum A_{jj}A_{kk}A_{ll}A_{mm}A_{jk}A_{lm}=\sum_{k,m}A_{kk}A_{mm}\Big(\sum_{j\ne l} A_{jj}A_{jk}A_{ll}A_{lm}+\sum_{j}A^2_{jj}A_{jk}A_{jm} \Big)\,. 
\]
The desired result then follows immediately from Lemma \ref{prop4.4}.

The proof of \eqref{nodomination} follows from \eqref{afterall} and the fact $|m(z)|<1$ for all $z$.
\end{proof}

We end this section by defining a set of functions that will be used in the proof. For $z_1, z_2 \in \C \setminus [-2,2]$ define
\begin{equation} \label{488}
f_1(z_1,z_2)\deq-\frac{2}{\sqrt{z_1^2-4}}+\frac{2}{\sqrt{z_2-4}}\,,
\quad 
f_{2\pm}(z_1,z_2)\deq\frac{4+z_1z_2\pm\sqrt{z_1^2-4}\sqrt{z_2^{2}-4}}{\sqrt{z_1^2-4}\sqrt{z_2^{2}-4}(\sqrt{z_1^2-4}\mp\sqrt{z_2^{2}-4})^2}\,,
\end{equation} 
\begin{equation} \label{f34}
f_3(z_1,z_2)\deq\frac{2m(z_1)^2m(z_2)^2}{\sqrt{z_1^2-4}\sqrt{z_2^{2}-4}}\,,
\quad
f_4(z_1,z_2)\deq-\frac{m(z_1)m(z_2)(m(z_2)+m(z_1))}{\sqrt{z_1^2-4}\sqrt{z_2^{2}-4}}\,.
\end{equation}
In each argument $z_1$ and $z_2$, all of these functions have a branch cut on the line $[-2,2]$, and they can be extended to $\R \setminus \{-2,2\}$ either from above or from below. We use the notation $f(x ^{\pm}) \deq \lim_{\epsilon \downarrow 0} f(x \pm \ii \epsilon)$ for these extensions. For $E \in (-2,2)$ we define
\begin{equation}\label{491}
V(E) \deq \frac{2E(E^2-2)}{\kappa}+\frac{2\im m(E^+)^4}{\kappa^2}\,,
\end{equation}
where $\kappa\equiv \kappa_E \deq\sqrt{4-E^2}$. Note that all the functions defined above are bounded for bounded arguments separated away from $\pm 2$.

Furthermore, for $x_1,x_2 \in (-2,2)$, let 
\begin{multline} \label{g_1}
g_1(x_1,x_2) \deq 2\re \pb{ f_{2+}(x_1^{+},x_2^{-})- f_{2 - }(x_1^{+},x_2^{+})}\\=-\frac{4\big(4+x_1x_2+\sqrt{(4-x_1^2)(4-x_2^2)}\big)}{\sqrt{(4-x_1^2)(4-x_2^2)}+\big(\sqrt{4-x_1^2}+\sqrt{4-x_2^2}\big)}=\frac{4}{(x_1-x_2)^2}\bigg[\frac{4-x_1x_2}{\sqrt{4-x_1^2}\sqrt{4-x_2^2}}-1\bigg]\,,
\end{multline}
\begin{equation} \label{g_2}			
g_2(x_1,x_2)\deq 2\re \pb{ f_3(x_1^{+},x_2^{-})-f_3(x_1^{+},x_2^{+})}=\frac{2(x_1^2-2)(x_2^2-2)}{\sqrt{(4-x_1^2)(4-x_2^2)}\,}\,,
\end{equation}
and
\begin{equation} \label{g_3}
g_3(x_1,x_2)\deq 2 \re \pb{f_4(x_1^+,x_2^-)-f_4(x_1^+,x_2^+)}=\frac{x_1^2x_2+x_1x_2^2-2x_1-2x_2}{\sqrt{(4-x_1^2)(4-x_2^2)}\,}\,.
\end{equation}
For $i=1,\dots,3$ we define
\begin{equation} \label{FFF}
F_i(u,v)=g_i\bigg(E+\frac{u}{N\varrho_E},E+\frac{v}{N\varrho_E}\bigg).
\end{equation}

\section{Correlation of Green functions: proof of Theorem \ref{thm_resolvent}} \label{sec4}
In this section we prove Theorem \ref{thm_resolvent}. For the rest of the paper we set $\alpha\deq -\log_N \eta$ and $\gamma \deq -\log_N \omega$, so that $\eta=N^{-\alpha}$, $\omega=N^{-\gamma}$ and $0 \le \gamma \le \alpha \le 1-\tau$; we also abbreviate $G\deq G(z_1)$ and $F\deq G(z_2)$. An important ingredient for the proof is to use Lemmas \ref{lem:cumulant_expansion}  and \ref{lem:5.1} on the left-hand side of \eqref{3.14} and gradually improve the estimate we have in Lemma \ref{lem:4.2}. The finalized estimate is summarized in the following proposition, whose proof is postponed to Section \ref{sec5}.
\begin{proposition}  \label{prop4.3}
Fix $n \in \N$.
Let $(a_{i_1\ldots i_n})_{i_1,\ldots,i_n}$ be a family of uniformly bounded complex numbers. 
\begin{enumerate}
	\item Suppose $\cal A=\{G\},\{F\},\{G^*\}$, or $\{F^*\}$, and fix $P \in \cal P^{(n,t)}(\cal A)$ for some $t \in \bb R$. We have
	\begin{equation} \label{4.39}
	\sum_{i_1,\ldots,i_n}a_{i_1\ldots i_n}P_{i_1\ldots i_n} = \OO(N^{t+n+b(P)})\,,
	\end{equation}
	where $b(P) \deq -\nu_4(P)-(\nu_5(P)+\nu_6(P))/2$.
	\item Suppose $\cal A=\{G, F^*\}$ or $\{G^*, F\}$, and fix $P \in \cal P^{(n,t)}(\cal A)$ for some $t \in \bb R$. We have
	\begin{equation*} 
	\sum_{i_1,\ldots,i_n}a_{i_1\ldots i_n}P_{i_1\ldots i_n} = \OO(N^{t+n+b_*(P)})\,,
	\end{equation*}
	where $b_*(P) \deq \gamma v_2(P)-(1-\gamma)\nu_4(P)-\nu_5(P)/2-(1-\gamma)\nu_6(P)/2$.
\end{enumerate}
\end{proposition}
\begin{remark}
We make a comparison of our result Proposition \ref{prop4.3} and the rough bound from Lemma \ref{lem:4.2}. By our assumptions $\nu_1,...,\nu_6\ge 0$ and $0\leq\gamma\leq\alpha<1$, we have $b_0(P) \ge b(P)$ and $b_0(P) \ge b_*(P)$, thus Proposition \ref{prop4.3} improves Lemma \ref{lem:4.2}, and it also gives the sharp bound in many situations. To illustrate, we give two examples for real symmetric $H$.

Let us take $\cal A=\{G, F^*\}$ and $P=\bb E \langle \ul{G} \rangle 	\langle \ul{F^*} \rangle \in \cal P^{(0,0)}(\cal A)$. We see $b_*(P)=2\gamma-2$, and by Proposition \ref{prop4.3}(ii) we have $P=\OO(N^{t+n+b(P)})=\OO(N^{2\gamma-2})$. We also see that $b_0(P)=2\alpha-2$, and by Lemma \ref{lem:4.2} we have $P=\OO(N^{t+n+b_0(P)})=\OO(N^{2\alpha-2}) \geq \OO(N^{2\gamma-2})$. From Theorem \ref{thm_resolvent}(i) we see that
$$
P=-2N^{2\gamma-2}+\OO(N^{2\gamma-2-c})
$$
for some $c=c(\gamma)\geq0$, thus the bound from Proposition \ref{prop4.3} is not only better than that from Lemma \ref{lem:4.2} but also sharp. 

In the second example we take $\cal A=\{G\}$ and $P=\bb E \ul{G^2} \in \cal P^{(0,0)}(\cal A)$. In this case $b(P)=1$ and Proposition \ref{prop4.3}(i) gives $P=\OO(1)$. Also, $b_0(P)=\alpha$ and Lemma \ref{lem:4.2} gives $P=\OO(N^{\alpha}) \ge \OO(1)$. By Lemma \ref{lem4.3}(i) below we see that the bound $\OO(1)$ is sharp.
\end{remark}

The next Lemma is designed to estimate the higher cumulant terms by the lower ones in a cumulant expansion.
\begin{lemma} \label{useful}
Recall that $H$ is defined as in Definition \ref{def:wigner}. Let $\cal A \in \{G,G^*,F,F^*\}$, and $P \in \cal P^{(n,t)}(\cal A)$ for some $n \ge 2$ and $t \in \bb R$. Let $X \in  \cal M^{(l)}(\cal A)$ for some $l\in \{2,3,...,n\}$, such that $\bb E X$ is a factor of $P$. Let $H_{ij}$ be an entry of $H$, $i,j \in \{i_1,\ldots,i_l\}$, and $k \ge 1$, we write
\begin{equation} \label{k}
\sum_{i_1,\ldots,i_n}a_{i_1\ldots i_n}\frac{P_{i_1\ldots i_n}}{\bb E X_{i_1\dots i_l}} \cdot \bb E \partial^k_{ij} X_{i_1\dots i_l}\,,
\end{equation}
as a finite sum of terms in the form
\begin{equation*}
\sum_{i_1,\ldots,i_n}a'_{i_1\ldots i_n}P'_{i_1\ldots i_n}.
\end{equation*}
Then each $P'\in\cal P^{(n,t')}(\cal A)$ in the sum satisfies
\begin{equation} \label{5.2}
\nu_1(P')=\nu_1(P)\,, \ \ \nu_3(P')=\nu_3(P)\,,\ \ \nu_5(P') \ge \nu_5(P)-d(P')\,,\ \ \nu_6(P') \ge \nu_6(P)-(2-d(P'))
\end{equation}
for some $d(P') \in \{0,2\}$, and
\begin{equation} \label{5.3}
t'+\lambda\nu_2(P')-(1-\lambda)\nu_4(P')=t+\lambda\nu_2(P)-(1-\lambda)\nu_4(P)
\end{equation}
for any $\lambda \in \bb R$.
Moreover, if $i, j$ both appear even times in $P$, then each $P' \in \cal P^{(n,t')}(\cal A)$ satisfies
\begin{equation} \label{55}
\nu_5(P') \ge \nu_5(P)\,,\ \ \nu_6(P') \ge \nu_6(P)-2\,.
\end{equation}
\end{lemma} 
\begin{proof} 
By the definition of $\cal M^{(n)}(\cal A)$ and $l \ge 2$, we see that $\bb E X$ cannot be in the form $\bb E \ul{A^m}$, and we have $\nu_1(P)=\nu_1(P')$, $\nu_3(P)=\nu_3(P')$. For the same reason, we only need to consider whether it is $(A^m)_{x y}$ or $\langle \ul{A^m} \rangle$ that is differentiated by $H_{ij}$ in \eqref{k}, where $A \in \cal A$. By the differential rules \eqref{3.15} and \eqref{4.9} below, we see that in both the real symmetric and complex Hermitian cases the last two relations in \eqref{5.2} are satisfied. Similarly, we can also verify that
\begin{equation*}
t+\nu_2(P)=t'+\nu_2(P') \ \mbox{  and  }\ t-\nu_4(P)=t'-\nu_4(P')\,,
\end{equation*}
which implies \eqref{5.3}. Moreover, we see that \eqref{55} is also a simple consequence of \eqref{3.15} and \eqref{4.9} below.
\end{proof}
We present one more result, which carefully analyses the values of $\bb E\ul{G}$ and $\bb E \ul{G^2}$. It is a corollary of Proposition \ref{prop4.3} and Lemma \ref{useful}. We shall only give the proof of Lemma \ref{lem4.3}(i) at the end of Section \ref{sec4.2} below, and Lemma \ref{lem4.3}(ii) follows in a similar fashion.
\begin{lemma} \label{lem4.3}
Under the assumptions of Theorem \ref{thm_resolvent}, we have the following results.
\begin{enumerate}
	\item For real symmetric Wigner matrix $H$, we have
	\begin{equation*} 
	\bb E \ul{G}=m(z_1)-\frac{1}{N\sqrt{z_1^2-4}}\Big(-\frac{1}{2}+\frac{z_1}{2\sqrt{z_1^2-4}}+m(z_1)^4\sum_{i,j}\cal C_4(H_{ij})\Big)+\OO(N^{-3/2})\,,
	\end{equation*}
	and
	\begin{equation*}
	\bb E \ul{G^{2}}=-\frac{1}{2}+\frac{z_1}{2\sqrt{z_1^{2}-4}}+\OO(N^{-1})\,.
	\end{equation*}
	\item For complex Hermitian matrix $H$, we have
	\begin{equation*}
	\bb E \ul{G}=m(z_1)-\frac{m(z_1)^4}{N\sqrt{z_1^2-4}}\sum_{i,j}\cal C_{2,2}(H_{ij})+\OO(N^{-3/2})\,,
	\end{equation*}
	and
		\begin{equation*}
		\bb E \ul{G^{2}}=-\frac{1}{2}+\frac{z_1}{2\sqrt{z_1^{2}-4}}+\OO(N^{-1})\,.
		\end{equation*}
\end{enumerate}
\end{lemma}

In Sections \ref{sec4.1} -- \ref{sec:4.3} we prove Theorem \ref{thm_resolvent} for the real symmetric case ($\beta = 1$). We remark on the complex Hermitian case ($\beta = 2$) in Section \ref{sec:4.4}. Thus, throughout Sections \ref{sec4.1} -- \ref{sec:4.3} we assume that $\beta = 1$.

\subsection{Initial expansion and the leading term} \label{sec4.1} 
In this section we start with $\bb E \langle \ul{G} \rangle \langle \ul{F^*} \rangle$ and perform the initial cumulant expansion. This will reveal the leading term of our calculation.

Let $\cal A =\{G,F^*\}$, and we see that $\bb E \langle \ul{G} \rangle \langle \ul{F^*} \rangle \in \cal P^{(0,0)}(\cal A)$. Trivially we have $z_1G=GH-I$, thus
\begin{equation} \label{4.4}
z_1\bb E \langle \ul{G} \rangle \langle \ul{F^*} \rangle=\bb E \ul{GH}\langle \ul{F^*} \rangle=\frac{1}{N} \sum_{i,j}\bb E G_{ij}H_{ji}\langle \ul{F^*} \rangle\,.
\end{equation}
Since $H$ is symmetric, for any differentiable function $f=f(H)$ we define the directional derivative $\partial_{ij}f(H)\deq \frac{\mathrm{d}}{\mathrm{d}t}\Big{|}_{t=0} f\pb{H+t\,\Delta^{ij}},$
where $\Delta^{ij}$ denotes the matrix whose entries are zero everywhere except at the sites $(i,j)$ and $(j,i)$ where they are one: $\Delta^{ij}_{kl} =(\delta_{ik}\delta_{jl}+\delta_{jk}\delta_{il})(1+\delta_{ij})^{-1}$. 
We then compute the right-hand side of \eqref{4.4} using the formula \eqref{eq:cumulant_expansion} with $f=f_{ij}(H)\deq G_{ij}\langle \ul{ F^*} \rangle$, $h=H_{ji}$, and obtain
\begin{equation} \label{4.7}
z_1\bb E \langle \ul{G} \rangle \langle \ul{F^*} \rangle=\frac{1}{N^2}\sum_{i,j}(1+\delta_{ij}\bb E \partial_{ji} (G_{ij}\langle \ul{F^*} \rangle)+\sum_{k=2}^{l}W^{(1)}_{k}+ \frac{1}{N}\sum_{i,j}R^{(1,ji)}_{l+1}\,,
\end{equation}
where we used the notation
\begin{equation} \label{Wk}
W^{(1)}_k\deq \frac{1}{N} \sum_{i,j}\frac{1}{k!}\cal C_{k+1}(H_{ji}) \bb E \partial^k_{ji} (G_{ij}\langle\ul{F^*}\rangle)\,.
\end{equation}
Note that the sum in \eqref{4.7} begins from $k = 1$ because $\cal C_1(H_{ij}) = 0$. 

In \eqref{4.7}, $l$ is a fixed positive integer to be chosen later, and $R^{(1,ij)}_{l+1}$ is a remainder term defined analogously to $R_{l+1}$ in \eqref{eq:cumulant_expansion}. It is a routine verification (for details one can refer to Lemma 4.6(i) in \cite{HK}) to show that we can find a constant $l \in \bb N$ such that $(NT)^{-1}\sum_{i,j}R^{(1,ij)}_{l+1}=\OO(N^{-10})$. From now on, we shall always assume that the reminder term is negligible for large enough $l$. 

Note that for $G=(H-z_1)^{-1}$, we have
\begin{equation} \label{3.15}
\partial_{kl} G_{ij}=-(1+\delta_{kl})^{-1}(G_{ik}G_{lj}+G_{ il}G_{kj})\,,
\end{equation} 
which gives
\begin{multline*}
\frac{1}{N^2}\sum_{i,j} (1+\delta_{ij})\bb E \partial_{ji} (G_{ij}\langle\ul{F^*}\rangle)
=-N^{-2}\sum_{i,j}\bb E(G_{ii}G_{jj}+G_{ij}^2)\langle\ul{F^*}\rangle-N^{-3}\sum_{i,j}\bb EG_{ij}((F^{*2})_{ij}+(F^{*2})_{ji})\\
=-\bb E \ul{G}^2\langle \ul{F^*}\rangle -N^{-1}\bb E \ul{G^2} \langle \ul{F^*}\rangle-2N^{-2}\bb E \ul{GF^{*2}}
\\
=-2 \bb E \ul{G} \,\bb E \langle \ul{G} \rangle \langle \ul{F^*} \rangle -\bb E \langle \ul{G} \rangle^2 \langle \ul{F^*} \rangle-N^{-1}\bb E \langle\ul{G^2}\rangle \langle \ul{F^*}\rangle-2N^{-2}\bb E \ul{GF^{*2}}\,.
\end{multline*}
Together with \eqref{4.7} we have
\begin{equation} \label{4.11}
\bb E \langle \ul{G} \rangle \langle \ul{F^*} \rangle=\frac{1}{T}\bigg(\bb E \langle \ul{G} \rangle^2 \langle \ul{F^*} \rangle+\frac{1}{N}\bb E \langle\ul{G^2}\rangle \langle \ul{F^*}\rangle+\frac{2}{N^2}\bb E \ul{GF^{*2}}-\sum_{k=2}^l W^{(1)}_k+\OO(N^{-10})\bigg)\,,
\end{equation}
where $T\deq -z_1-2\bb E \ul{G}$. From \eqref{2.5} and \eqref{3.4} it is easy to see that 
\begin{equation*}  
\Big|\frac{1}{T}\Big|=\OO_{\tau}(1)\,,
\end{equation*}
and we recall $\tau>0$ is the constant defined in Theorem \ref{thm_resolvent}. By Lemma \ref{lem} and \eqref{3.15} we see that for $k \ge 2$, $T^{-1}W^{(1)}_k$ is a finite sum of the terms in the form $\sum_{i,j}a_{ij}P_{ij}$, where $P \in \cal P^{(2,-(k+5)/2)}(\cal \{G,F^*\})$ or $P \in \cal P^{(2,-(k+3)/2)}(\cal \{G,F^*\})$, depending on whether $\langle \ul{F^*} \rangle$ is differentiated or not.

 Consider $W^{(1)}_4$, one can readily check that by \eqref{3.15} we have $t(P)+b_*(P) \le \gamma-5.5$ for each $P$ appearing as a term $\sum_{i,j}a_{ij}P_{ij}$ in $T^{-1}W^{(1)}_4$. Thus by Proposition \ref{prop4.3}(ii) we know
\begin{equation*}
T^{-1}W^{(1)}_4=\OO(N^{\gamma-3.5})\,.
\end{equation*}
For $k \ge 5$, let $\sum_{i,j} a'_{ij} P'_{ij}$ be a term in $T^{-1}W^{(1)}_k$. By \eqref{Wk} we have
\[
W^{(1)}_k=\frac{1}{N}\sum_{i,j}\frac{4!\,\cal C_{k+1}(H_{ij})}{k!\,\cal C_{5}(H_{ij})} \frac{1}{4!}\cal C_5(H_{ij}) \bb E \partial^{k-4}_{ji}(\partial^4_{ji} (G_{ij}\langle\ul{F^*}\rangle))\,,
\]
and hence there is a term  $\sum_{ij}a_{ij}P_{ij}$ in $T^{-1}W^{(1)}_4$ for some $P_{ij}$ such that $\sum_{i,j} a'_{ij} P'_{ij}$ appears in the sum
\begin{equation*}
\sum_{i,j}\frac{4!\,\cal C_{k+1}(H_{ij})}{k!\,\cal C_{5}(H_{ij})}\, a_{ij}\,\bb E \partial^{k-4}_{ij}P_{ij}\,.
\end{equation*}
Now we apply Lemma \ref{useful}. By \eqref{5.2} and choosing $\lambda=\gamma$ in \eqref{5.3} we know that $b_*(P')\le b_*(P)+1$. Also, by Lemma \ref{lem}(i) we have $t'=t-(k+1)/2+5/2\leq t-1/2$. Thus $t'+b_*(P')\le t-1/2+b_*(P)+1 \le \gamma-5$. By Proposition \ref{prop4.3}(ii) we know
\begin{equation*}
T^{-1}W^{(1)}_k=\OO(N^{\gamma-3})
\end{equation*}
for $k \ge 5$. Together with \eqref{4.11} we see that
\begin{equation} \label{417}
\bb E \langle \ul{G} \rangle \langle \ul{F^*} \rangle=\frac{1}{T}\bigg(\bb E \langle \ul{G} \rangle^2 \langle \ul{F^*} \rangle+\frac{1}{N}\bb E \langle\ul{G^2}\rangle \langle \ul{F^*}\rangle+\frac{2}{N^2}\bb E \ul{GF^{*2}}- W^{(1)}_2-W^{(1)}_3\bigg)+\OO(N^{\gamma-3})\,.
\end{equation}
\begin{remark}
In the rest of the paper, we will repeatedly use the above argument. More precisely, after performing a cumulant expansion, we estimate the $n$th cumulant term ($n=3$ or $5$, depending on how small we need the bound to be) by Proposition \ref{prop4.3}, and then compare the higher cumulant terms to the $n$th cumulant term using Lemma \ref{useful}. This allows us to have bounds for any high order cumulant term. In addition, we repeat the proof of Lemma 4.6(i) in \cite{HK} to estimate the remainder term. In this way we get a cumulant expansion with only the first few cumulant terms together with a ``good'' error term.

We will not repeat this strategy in detail for the rest of the paper, and the reader can check it also works for all other cumulant expansions appearing in this paper.
\end{remark}
We close this section with the following simple version of Theorem \ref{thm_resolvent} for $\beta = 1$, which reveals the leading order term.
\begin{lemma} \label{cov}
Under the assumptions of Theorem \ref{thm_resolvent}, for $\beta = 1$ we have
\begin{equation*}
\bb E \langle \ul{G} \rangle \langle \ul{F^*} \rangle=-\frac{2}{N^2(z_1-z_2^*)^2}+\OO(N^{\gamma-2}+N^{3\gamma-3})\,.
\end{equation*}
\end{lemma}
\begin{proof}
The proof is a simple computation using \eqref{417},  Lemmas \ref{resolvent} and \ref{lem4.3} and Proposition \ref{prop4.3}. We omit the details.
\end{proof}
In the next few sections we shall analyse the right-hand side of \eqref{417} carefully.
\subsection{Second cumulant} \label{sec:4.2} In this section we analyse the second-cumulant terms on the right-hand side of \eqref{417}.

\begin{lemma} \label{lem4.2}
Let $\cal E_1\equiv \cal E_1(\omega)$ be as in \eqref{cal E_1}.	 For the first term on the right-hand side of \eqref{417}, we have 
\begin{equation} \label{GF_est}
\frac{1}{T}\bb E \langle \ul{G} \rangle^2 \langle \ul{F^*} \rangle=\cal E_1\,.
\end{equation}
\end{lemma}
\begin{proof}
By Proposition \ref{prop4.3}(ii) and $|T^{-1}|=\OO(1)$ we have $T^{-1}\bb E \langle \ul{G} \rangle^2 \langle \ul{F^*} \rangle=\OO(N^{3\gamma-3})$. In order to get the desired bound, we perform a cumulant expansion whenever our estimate is not precise enough. We divide the proof into four steps.

\emph{Step 1.} In this step we perform the cumulant expansion for $\bb E \langle \ul{G} \rangle^2 \langle \ul{F^*} \rangle$. By the identity $z_2^*F^*=F^*H-I$ we have
\begin{equation} \label{5.21}
z_2^*\bb E \langle \ul{G} \rangle^2 \langle \ul{F^*} \rangle=\bb E \langle \langle \ul{G} \rangle^2\rangle  \ul{F^*H}=\frac{1}{N} \sum_{i,j}\bb E\langle \langle \ul{G} \rangle^2\rangle F^*_{ij}H_{ji}\,.
\end{equation}
Similar to \eqref{4.11}, by calculating the last sum in \eqref{5.21} using the formula \eqref{eq:cumulant_expansion} with $f=f_{ij}(H)\deq \langle \langle \ul{G} \rangle^2\rangle F^*_{ij}$ and $h=H_{ji}$ we have for some fixed $l \in \bb N$,
\begin{multline*} 
\bb E \langle \ul{G} \rangle^2 \langle \ul{F^*} \rangle
\\
=\frac{1}{T_*}\bigg(\bb E \langle\ul{G}\rangle^2\langle \ul{F^*}\rangle^2-\bb E \langle\ul{G}\rangle^2\bb E\langle \ul{F^*}\rangle^2+\frac{4}{N^2}\bb E\langle\ul{G}\rangle \langle \ul{G^2F^*}\rangle
+\frac{1}{N}\bb E \langle \ul{G}  \rangle^2\langle \ul{F^{*2}}\rangle-\sum_{k=2}^l W^{(2)}_k\bigg)+O(N^{-10})\,,
\end{multline*}
where we used the notations $T_* \deq -z_2^*-2 \bb E \ul{F^*}$,
\begin{equation*} 
W^{(2)}_k\deq \frac{1}{N} \sum_{i,j}\frac{1}{k!}\cal C_{k+1}(H_{ji}) \bb E \frac{\partial^k ((\langle\ul{G}\rangle^2 -\bb E \langle\ul{G}\rangle^2) (F^*)_{ij})}{\partial H^k_{ji}}\,.
\end{equation*}
From \eqref{2.5} and \eqref{3.4} we see that $|\frac{1}{T_*}|=\OO_{\tau}(1)$,
and we recall $\tau>0$ is the constant defined in Theorem \ref{thm_resolvent}. As in Section \ref{sec4.1}, we use Proposition \ref{prop4.3} and Lemma \ref{useful} to deduce that
\begin{equation*}
	T_*^{-1}W_k^{(2)}=\OO(N^{2\gamma-3.5})
\end{equation*}
for $k=2$ and $k\ge 4$. Thus
\begin{multline} \label{4.233}
\bb E \langle \ul{G} \rangle^2 \langle \ul{F^*} \rangle=\frac{1}{T_*}\bigg(\bb E \langle\ul{G}\rangle^2\langle \ul{F^*}\rangle^2-\bb E \langle\ul{G}\rangle^2\bb E\langle \ul{F^*}\rangle^2+\frac{4}{N^2}\bb E\langle\ul{G}\rangle \langle \ul{G^2F^*}\rangle+\frac{1}{N}\bb E \langle \ul{G}  \rangle^2\langle \ul{F^{*2}}\rangle- W^{(2)}_3\bigg)+\cal E_1\,,
\end{multline}
where $\cal E_1=\cal E_1(\omega)$ was defined in \eqref{cal E_1}. By Proposition \ref{prop4.3}(i) we conclude that 
\begin{equation} \label{425}
\bb E \langle\ul{G}\rangle^2\bb E\langle \ul{F^*}\rangle^2=\OO(N^{-4})\,.
\end{equation}
For the rest terms on the right-hand side of \eqref{4.233}, the estimates from Proposition \ref{prop4.3} are not enough, and we need to perform further cumulant expansions on them. In particular, the first and third term on the right-hand side of \eqref{4.233} have true size $N^{4\gamma-4}$, and we will show in the next step that there is a cancellation between those two, thus their sum satisfies the desired bound. 

\emph{Step 2.} In this step we show there is a cancellation between the first and third term on the right-hand side of \eqref{4.233}. Let us consider the first term of \eqref{4.233}. Analogously to \eqref{4.11}, we see that for some fixed $l \in \bb N$,
\begin{multline} \label{4.24}
\bb E \langle \ul{G} \rangle^2 \langle \ul{F^*} \rangle^2=\frac{1}{T_*}\bigg(\bb E \langle\ul{G}\rangle^2\langle \ul{F^*}\rangle^3-\bb E \langle\ul{G}\rangle^2\langle \ul{F^*}\rangle\bb E\langle \ul{F^*}\rangle^2+\frac{1}{N}\bb E \langle \ul{G}  \rangle^2\langle \ul{F^*} \rangle \langle \ul{F^{*2}}\rangle+\frac{2}{N^2}\bb E \langle \ul{G} \rangle ^2\bb E\ul{F^{*3}}
\\
+\frac{2}{N^2}\bb E \langle \ul{G} \rangle ^2\langle\ul{F^{*3}}\rangle +\frac{4}{N^2}\bb E\langle\ul{G}\rangle \langle\ul{F^*} \rangle\langle \ul{G^2F^*}\rangle-\sum_{k=2}^l W^{(3)}_k+\frac{4}{N^2}\bb E\langle\ul{G}\rangle \langle\ul{F^*} \rangle\bb E\ul{G^2F^*}\bigg)+\OO(N^{-10})\,,
\end{multline}
where $W^{(3)}_k$ is defined analogously to $W^{(1)}_k$ in \eqref{4.11}. As in Section \ref{sec4.1}, we can use Proposition \ref{prop4.3} and Lemma \ref{useful} to conclude that on the right-hand side of \eqref{4.24}, all but the last term are bounded by $\cal E_1$. Thus
\begin{equation*}
\begin{aligned}
\bb E \langle \ul{G} \rangle^2 \langle \ul{F^*} \rangle^2
=\frac{4}{N^2T_*}\bb E\langle\ul{G}\rangle \langle\ul{F^*} \rangle\bb E\ul{G^2F^*}+\cal E_1
=-\frac{4}{N^2(z_1-z_2^*)^2}\bb E \langle\ul{G}\rangle \langle\ul{F^*} \rangle+\cal E_1\,,
\end{aligned}
\end{equation*}
where in the second step we used Lemmas \ref{resolvent} and \ref{lem4.3}(i). Summing the first and third terms on the right-hand side of \eqref{4.233}, we have
\begin{multline} \label{watsthestory}
\bb E \langle \ul{G} \rangle^2 \langle \ul{F^*} \rangle^2+\frac{4}{N^2}\bb E\langle\ul{G}\rangle \langle \ul{G^2F^*}\rangle\\
= \bb E \langle \ul{G} \rangle^2 \langle \ul{F^*} \rangle^2+\frac{4}{N^2(z_1-z_2^*)^2}\bb E \langle\ul{G}\rangle \langle\ul{F^*} \rangle-\frac{4}{N^2(z_1-z_2^*)^2}\bb E \langle\ul{G}\rangle^2+\frac{4}{N^2(z_1-z_2^*)}\bb E \langle\ul{G}\rangle \langle\ul{G^2} \rangle=\cal E_1\,,
\end{multline}
where in the second step the first two terms have a cancellation, and we used Proposition \ref{prop4.3} to estimate the last two terms. 

\emph{Step 3.} In this step we apply the cumulant expansion on the fourth term on the right-hand side of \eqref{4.233}. Once again, as in \eqref{4.11}, we see that for some fixed $l \in \bb N$,
\begin{equation*} 
\begin{aligned}
\bb E \langle \ul{G} \rangle^2 \langle \ul{F^{*2}} \rangle&=\frac{1}{T_*}\bigg(\bb E \langle\ul{G}\rangle^2\langle \ul{F^*}\rangle+2\bb E \langle\ul{G}\rangle^2\langle\ul{F^{*2}}\rangle \langle\ul{F^*}\rangle-2\bb E \langle\ul{G}\rangle^2\bb E\langle\ul{F^{*2}}\rangle\langle \ul{F^*}\rangle+\frac{2}{N}\bb E \langle \ul{G}  \rangle^2\langle \ul{F^{*3}}\rangle\\
&\msp+2\bb E \langle \ul{G} \rangle^2 \langle \ul{F^{*}} \rangle \bb E \ul{F^{*2}}+\frac{4}{N^2}\bb E \langle \ul{G} \rangle \langle\ul{G^2F^{*2}}\rangle -\sum_{k=2}^l W^{(4)}_k\bigg)+\OO(N^{-10})\,,
\end{aligned}
\end{equation*} 
where $W^{(4)}_k$ is defined analogously to $W^{(1)}_k$ in \eqref{4.7}. By estimating the above terms using Proposition \ref{prop4.3} and Lemma \ref{useful}, we have
\begin{equation} \label{qqq}
\frac{1}{N}\bb E \langle \ul{G} \rangle^2 \langle \ul{F^{*2}} \rangle=\cal E_1\,.
\end{equation}

\emph{Step 4.} Finally we deal with $W_3^{(2)}$ in \eqref{4.233}. By applying the differential \eqref{3.15} carefully, we see that the two dangerous terms in $W^{(2)}_3$ are
\begin{equation} \label{429}
\frac{1}{6N}\sum\limits_{i,j} \cal C_4(H_{ij}) (1+\delta_{ij})^{-3}\bb E \langle \ul{G} \rangle^2 F^{*2}_{ii}F^{*2}_{jj}\,,
\end{equation}
and
\begin{equation} \label{430}
\frac{1}{6N^2}\sum\limits_{i,j} \cal C_4(H_{ij}) (1+\delta_{ij})^{-3}\bb E \langle \ul{G} \rangle (G^2)_{ii}G_{jj}F^{*}_{ii}F^{*}_{jj}\,,
\end{equation}
and all other terms are bounded by $\OO(N^{3\gamma-4})$ by Proposition \ref{prop4.3}. Let us first consider \eqref{429}. Similar as in \eqref{4.11}, by writing $z_1\bb E \langle \ul{G} \rangle^2 F^{*2}_{ii}F^{*2}_{jj}=\frac{1}{N}\sum_{a,b}\bb E \langle\langle \ul{G}\rangle F^{*2}_{ii}F^{*2}_{jj}\rangle G_{ab}H_{ba}$ and applying formula \eqref{eq:cumulant_expansion} on the last averaging, we have for some fixed $l \in \bb N$,
\begin{multline*} 
\bb E \langle \ul{G} \rangle^2 F^{*2}_{ii}F^{*2}_{jj}=\frac{1}{T}\bigg(\bb E \langle \ul{G} \rangle^3 F^{*2}_{ii}F^{*2}_{jj}-\bb E \langle \ul{G} \rangle F^{*2}_{ii}F^{*2}_{jj} \bb E \langle \ul{G} \rangle^2+\frac{1}{N} \bb E \langle \ul{G} \rangle \langle \ul{G^2} \rangle F^{*2}_{ii}F^{*2}_{jj} +\frac{2}{N^2}\bb E  F^{*2}_{ii}F^{*2}_{jj} \bb E \ul{G^3}\\
\msp+\frac{2}{N^2}\bb E  F^{*2}_{ii}F^{*2}_{jj} \langle \ul{G^3} \rangle+\frac{4}{N^2}\bb E \langle \ul{G} \rangle (G^2F^*)_{ii}F^{*}_{ii}F^{*2}_{jj}+\frac{4}{N^2}\bb E \langle \ul{G} \rangle (G^2F^*)_{jj}F^{*2}_{ii}F^{*}_{jj}
\msp-\sum_{k=2}^l W^{(5)}_k\bigg)+\OO(N^{-10})\,,
\end{multline*}
where $W^{(5)}_k$ is defined analogously to $W^{(1)}_k$ in \eqref{4.11}. Note that $\cal C_4(H_{ij})=\OO(N^{-2})$. Now we plug the last equation into \eqref{429}. By estimating the result using Proposition \ref{prop4.3} and Lemma \ref{useful}, we see that $\eqref{429}=\OO(N^{3\gamma-4})$. Similarly, for \eqref{430} we have
\begin{multline} \label{433vayne}
\bb E \langle \ul{G} \rangle (G^2)_{ii}G_{jj}F^{*}_{ii}F^{*}_{jj}
\\
=\frac{1}{T}\bigg(\bb E \langle \ul{G} \rangle^2 (G^2)_{ii}G_{jj}F^{*}_{ii}F^{*}_{jj} -\bb E (G^2)_{ii}G_{jj}F^{*}_{ii}F^{*}_{jj}\bb E \langle \ul{G} \rangle^2 
\msp+\frac{1}{N}\bb E \langle\ul{G^2}\rangle (G^2)_{ii}G_{jj}F^{*}_{ii}F^{*}_{jj}\\+\frac{4}{N^2}\bb E  (G^3)_{ii}G_{jj}F^{*}_{ii}F^{*}_{jj}+\frac{2}{N^2}\bb E  (G^2)_{ii}(G^2)_{jj}F^{*}_{ii}F^{*}_{jj}
\msp+\frac{2}{N^2}\bb E (G^2)_{ii} (G^2)_{jj}F^{*}_{ii}F^{*}_{jj}
\\\msp
+ \frac{2}{N^2}\bb E (G^2)_{ii}G_{jj}(GF^{*2})_{ii}F^{*}_{jj}
+\frac{2}{N^2}\bb E  (G^2)_{ii}G_{jj}F^*_{ii}(GF^{*2})_{jj}-\sum_{k=2}^l W^{(6)}_k\bigg)+\OO(N^{-10})
\end{multline}
for some fixed $l \in \bb N$, where $W^{(6)}_k$ is defined analogously to $W^{(1)}_k$ in \eqref{4.11}. Now we plug \eqref{433vayne} into \eqref{430} and estimate the result by Proposition \ref{prop4.3} and Lemma \ref{useful}. We get $\eqref{430}=\OO(N^{3\gamma-4})+\OO(N^{\gamma-3})$. Thus we have arrived at
\begin{equation} \label{434}
W_3^{(2)}=\OO(N^{3\gamma-4}+N^{\gamma-3})\leq\cal E_1\,.
\end{equation}
Combining \eqref{4.233}, \eqref{425}, \eqref{watsthestory}, \eqref{qqq}, \eqref{434} and using Cauchy-Schwarz Inequality, we conclude the proof.
\end{proof} 

In the next lemma we deal with the second term on the right-hand side of \eqref{417}. By comparing with \eqref{498} below, one observes that this term does not appear in the complex Hermitian case, and we shall see that it is indeed this term that results in the correction term of order $N^{4\gamma-4}$, which explains why we only have the correction term in the real symmetric case. 
\begin{lemma} \label{lem44}
Recall $f_1$ from \eqref{488}. For the second term on the right-hand side of \eqref{417}, we have
\begin{equation*}
\begin{aligned}
\frac{1}{TN}\bb E\langle\ul{G^2}\rangle \langle \ul{F^*}\rangle=\frac{f_1(z_1,z^*_2)}{N^3(z_1-z_2^*)^3}+\frac{2E\,\mathrm{i}}{N^3\kappa^3(z_1-z_2^*)^2}+\frac{12}{N^4(z_1-z_2^{*})^4\kappa^2}+\cal E_1\,.
\end{aligned}
\end{equation*}
\end{lemma}

\begin{proof}
The proof is similar to that of Lemma \ref{lem4.2}, i.e.\ we keep using cumulant expansion formula \ref{eq:cumulant_expansion} until we can apply Proposition \ref{prop4.3} to estimate the results as desired. 

\emph{Step 1.} In this step we perform the cumulant expansion for the term $\bb E\langle\ul{G^2}\rangle \langle \ul{F^*}\rangle$. Similar to \eqref{4.11}, we have for some fixed $l \in \bb N$ that
\begin{equation*}
\bb E \langle \ul{G^2} \rangle \langle \ul{F^*} \rangle=\frac{1}{T_*} \bigg(\bb E \langle \ul{G^2} \rangle \langle \ul{F^*} \rangle^2 +\frac{1}{N}\bb E \langle \ul{G^2} \rangle \langle \ul{F^{*2}} \rangle+\frac{2}{N^2}\bb E \ul{G^3F^{*}}-\sum_{k=2}^l W^{(7)}_k+O(N^{-10})\bigg)\,,
\end{equation*}
where 
\begin{equation*}
W^{(7)}_k\deq \frac{1}{N} \sum_{i,j}\frac{1}{k!}\cal C_{k+1}(H_{ji}) \bb E \partial^k_{ji} (\langle \ul{G^2}\rangle (F^*)_{ij})\,,
\end{equation*}
As in Section \ref{sec4.1}, we can use Proposition \ref{prop4.3} and Lemma \ref{useful} to show that
$
T_*^{-1}W_k^{(7)}=\OO(N^{2\gamma-2.5})
$
for $k=2$ and $k\ge 4$. Thus
\begin{equation} \label{441}
\frac{1}{TN}\bb E \langle \ul{G^2} \rangle \langle \ul{F^*} \rangle=\frac{1}{TT_*N} \bigg(\bb E \langle \ul{G^2} \rangle \langle \ul{F^*} \rangle^2 +\frac{1}{N}\bb E \langle \ul{G^2} \rangle \langle \ul{F^{*2}} \rangle+\frac{4}{N^2}\bb E \ul{G^3F^{*}}- W^{(7)}_3\bigg)+\OO(N^{2\gamma-7/2})\,.
\end{equation}
Now we analyse each term on the right-hand side of \eqref{441} separately. In Steps 2 to 5 below, we are going to compute/estimate the first to fourth term on the right-hand side of \eqref{441} respectively.

\emph{Step 2.} Similar to \eqref{4.11}, we have
\begin{equation*}
\begin{aligned}
\bb E \langle \ul{G^2} \rangle \langle \ul{F^*} \rangle^2 &=\frac{1}{T}\bigg(\bb E \langle \ul{G} \rangle\langle \ul{F^*} \rangle^2+2\bb E \langle \ul{G^2} \rangle\langle\ul{G} \rangle \langle \ul{F^*} \rangle^2-2\bb E \langle \ul{F^*} \rangle^2	\bb E \langle \ul{G^2} \rangle \langle \ul{G} \rangle +\frac{2}{N} \bb E \langle \ul{G^3}\rangle  \langle\ul{F^*}\rangle^2\\
&\msp+\bb E \langle \ul{G}\rangle \langle \ul{F^*}\rangle^2 \bb E \ul{G^2}+\frac{4}{N^2}\bb E \langle \ul{F^{*}}\rangle \langle \ul{G^2F^{*2}} \rangle -\sum_{k=2}^l W^{(8)}_k+\OO(N^{-10})\bigg)
\end{aligned}
\end{equation*}
for some fixed $l \in \bb N$, where $W^{(8)}_k$ is defined analogously to $W^{(1)}_k$ in \eqref{4.11}. Estimating the above by Proposition \ref{prop4.3} and Lemma \ref{useful} shows
\begin{equation} \label{4433}
\frac{1}{TT_*N}\bb E \langle \ul{G^2} \rangle \langle \ul{F^*} \rangle^2=\OO(N^{3\gamma-4}+N^{5\gamma-5})=\cal E_1\,.
\end{equation}

\emph{Step 3.} Similar to \eqref{4.11}, we have
\begin{equation} \label{444}
\begin{aligned}
\bb E \langle \ul{G^2} \rangle \langle \ul{F^{*2}} \rangle&=\frac{1}{T_*}\bigg(\bb E \langle \ul{G^2}\rangle \langle \ul{F^*}\rangle +2\bb E \langle\ul{G^{2}} \rangle \langle \ul{F^{*2}}\rangle \langle\ul{F^*}\rangle+\frac{2}{N}\bb E \langle \ul{F^{3*}}\rangle\langle \ul{G^2}\rangle+2\bb E \langle\ul{G^2} \rangle \langle \ul{F^*}\rangle \bb E \ul{F^{*2}}\\&\msp+\frac{4}{N^2}\bb E \ul{G^3F^{*2}}-\sum_{k=2}^l W^{(9)}_k+\OO(N^{-10})\bigg)
\end{aligned}
\end{equation}
for some fixed $l \in \bb N$, where $W^{(9)}_k$ is defined analogously to $W^{(1)}_k$ in \eqref{4.11}. We are going to see the fifth term of the right-hand side of \eqref{444} is the one that gives us the term of order $N^{4\gamma-4}$, and this is the term does not appear in the complex Hermitian case. By Proposition \ref{prop4.3} and Lemma \ref{useful}, we see that all but the fifth term on the right-hand side of \eqref{444} are bounded by $\OO(N^{3\gamma-2})+\OO(N^{5\gamma-3})\leq \cal E_1$. Thus by the resolvent identity we see that
\begin{multline} \label{4.46}
\bb E \langle \ul{G^2} \rangle \langle \ul{F^{*2}} \rangle=\frac{4}{T_*N^2}\bb E \ul{G^3F^{*2}}+\cal E_1\\
=\frac{4}{T_*N^2}\cdot \bigg( \frac{3(\bb E \ul{G}-\bb E \ul{F^*})}{(z_1-z_2^*)^4}+ \frac{\bb E \ul{G^3}}{(z_1-z_2^*)^2}-\frac{2\bb E \ul{G^2}+\bb E \ul{F^{*2}}}{(z_1-z_2^*)^3}\bigg)+\cal E_1
= \frac{12(\bb E \ul{G}-\bb E \ul{F^*})}{N^2(z_1-z_2^*)^4}+\cal E_1\,,
\end{multline}
where in the third step we use Proposition \ref{prop4.3}(i) to estimate $\bb E \ul{G^3}$, $\bb E \ul{G^2}$, and $\bb E \ul{F^{*2}}$ by $\OO(1)$\,.
Lemma \ref{lem4.3} implies 
\begin{equation} \label{449}
T=-z_1-2\bb E\ul{G}=-z_1-2m(z_1)+\OO(N^{-1})=-\sqrt{z_1^2-4}+\OO(N^{-1})=-\kappa\mathrm{i} +\OO(N^{-\gamma})\,.
\end{equation}
Thus Proposition \ref{prop4.3}, \eqref{4.46} and \eqref{449} shows
\begin{equation} \label{446}
\frac{1}{TT_*N^2}\bb E \langle \ul{G^2} \rangle \langle \ul{F^{*2}} \rangle=\frac{12}{N^4(z_1-z_2^{*})^4\kappa^2}+\cal E_1\,.
\end{equation}

\emph{Step 4.} By resolvent identity and Lemma \ref{lem4.3} we have
\begin{align} \label{452}
\frac{4}{TT_*N^3}\bb E \ul{G^3F^*}  &=\frac{4}{TT_*N^3}\bigg(\frac{\bb E \ul{G}-\bb E \ul{F^*}}{(z_1-z_2^*)^3}-\frac{\bb E \ul{G^2}}{(z_1-z_2^*)^2}+\frac{\bb E \ul{G^3}}{z_1-z_2^*}\bigg)
\\
&=-\frac{f_1(z_1,z_2^*)}{N^3(z_1-z_2^*)^3}+\frac{2E\,\mathrm{i}}{N^3\kappa^{3}(z_1-z_2^*)^2}+\cal E_1\,,
\end{align}
where in the second step we used Proposition \ref{prop4.3} to estimate $\bb E \ul{G^3}$ by $\OO(1)$.

\emph{Step 5.} Lastly we look at $W_3^{(7)}$. By applying the differential rule \eqref{3.15} carefully, we see that the two non-trivial terms in $W^{(7)}_3$ are
\begin{equation} \label{454}
\frac{1}{6N}\sum\limits_{i,j} \cal C_4(H_{ij}) (1+\delta_{ij})^{-3}\bb E \langle \ul{G^2} \rangle F^{*2}_{ii}F^{*2}_{jj}\quad \mbox{ and } \quad 
\frac{1}{6N^2}\sum\limits_{i,j} \cal C_4(H_{ij}) (1+\delta_{ij})^{-3}\bb E (G^3)_{ii}G_{jj}F^{*}_{ii}F^{*}_{jj}\,,
\end{equation}
where all other terms are bounded by $\OO(N^{3\gamma-3})$ by Proposition \ref{prop4.3}. Similarly to the way we dealt with \eqref{429} and \eqref{430}, we can estimate \eqref{454} and show that
\begin{equation} \label{456}
\frac{1}{TT_*N}W^{(7)}_3= \OO(N^{\gamma-3})+\OO(N^{3\gamma-4})\leq \cal E_1\,.
\end{equation}

By plugging \eqref{4433}, \eqref{446}, \eqref{452}, \eqref{456} into \eqref{441}, we conclude the proof of the lemma.
\end{proof}

\begin{lemma} \label{lem4.5}
Recall $f_{2+}$ from \eqref{488}. For the third term on the right-hand side of \eqref{417}, we have
\begin{multline*}
\frac{2}{TN^2}\bb E \ul{GF^{*2}}
\\
=-\frac{2}{N^2(z_1-z^*_2)^2}+\frac{1}{N^2}f_{2+}(z_1,z_2^*)
+\frac{\mathrm{i}}{N^3(z_1-z_2^{*})^2}\bigg(-\frac{E}{\kappa^{3}}+\frac{2\im m(E)^4}{\kappa^2}\sum_{i,j}\cal C_4(H_{ij})\bigg)+\cal E_1
\end{multline*}
\end{lemma}
\begin{proof}
The proof is a straightforward computation using the resolvent identity and Lemma \ref{lem4.3}. By the resolvent identity and $T=-z_1-2\bb E \ul{G}$ we have
\begin{equation} \label{458}
\frac{1}{T}\bb E \ul{GF^{*2}}
=-\frac{1}{(z_1-z_2^*)^2}+\frac{1}{z_1+2\bb E\ul{G}}\bigg(\frac{z_1+\bb E \ul{G}+\bb E \ul{F^*}}{(z_1-z_2^*)^2}+\frac{\bb E \ul{F^{*2}}}{z_1-z_2^*}\bigg)\,.
\end{equation}
By Lemma \ref{lem4.3}(i) we see that
\begin{multline*}
\frac{z_1+\bb E \ul{G}+\bb E \ul{F^*}}{(z_1-z_2^*)^2}+\frac{\bb E \ul{F^{*2}}}{z_1-z_2^*} = \frac{4+z_1z_2^*+\sqrt{z_1^2-4}\sqrt{z_2^{*2}-4}}{2\sqrt{z_2^{*2}-4}(\sqrt{z_1^2-4}-\sqrt{z_2^{*2}-4})^2}
\\
+\frac{1}{N(z_1-z_2^{*})^2}\bigg(\frac{E}{\kappa^2}-\frac{2\im m(E)^4}{\kappa}\sum_{i,j}\cal C_4(H_{ij})\bigg)+\OO(N^{\gamma-1}+N^{2\gamma-3/2})\,,
\end{multline*}
then the lemma follows by \eqref{458} and the fact $z_1+2\bb E \ul{G}=\sqrt{z_1^2-4}+\OO(N^{-1})=\kappa\,\mathrm{i}+\OO(N^{-\gamma})$.
\end{proof}
\subsection{Third cumulant} \label{section4.3} In this section we analyze the third-cumulant term on the right-hand side of \eqref{417}. We follow the same strategy as in Section \ref{sec:4.2}, namely we repeatedly use cumulant expansion formula \eqref{eq:cumulant_expansion} until the results are either computable by Lemma \ref{lem4.3} or can be estimated by Proposition \ref{prop4.3}.
\begin{lemma} \label{lem4.6}
Recall $f_4$ from \eqref{f34}. For the fourth term on the right-hand side of \eqref{417}, we have
\begin{equation*}
-\frac{1}{T}W^{(1)}_2=\frac{1}{N^2}f_4(z_1,z_2^*)\sum_i\cal C_3(H_{ii})+\cal E_1\,.
\end{equation*}
\end{lemma}
\begin{proof}
	Note that $W_2^{(1)}$ contains two differentials, we split 
	$$
	W_2^{(1)}=W^{(1,1)}_2+W^{(1,2)}_2\,,
	$$
	where $W^{(1,1)}_2$ corresponds to the terms where $\langle \ul{F^*}\rangle $ is differentiated, and $W^{(1,2)}_2$ corresponds to the terms where $\langle \ul{F^*}\rangle $ is not differentiated. We shall see that the order $\OO(N^{-5/2})$ terms will come from $W^{(1,1)}_2$, and $W^{(1,2)}_2$ is of order $\OO(N^{2\gamma-7/2})$. 
	
\emph{Step 1.} In this step we consider the term $W^{(1,1)}_2$. By applying the differential rule \eqref{3.15} carefully, we have	
	\begin{equation} \label{466}
	\begin{aligned}
	W^{(1,1)}_2&=\frac{1}{2N^2}\sum_{i,j} \cal C_3(H_{ij}) (1+\delta_{ij})^{-2}\bb E \Big(2(F^{*2})_{ii}F^*_{jj}G_{ij}+2(F^{*2})_{jj}F^*_{ii}G_{ij}+4(F^{*2})_{ij}F^*_{ij}G_{ij}\\
	&\msp+4(F^{*2})_{ij}G_{ii}G_{jj}+4(F^{*2})_{ij}G_{ij}^2\Big)\,.
	\end{aligned}
	\end{equation}
	From Lemma \ref{lem4.3} we see that $\bb E \ul{G} =m(z_1)+\OO(N^{-1})$. In the same spirit, we are going to show that in our computations for the above terms, $F^*_{ii}$ can be treated as $m(z_2^*)$ up to an error of order $\OO(N^{-1})$, and we have analogous approximations for other matrix entries. Let us consider the first term on the right-hand side of \eqref{466}. As in \eqref{4.11}, by the identity $z_1G=GH-I$ and cumulant expansion formula \eqref{eq:cumulant_expansion} we have
	\begin{multline} \label{467}
\frac{1}{N^2}\sum_{i,j} \cal C_3(H_{ij}) (1+\delta_{ij})^{-2}\bb E (F^{*2})_{ii}F^*_{jj}G_{ij}
	=\frac{1}{N^2}\sum_{i,j} \cal C_3(H_{ij}) (1+\delta_{ij})^{-2}\frac{1}{U}\bigg(\bb E (F^{*2})_{ii}F^*_{jj}\delta_{ij}
	\\
	+\bb E (F^{*2})_{ii}F^*_{jj}G_{ij}\langle\ul{G}\rangle +\frac{1}{N}\bb E F^*_{ij}F^*_{jj}(F^{*2}G)_{ii}
	+\frac{1}{N}\bb E F^*_{ii}F^*_{jj}(F^{*2}G)_{ij}\\+\frac{1}{N}\bb E F^{*2}_{ij}F^*_{jj}(F^{*}G)_{ii}+\frac{1}{N}\bb E F^{*2}_{ij}F^*_{jj}(F^{*}G)_{ii}+\frac{1}{N}\bb E (F^{*2})_{ii}F^*_{jj}G^2_{ij}
	-\sum_{k=2}^lW^{(10)}_k+O(N^{-10})\bigg)
	\end{multline}
	for some fixed $l \in \bb N$, where $W^{(10)}_k$ is defined analogously to $W^{(1)}_k$ in \eqref{4.11}, and $U\deq -z_1-\bb E \ul{G}$. Now we estimate the second cumulant terms by Proposition \ref{prop4.3} and the rest terms by the strategy described in Section \ref{sec4.1}, and we see that all but the first term on the right-hand side of \eqref{467} are bounded by $\OO(N^{2\gamma-7/2})$, thus
	\begin{equation*}
	\eqref{467}=\frac{1}{4UN^2}\sum_i \cal C_3(H_{ii})\bb E  (F^{*2})_{ii}F^*_{ii}+\cal E_1\,.
	\end{equation*}
	Similarly, one further replace the factors $(F^{*2})_{ii}$ and $F^*_{ii}$ in the last equation by $m'(z_2^*)$ and $m(z_2^*)$ respectively, and get
	\begin{equation*} 
	\eqref{467}=-\frac{m(z_1)m(z_2^*)^2}{4\sqrt{z_2^{*2}-4}N^2}\sum_i\cal C_3(H_{ii})+\cal E_1\,.
	\end{equation*}
	By computing other terms in \eqref{466} in a similar fashion,
	we have
	\begin{equation} \label{473}
	W_2^{(1,1)}=-\frac{m(z_1)m(z_2^*)(m(z_2^*)+m(z_1))}{\sqrt{z_2^{*2}-4}N^2}\sum_i\cal C_3(H_{ii})+\cal E_1\,.
	\end{equation} 
	
	\emph{Step 2.} When $\langle \ul{F^*}\rangle $ is not differentiated, we have
	\begin{equation} \label{474}
	\begin{aligned}
	W^{(1,2)}_2=-\frac{1}{2N}\sum_{i,j} \cal C_3(H_{ij}) (1+\delta_{ij})^{-2}\bb E \langle \ul{F^{*}} \rangle \big(6G_{ii}G_{jj}G_{ij}+2G_{ij}^3\big)\,.
	\end{aligned}
	\end{equation}
As in Step 1, we first perform cumulant expansion for the above terms, and then estimate or compute the resulting terms. Let us consider the first term on the right-hand side of \eqref{474}. As in Step 1, we have
	\begin{equation*}
	-\frac{3}{N}\sum_{i,j} \cal C_3(H_{ij}) (1+\delta_{ij})^{-2}\bb E \langle \ul{F^{*}} \rangle G_{ii}G_{jj}G_{ij}=-\frac{3}{4U^3N}\sum_i \cal C_3(H_{ii}) \bb E \langle \ul{F^*} \rangle+\OO(N^{2\gamma-7/2})=\cal E_1\,.
	\end{equation*}
	By estimating the other term on the right-hand side of \eqref{474} in a similar fashion, we can show that it is also bounded by $\OO(N^{2\gamma-7/2})$. Thus
	\begin{equation} \label{479}
	W^{(1,2)}_2=\cal E_1\,.
	\end{equation}
	By Lemma \ref{lem4.3} we see that $T=-\sqrt{z_1^2-4}+\OO(N^{-1})$,
	together with \eqref{473} and \eqref{479} we conclude the proof.
\end{proof}
\subsection{Fourth cumulant} \label{sec4.22}  In this section we analyse the fourth-cumulant terms on the right-hand side of \eqref{417}. The strategy is the same as in the previous two sections, and we shall omit any extraneous details.
\begin{lemma} \label{lem4.7}
Recall $f_3$ from \eqref{f34}. For the fifth term on the right-hand side of \eqref{417}, we have
\begin{equation*}
-\frac{1}{T}W^{(1)}_3=\bigg(\frac{1}{N^2}f_3(z_1,z_2^*)+\frac{2E(E^2-2)\mathrm{i}}{(z_1-z_2^*)^2N^3\kappa}\bigg)\sum_{i,j}\cal C_4(H_{ij})+\cal E_1\,.
\end{equation*}
\end{lemma}
\begin{proof}
By applying the differential \eqref{3.15} carefully, we see that
\begin{align*} 
&-\frac{1}{T}W^{(1)}_3
\\
=&\;\frac{1}{6TN^2}\sum_{i,j} \cal C_4(H_{ij})(1+\delta_{ij})^{-3}\bb E \Big(12 (F^{*2})_{ii}F^*_{jj}F^*_{ij}G_{ij}+12(F^{*2})_{ij}F^*_{ii}F^*_{jj}G_{ij}
+6(F^{*2})_{ii}F^*_{jj}G_{ii}G_{jj}\\&+6(F^{*2})_{ii}F^*_{jj}G_{ij}^2+12(F^{*2})_{jj}F^*_{ii}F^*_{ij}G_{ij}
+6(F^{*2})_{jj}F^*_{ii}G_{ii}G_{jj}+6(F^{*2})_{jj}F^*_{ii}G_{ij}^2\\&+12(F^{*2})_{ij}F^{*2}_{ij}G_{ij}
+12(F^{*2})_{ij}F^*_{ij}G_{ii}G_{jj}+12(F^{*2})_{ij}F^*_{ij}G_{ij}^2+36(F^{*2})_{ij}G_{ii}G_{jj}G_{ij}\\
&+12(F^{*2})_{ij}G_{ij}^3+36N\langle\ul{F^*}\rangle G_{ii}G_{jj}G_{ij}^2+6N\langle \ul{F^*}\rangle G_{ii}^2G_{jj}^2+6N\langle \ul{F^*}\rangle G_{ij}^4 \Big)\,.
\end{align*}
The terms on the right-hand side of the above can be estimated in similar way as in the proof of Lemma \ref{lem4.6}. As a result, $-T^{-1}W^{(1)}_3$ simplifies to
\begin{multline*}
\sum_{i,j}\frac{(1+\delta_{ij})^{-3}}{6TN^2} \cal C_4(H_{ij})\bb E \Big(12(F^{*2})_{ij}F^*_{ii}F^*_{jj}G_{ij}
+12(F^{*2})_{ii}F^*_{jj}G_{ii}G_{jj}
\\
+36(F^{*2})_{ij}G_{ii}G_{jj}G_{ij}+6N\langle \ul{F^*}\rangle G_{ii}^2G_{jj}^2 \Big)+\cal E_1\\=-\frac{2m(z^*_2)^2}{(z_1-z_2^*)^2N^3}\sum_{i,j}\cal C_4(H_{ij})+\bigg(\frac{1}{N^2}f_3(z_1,z^*_2)-\frac{4}{(z_1-z_2^*)^2N^3}\bigg)\sum_{i,j}\cal C_4(H_{ij})\\-\frac{6m(z_1)^2}{(z_1-z_2^*)^2N^3}\sum_{i,j}\cal C_4(H_{ij})+\frac{8m(z_1)^3}{\sqrt{z_1^2-4}(z_1-z_2^*)^2N^3}\sum_{i,j} \cal C_4(H_{ij})+\cal E_1\,,
\end{multline*}
and we conclude the proof by using
\begin{equation*} 
-2m(z_2^*)^2-4-6m(z_1)^2+8m(z_1)^3(z_1^2-4)^{-1/2}=\frac{2E(E^2-2)\mathrm{i}}{\kappa}+\OO(N^{-\gamma})\,. \qedhere
\end{equation*}
\end{proof}

\subsection{Putting everything together} \label{sec:4.3} 
By plugging Lemmas \ref{lem4.2} -- \ref{lem4.7}  and \eqref{488} -- \eqref{491} into \eqref{417} we complete the proof Theorem \ref{thm_resolvent} in the real symmetric case.

\subsection{Complex Hermitian case} \label{sec:4.4} 
In this section we focus on Theorem \ref{thm_resolvent} when $H$ is complex Hermitian ($\beta = 2$). The proof is similar to the real symmetric case ($\beta = 1$), and we only indicate the arguments that differ from the previous ones. For any differentiable $f=f(H)$ we set
\begin{equation*} 
\partial_{ij}f(H)\deq \frac{\mathrm{d}}{\mathrm{d}t}\Big{|}_{t=0} f\pb{H+t\,\tilde{\Delta}^{(ij)}}\,,
\end{equation*}
where $\tilde{\Delta}^{(ij)}$ denotes the matrix whose entries are zero everywhere except at the site $(i,j)$ where it is one: $\tilde{\Delta}^{(ij)}_{kl} =\delta_{ik}\delta_{jl}$. Then by using Lemmas \ref{lem:cumulant_expansion} (for the diagonal entries of $H$) and \ref{lem:5.1} (for the off-diagonal entries of $H$) we have for any some fixed $l \in \bb N$,
\begin{equation} \label{4.5c}
\begin{aligned}
z_1\bb E \langle \ul{G} \rangle \langle \ul{F^*} \rangle=\frac{1}{N} \sum_{i,j}\bb E G_{ij}H_{ji}\langle \ul{F^*} \rangle
=\frac{1}{N^2} \sum_{i,j} \bb E \partial_{ij} (G_{ij}\langle \ul{F^*}\rangle)+\sum_{k=2}^l W^{(11)}_k+O(N^{-10})\,,
\end{aligned}
\end{equation}
where
\begin{equation} \label{494}
W^{(11)}_k\deq \frac{1}{N} \sum_{i,j}\sum_{p+q=k} \frac{1}{p!\,q!\,(1+\delta_{ij})^{k}}\cal C_{p,q+1}(H_{ij}) \bb E \partial^p_{ij}\partial^q_{ji}(G_{ij}\langle \ul{F^*}\rangle)\,.
\end{equation}
The $(1+\delta_{ij})^{-k}$ factor in \eqref{494} is because for $i=j$, we need to use Lemma \ref{lem:cumulant_expansion} for the real random variables $H_{ii}$, and we have
$$
\sum_{p+q=k} \frac{1}{p!\,q!\,2^k}=\frac{1}{k!}\,.
$$ Note that for the complex Hermitian case,
\begin{equation} \label{4.9}
\partial_{kl} G_{i  j}=-G_{ik}G_{kj}\,,
\end{equation}
which gives
\begin{equation*}
\frac{1}{N^2} \sum_{i,j} \bb E \partial_{ij} (G_{ij}\langle \ul{F^*}\rangle)=-2 \bb E \ul{G} \,\bb E \langle \ul{G} \rangle \langle \ul{F} \rangle -\bb E \langle \ul{G} \rangle^2 \langle \ul{F} \rangle-\frac{1}{N^2} \bb E \ul{GF^{*2}}\,.
\end{equation*}
By applying Lemma \ref{useful} and Proposition \ref{prop4.3}(ii), and following a similar argument as in Section \ref{sec4.1}, one can show that
\begin{equation*}
\frac{1}{T}W^{(11)}_k=\OO(N^{\gamma-3})
\end{equation*}
for all $k \ge 4$. Together with \eqref{4.5c} we have 
\begin{equation} \label{498}
\bb E \langle \ul{G} \rangle \langle \ul{F^*} \rangle=\frac{1}{T}\bigg(\bb E \langle \ul{G} \rangle^2 \langle \ul{F^*} \rangle+\frac{1}{N^2}\bb E \ul{GF^{*2}}-W^{(11)}_2-W^{(11)}_3\bigg)+\OO(N^{\gamma-3})\,.
\end{equation}
where $T=-z-2\bb E\ul{G}$. By a comparison of \eqref{417} and \eqref{498} we can already see some difference between the real and complex cases. For examples, the leading term $2(TN^2)^{-1}\bb E \ul{GF^{*2}}$ is halved in the complex case, and we do not have the term $(TN)^{-1}\bb E \langle \ul{G^2} \rangle \langle \ul{F^*}\rangle$. 

Now we take a closer look at the terms in \eqref{498}. As in Lemmas \ref{lem4.2} and \ref{lem4.5}--\ref{lem4.7}, we have the following lemma.
\begin{lemma} \label{lem4.8}
Recall $f_{2+},f_3,f_4$ from \eqref{488} and \eqref{f34}. For the terms on the right-hand side of \eqref{498}, we have the estimate \eqref{GF_est}; we also have
\begin{equation}\label{5.12 ii}
\frac{1}{TN^2}\bb E \ul{GF^{*2}}=-\frac{1}{N^2(z_1-z^*_2)^2}+\frac{1}{2N^2}f_{2+}(z_1,z_2^*)
+\frac{\im m(E)^4\mathrm{i}}{N^3(z_1-z_2^{*})^2\kappa^2}\sum_{i,j}\cal C_{2,2}(H_{ij})+\cal E_1\,,
\end{equation}
\begin{equation} \label{5.12 iii}
-\frac{1}{T}W^{(11)}_2= \frac{1}{N^2}f_4(z_1,z_2^*)\sum_i\cal C_3(H_{ii})+\cal E_1\,,
\end{equation}
and
\begin{equation} \label{5.12 iv}
-\frac{1}{T}W^{(11)}_3=\bigg(\frac{1}{N^2}f_3(z_1,z_2^*)+\frac{E(E^2-2)\mathrm{i}}{(z_1-z_2^*)^2N^3\kappa}\bigg)\sum_{i,j}\cal C_{2,2}(H_{ij})+\cal E_1\,.
\end{equation}
\end{lemma}
It is easy to see that \eqref{498} and Lemma \ref{lem4.8} imply \eqref{2.4ii}, which finishes the proof of Theorem \ref{thm_resolvent} for the complex Hermitian case. The proof of Lemma \ref{lem4.8} is similar to those of Lemmas \ref{lem4.2} and \ref{lem4.5} - \ref{lem4.7}, and we shall give a rough proof to sketch the difference.
\begin{proof} [Proof of Lemma \ref{lem4.8}]
(i) The proof of \eqref{GF_est} in the complex Hermitian case is similar to that of Lemma \ref{lem4.2}. By Lemmas \ref{lem:cumulant_expansion}, \ref{lem:5.1}, \ref{useful}, Proposition \ref{prop4.3}(ii), and following a similar argument as in Section \ref{sec4.1}, one can show that 
\begin{equation*} 
\bb E \langle \ul{G} \rangle^2 \langle \ul{F^*} \rangle=\frac{1}{T_*}\bigg(\bb E \langle\ul{G}\rangle^2\langle \ul{F^*}\rangle^2-\bb E \langle\ul{G}\rangle^2\bb E\langle \ul{F^*}\rangle^2+\frac{2}{N^2}\bb E\langle\ul{G}\rangle \langle \ul{G^2F^*}\rangle- W^{(12)}_3\bigg)+\cal E_1\,,
\end{equation*}
where $W_3^{(12)}$ is defined analogously as $W^{(11)}_3$ in \eqref{494}.
As in the proof of Lemma \ref{lem4.2}, we can show that 
\begin{equation*}
\bb E \langle \ul{G} \rangle^2 \langle \ul{F^*} \rangle^2+\frac{2}{N^2}\bb E\langle\ul{G}\rangle \langle \ul{G^2F^*}\rangle
=\cal E_1\,,\quad 
\bb E \langle\ul{G}\rangle^2\bb E\langle \ul{F^*}\rangle^2=\OO(N^{-4})\,,
\quad\mbox{and} \quad
W^{(12)}_3=\cal E_1\,.
\end{equation*}
The above relations together with $|T|^{-1}=\OO(1)$ imply the desired result.

(ii) The proof of \eqref{5.12 ii} is close to that of Lemma \ref{lem4.5}. We have
	\begin{equation*}
	\frac{1}{T}\bb E \ul{GF^{*2}}=-\frac{1}{(z_1-z_2^*)^2}+\frac{1}{z_1+2\bb E\ul{G}}\bigg(\frac{z_1+\bb E \ul{G}+\bb E \ul{F^*}}{(z_1-z_2^*)^2}+\frac{\bb E \ul{F^{*2}}}{z_1-z_2^*}\bigg)\,,
	\end{equation*}
	and together with Lemma \ref{lem4.3} (ii) we have the desired result.

(iii) The proof of \eqref{5.12 iii} is similar to that of Lemma \ref{lem4.6}. We split 
$$
W_2^{(11)}=W^{(11,1)}_2+W^{(11,2)}_2\,,
$$
where $W^{(11,1)}_2$ corresponds to the terms that $\langle \ul{F^*}\rangle $ is differentiated, and $W^{(11,2)}_2$ corresponds to the terms that $\langle \ul{F^*}\rangle $ is not differentiated. When $\langle \ul{F^*} \rangle $ is differentiated, we have	
\begin{equation*} 
\begin{aligned}
W^{(11,1)}_2&=\frac{1}{2N^2}\sum_{i,j} \frac{1}{(1+\delta_{ij})^{2}}\cal C_{2,1}(H_{ij}) \bb E \Big(2G_{ii}G_{jj}(F^{*2})_{ji}+2G_{ij}F^*_{ji}(F^{*2})_{ji}\Big)\\
&\msp+\frac{1}{N^2}\sum_{i,j}  \frac{1}{(1+\delta_{ij})^{2}}\cal C_{1,2}(H_{ij}) \bb E \Big(G_{ii}G_{jj}(F^{*2})_{ij}+G_{ij}^2(F^{*2})_{ij}+G_{ij}F^*_{ii}(F^{*2})_{jj}+G_{ij}F^*_{jj}(F^{*2})_{ii}\Big)\\
&\msp+\frac{1}{2N^2}\sum_{i,j}  \frac{1}{(1+\delta_{ij})^{2}}\cal C_{0,3}(H_{ij}) \bb E \Big(2G_{ij}^2(F^{*2})_{ij}+2G_{ij}F^*_{ij}(F^{*2})_{ij}\Big)\,.
\end{aligned}
\end{equation*}
By Lemmas \ref{lem:cumulant_expansion}, \ref{lem:5.1}, and \ref{useful}, together with Proposition \ref{prop4.3}, one can deduce that
\begin{equation*}
-\frac{1}{T}W^{(11,1)}_2=\frac{1}{N^2}f_4(z_1,z_2^*)\sum_i\cal C_3(H_{ii})+\cal E_1\,.
\end{equation*}
As in the proof of Lemma \ref{lem4.6}, we can also show that
$
-T^{-1}W_2^{(11,2)}=\cal E_1,
$
and this completes the proof.

(iv) The proof of \eqref{5.12 iv} is similar to that of Lemma \ref{lem4.7}. Let us split
\begin{equation*}
W^{(11)}_3=W^{(11,1)}_3+W^{(11,2)}_3+W^{(11,3)}_3+W^{(11,4)}_3\,,
\end{equation*}
where for $n=1,\dots,4$,
\begin{equation*}
W^{(11,n)}_3\deq  \frac{1}{N} \sum_{i,j} \frac{1}{(n-1)!\,(4-n)!\,(1+\delta_{ij})^{3}}\cal C_{n-1,5-n}(H_{ij}) \bb E \partial^{n-1}_{ij}\partial^{4-n}_{ji}(G_{ij}\langle \ul{F^*}\rangle)\,.
\end{equation*}
As in the proof of Lemma \ref{lem4.7}, we can show that
\begin{equation*}
W^{(11,n)}_3=\OO(N^{3\gamma-4}+N^{\gamma-3})
\end{equation*}
for $n=1,2,4.$ The leading terms are contained in $W^{(11,3)}_3$. More precisely, by applying \eqref{4.9} carefully we have
\begin{equation*}
\begin{aligned}
-\frac{1}{T}W^{(11,3)}_3&=\frac{1}{2TN^2}\sum_{i,j}(1+\delta_{ij})^{-3}\cal C_{2,2}(H_{ij}) \bb E \Big(4G_{ii}G_{jj}G_{ij}G_{ji}\langle\ul{F^*}\rangle+2G_{ii}^2G_{jj}^2\langle\ul{F^*}\rangle\\
&\msp+2G_{ii}G_{jj}G_{ji}(F^{*2})_{ij}+4G_{ii}G_{jj}G_{ij}(F^{*2})_{ji}+2G_{ii}G_{jj}(F^{*2})_{ii}F^*_{jj}+2G_{ii}G_{jj}(F^{*2})_{jj}F^*_{ii}\\
&\msp+2G_{ij}^2(F^{*2})_{ji}F^*_{ji}+2G_{ij}(F^{*2})_{ii}F^*_{jj}F^*_{ji}+2G_{ij}(F^{*2})_{jj}F^*_{ii}F^*_{ji}+2G_{ij}(F^{*2})_{ji}F^*_{ii}F^*_{jj}\Big)\,.
\end{aligned}
\end{equation*}
By Lemmas \ref{lem:cumulant_expansion}, \ref{lem:5.1}, and \ref{useful}, together with Proposition \ref{prop4.3}, we can show that $-T^{-1}W^{(11,3)}_3$ simplifies to
\begin{multline*} 
\frac{1}{2TN^2}\sum_{i,j}(1+\delta_{ij})^{-3}\cal C_{2,2}(H_{ij}) \bb E \Big(2G_{ii}^2G_{jj}^2\langle\ul{F^*}\rangle
+2G_{ii}G_{jj}G_{ji}(F^{*2})_{ij}+4G_{ii}G_{jj}G_{ij}(F^{*2})_{ji}\\+2G_{ii}G_{jj}(F^{*2})_{ii}F^*_{jj}+2G_{ii}G_{jj}(F^{*2})_{jj}F^*_{ii}
+2G_{ij}(F^{*2})_{ji}F^*_{ii}F^*_{jj}\Big)+\cal E_1
\\=\bigg(\frac{1}{N^2}f_3(z_1,z^*_2)+\frac{4m(z_1)^3}{\sqrt{z_1^2-4}(z_1-z_2^*)^2N^3}-\frac{3m(z_1)^2}{(z_1-z_2^*)^2N^3}
\\
-\frac{2}{(z_1-z_2^*)^2N^3}-\frac{m(z_2^*)^2}{(z_1-z_2^*)^2N^3}\bigg)\sum_{i,j}\cal C_{2,2}(H_{ij})+\cal E_1\,,
\end{multline*}
and the proof finishes by using
\begin{equation*} 
4m(z_1)^3(z_1^2-4)^{-1/2}-3m(z_1)^2-2-m(z_2^*)^2=\frac{E(E^2-2)\mathrm{i}}{\kappa}+\OO(N^{-\gamma})\,. \qedhere
\end{equation*}
\end{proof}

\section{Correlation of general functions: proof of Theorem \ref{mainthm}} \label{section5}
In this section we prove Theorem \ref{mainthm} for real and symmetric $H$; the complex case follows in a similar fashion.  
We begin with a version of the Cauchy integral formula for smooth functions, sometimes referred to as the Cauchy-Pompeiu formula.
\begin{lemma} \label{lem5.1}
Let $\phi \in C^{1}(\bb C)$ such that $\phi(z) = 0$ for large enough $\abs{\re z}$.
For $a>0$ define the strip $D_a=\{(x+\ii y)\col x \in \bb R, |y|\le a \}$. For any $\lambda \in \bb R$ we have
\begin{equation*}
\phi(\lambda)=\frac{\ii}{2\pi} \oint_{\partial D_a}\frac{\phi(z)}{\lambda-z}\, \dd z+\frac{1}{\pi} \int_{D_a} \frac{\partial_{\bar{z}} \phi (z)}{\lambda-z}\, \dd^2z\,.
\end{equation*}
\end{lemma}
\begin{proof}
Let $\varepsilon >0$ be small such that $B_{\varepsilon} (\lambda) \subset D_a$, where we use $B_{\varepsilon}(\lambda)$ to denote the ball centred at $\lambda$ with radius $\varepsilon$. The result follows by first using Green's formula for the function $\phi(z)/(\lambda-z)$ on $D_a\backslash B_{\varepsilon}(\lambda)$, and then letting $\varepsilon \downarrow 0$.
\end{proof}

\begin{proof} [Proof of Theorem \ref{mainthm}]Let us look at the real symmetric case. From \eqref{1.3}, one readily checks that $p_E(u,v)$ can be rewrote into
\begin{equation*}
\frac{1}{(N\varrho_E)^2}\sum_{i,j}\bb E \Big \langle \delta\Big(E+\frac{u}{N\varrho_E}-\lambda_i\Big) \Big \rangle \Big \langle \delta\Big(E+\frac{v}{N\varrho_E}-\lambda_j\Big) \Big\rangle
-\frac{1}{(N\varrho_E)^2} \sum_i \bb E \delta(u-v)\delta\Big(E+\frac{u}{N\varrho_E}-\lambda_i\Big)\,,
\end{equation*}
thus
\begin{multline} \label{n5.7}
\int p_E(u,v) f_{-}(u)g_{+}(v) \,\mathrm{d}u\,\mathrm{d}v=\frac{1}{N^2}\bb E \Big\langle \tr \frac{1}{\eta}f\Big(\frac{(H-E)\varrho_E-\omega}{\eta}\Big) \Big \rangle \Big\langle \tr \frac{1}{\eta}g\Big(\frac{(H-E)\varrho_E+\omega}{\eta}\Big) \Big \rangle\\
-\frac{1}{N^2\eta^2} \bb E \tr f\Big(\frac{(H-E)\varrho_E-\omega}{\eta}\Big)g\Big(\frac{(H-E)\varrho_E+\omega}{\eta}\Big) \,.
\end{multline}
Since $\supp f, \supp g \in [-M,M]$, where $3M\eta \leq \omega$, then $f(x-\omega/\eta)g(x+\omega/\eta)\equiv 0$.  
Thus the last term on right-hand side of \eqref{n5.7} vanishes, and  
\begin{equation} \label{n5.3}
\begin{aligned}
\int p_{E}(u,v) f_{-}(u)g_{+}(v) \,\mathrm{d}u\,\mathrm{d}v=\frac{1}{N^2}\bb E \Big\langle \tr \frac{1}{\eta}f\Big(\frac{(H-E)\varrho_E-\omega}{\eta}\Big) \Big \rangle \Big\langle \tr \frac{1}{\eta}g\Big(\frac{(H-E)\varrho_E+\omega}{\eta}\Big) \Big \rangle\,.
\end{aligned}
\end{equation}
(Note that \eqref{n5.3} is only valid for compactly supported functions $f$ and $g$; for more general $f,g$ (e.g.\ with polynomial decay), one also needs to compute the last term on the right-hand side of \eqref{n5.7}, which for simplicity we do not do in this paper.)

We now use Lemma \ref{lem5.1} and Theorem \ref{thm_resolvent}(i) to compute the right-hand side of \eqref{n5.3}. For any $f \in C^\infty_c(\R)$ and $k \in \bb N_+$, let $\tilde{f}$ denote the almost analytic extension of $f$ of order $k$, defined as
\begin{equation*}
\tilde{f}(x+\ii y) \deq f(x)+\sum_{j=1}^{k}\frac{1}{j!}(\ii y)^j f^{(j)}(x)\,.
\end{equation*}
We abbreviate $f_{\eta}(x)=\eta^{-1}f(((x-E)\varrho_E-\omega)/\eta)$ and $g_{\eta}(x)=\eta^{-1}g(((x-E)\varrho_E+\omega)/\eta)$. We apply Lemma \ref{lem5.1} with $a\deq\sqrt{\eta/N}$ to the functions $\phi = \tilde f_\eta$ and $\phi = \tilde g_\eta$, where we choose $k\deq\ceil{6/\tau}$. (From the computations below (e.g.\ \eqref{5.8}), we see that a larger $k$ results in a smaller bound for the error terms.) This gives
\begin{equation*}
f_{\eta}(H)=\frac{\ii}{2\pi} \oint_{\partial D_a}\tilde f_{\eta}(z)G(z)\, \dd z+\frac{1}{\pi} \int_{D_a} \partial_{\bar{z}}\tilde f_{\eta}(z)G(z)\, \dd^2z\,,
\end{equation*}
and together with \eqref{n5.3} we have
\begin{equation} \label{n5.4}
\begin{aligned}
\int p_{E}(u,v) f_{-}(u)g_{+}(v) \,\mathrm{d}u\,\mathrm{d}v=&\,-\frac{1}{4\pi^2} \oint_{\partial D_a}\oint_{\partial D_a}\tilde f_{\eta}(z_1)\tilde g_{\eta}(z_2)\bb E \langle \ul{G}(z_1)\rangle  \langle \ul {G}(z_2)\rangle \, \dd z_1\,\dd z_2\\
&\,+\frac{\ii}{2\pi^2} \int_{D_a} \oint_{\partial D_a} \tilde{f}_{\eta}(z_1)\partial_{\bar{z}}\tilde g_{\eta}(z_2)\bb E \langle \ul {G}(z_1)\rangle  \langle \ul {G}(z_2)\rangle\, \dd z_1 \,\dd^2z_2\\
&\,+\frac{\ii}{2\pi^2} \oint_{\partial D_a} \int_{ D_a} \partial_{\bar{z}}\tilde{f}_{\eta}(z_1)\tilde g_{\eta}(z_2)\bb E \langle \ul G(z_1)\rangle  \langle \ul G(z_2)\rangle\, \dd^2 z_1 \,\dd z_2\\
&\, + \frac{1}{\pi^2} \int_{ D_a} \int_{ D_a} \partial_{\bar{z}}\tilde{f}_{\eta}(z_1)\partial_{\bar{z}}\tilde g_{\eta}(z_2)\bb E \langle \ul G(z_1)\rangle  \langle \ul G(z_2)\rangle\, \dd^2 z_1 \,\dd^2 z_2\,.
\end{aligned}
\end{equation}

Let us first look at the last term on the right-hand side of \eqref{n5.4}. We write $z_1=x_1+\ii y_1$, and $z_2=x_2+\ii y_2$. By \eqref{3.4} we know
\begin{equation*}
| \langle \ul G(z_1)\rangle| \prec \frac{1}{|Ny_1|}  \ \ \mbox{and}\ \  |\langle \ul G(z_2)\rangle | \prec \frac{1}{|Ny_2|}\,,
\end{equation*}
together with $\partial_{\bar{z}}f_{\eta}(z)=\frac{1}{2k!}(\ii y)^kf^{(k+1)}_{\eta}(x)$ we have 
\begin{equation*}
\begin{aligned} 
&\mspace{-40mu}\bigg|\frac{1}{\pi^2}\int_{ D_a} \int_{ D_a} \partial_{\bar{z}}\tilde{f}_{\eta}(z_1)\partial_{\bar{z}}\tilde g_{\eta}(z_2)\bb E \langle \ul G(z_1)\rangle  \langle \ul G(z_2)\rangle\, \dd^2 z_1 \,\dd^2 z_2\,\bigg|
\\
&\prec \frac{1}{N^2\eta^{2k+2}} \int_{-a}^{a} \int_{-a}^{a} \big|y_1^{k-1}y_2^{k-1}\big|\,\dd y_1 \dd y_2 \prec  \frac{1}{N^2\eta^{2k+2}}\cdot \Big(\frac{\eta}{N}\Big)^k=\frac{1}{(N\eta)^{k+2}}=\OO(N^{-6})\,,
\end{aligned}
\end{equation*}
where in the second step we use the fact $\int |f^{(k+1)}_{\eta}| \prec \eta^{-(k+1)}$.

A similar computation works for the third term on the right-hand side of \eqref{n5.4}. More precisely, note that $\partial D_a=\{z \in \bb C\col \im z=a\}\cup \{z \in \bb C\col \im z=-a\}$, and we have
\begin{equation} \label{5.8}
\begin{aligned} 
&\mspace{-20mu}\bigg|\frac{\ii}{2\pi^2} \int_{\{\im z_2=a\}} \int_{ D_a} \partial_{\bar{z}}\tilde{f}_{\eta}(z_1)\tilde g_{\eta}(z_2)\bb E \langle \ul G(z_1)\rangle  \langle \ul G(z_2)\rangle\, \dd^2 z_1 \,\dd z_2\,\bigg| \notag\\
&\prec  \frac{1}{N^2} \int_{ D_a} \Big|\frac{1}{y_1a}y_1^kf^{(k+1)}_{\eta}(x_1)\Big|\, \dd^2 z_1 \prec \frac{1}{N^2a\eta^{k+1}}\int_{-a}^{a} \big|y_1^{k-1}\big|\, \dd y_1\prec \frac{1}{(N\eta)^{(k+3)/2}} =\OO(N^{-3})\,,
\end{aligned}
\end{equation}
where in the second step we use the fact $\int |\tilde g_{\eta}(x+\ii a)| \,\dd x =\OO(1)$. In the same way we can estimate the above with $\{\im z_2=a\}$ replaced by $\{\im z_2=-a\}$, and this completes the estimate of the third term on the right-hand side of \eqref{n5.4}.  

The same thing works for the second term on the right-hand side of \eqref{n5.4}, and we arrive at
\begin{equation} \label{510}
\int p_{E}(u,v) f_{-}(u)g_{+}(v) \,\mathrm{d}u\,\mathrm{d}v=\,-\frac{1}{4\pi^2} \oint_{\partial D_a}\oint_{\partial D_a}\tilde f_{\eta}(z_1)\tilde g_{\eta}(z_2)\bb E \langle \ul{G}(z_1)\rangle  \langle \ul {G}(z_2)\rangle \, \dd z_1\,\dd z_2+\OO(N^{-3})\,.
\end{equation}
Now again we split $\partial D_a=\{z \in \bb C\col \im z=a\}\cup \{z \in \bb C\col \im z=-a\}$. It follows from \eqref{2.4i} that
\begin{multline} \label{515}
-\frac{1}{4\pi^2}\int_{\{\im z_2=-a\}}\int_{\{\im z_1=a\}}\tilde f_{\eta}(z_1)\tilde g_{\eta}(z_2)\bb E \langle \ul{G}(z_1)\rangle  \langle \ul {G}(z_2)\rangle \, \dd z_1\,\dd (-z_2)\\=-\frac{1}{4\pi^2}
\int_{+\infty}^{-\infty}\int_{-\infty}^{+\infty}\tilde f_{\eta}(x_1+\ii a)\tilde g_{\eta}(x_2-\ii a)F_4(x_1+\ii a, x_2-\ii a) \, \dd x_1\,\dd x_2
+\cal E_1\,.
\end{multline}
 Here we abbreviate
\begin{multline*}
F_4(z_1, z_2)
=-\frac{2}{N^2(z_1-z_2)^2}+\frac{f_1(z_1,z_2)}{N^3(z_1-z_2)^3}+\frac{12}{N^4(z_1-z_2)^4\kappa^2}
+\frac{1}{N^2}\Big(f_{2+}(z_1,z_2)\\+f_3(z_1,z_2)\sum_{i,j}\cal C_4(H_{ij})+f_4(z_1,z_2)\sum_{i}\cal C_3(H_{ii})\Big)
+\frac{\ii }{N^3(z_1-z_2)^2}\Big(-\frac{E}{\kappa^{3}}+V(E)\sum_{i,j}\cal C_4(H_{ij})\Big)
\end{multline*}
for all $(z_1,z_2) \in A\times B \deq \{z \in \bb C\col |\re z|\le 2-\delta, \im z \ge 0\}\times \{z \in \bb C\col |\re z|\le 2-\delta, \im z \le 0\}$, where $\delta$ is any fixed positive number, and we recall the definition of $f_1,...,f_4$ and $V$ from \eqref{488} -- \eqref{491}. 

Now we define $D_a^1\eqd \{z \in \bb C\col 0 \le \im z\le a\}$ and $D_a^2 \eqd \{z \in \bb C\col -a \le \im z\le 0\}$, applying Green's Theorem twice
gives
\begin{align} \label{gaoqi}
&-\frac{1}{4\pi^2}
\int_{+\infty}^{-\infty}\int_{-\infty}^{+\infty}\tilde f_{\eta}(x_1+\ii a)\tilde g_{\eta}(x_2-\ii a)F_4(x_1+\ii a, x_2-\ii a) \, \dd x_1\,\dd x_2 \notag \\
=&{}-\frac{1}{4\pi^2}
\int_{-\infty}^{+\infty}\int_{+\infty}^{-\infty} f_{\eta}(x_1) g_{\eta}(x_2)F_4(x_1, x_2) \, \dd x_1\,\dd x_2
\notag \\
&+\frac{1}{4\pi^2}
\int_{-\infty}^{+\infty}\int_{D_a^1} \partial_{\bar{z}}\tilde f_{\eta}(z_1)  g_{\eta}(x_2)F_4(z_1, x_2) \, \dd^2 z_1\,\dd x_2\notag \\
&-\frac{1}{4\pi^2}\int_{D_a^2}\int_{-\infty}^{+\infty} \tilde f_{\eta}(x_1+\ii a) \partial_{\bar{z}}\tilde g_{\eta}(z_2)F_4(x_1, z_2) \, \dd x_1\,\dd^2 z_2\,,
\end{align}
where we used the fact that $F_4$ is analytic in $A \times B$. Since $F_4 \prec 1$ in $A \times B$,
\begin{equation*}
 \int_{D_a^1} \big|\partial_{\bar z} \tilde f_{\eta}(z_1)\big| \, \dd^2 z_1
= \frac{1}{2} \int_{D_a^1} \big|y_1^kf_{\eta}^{(k+1)}(x_1)\big| \, \dd^2 z_1 \prec \frac{1}{(N\eta)^{(k+1)/2}}=\OO(N^{-3})\,,
\end{equation*}
and similarly $\int_{D_a^2} \big|\partial_{\bar z} \tilde g_{\eta}(z_2)\big|\, \dd z_2 \prec N^{-3}$, then the last two terms in \eqref{gaoqi} are of order $O_{\prec}(N^{-3})$.
Then we have
\begin{equation} \label{518}
\eqref{515}
=\frac{1}{4\pi^2}
\int_{-\infty}^{+\infty}\int_{-\infty}^{+\infty} f_{\eta}(x_1) g_{\eta}(x_2)F_4(x_1, x_2) \, \dd x_1\,\dd x_2+\cal E_1\,.
\end{equation}
Similarly, by Proposition \ref{prop6.1} below one can derive 
\begin{multline} \label{521}
-\frac{1}{4\pi^2}\int_{\{\im z_2=a\}}\int_{\{ \im z_1=a\}}\tilde f_{\eta}(z_1)\tilde g_{\eta}(z_2)\bb E \langle \ul{G}(z_1)\rangle  \langle \ul {G}(z_2)\rangle \, \dd z_1\,\dd z_2\\
=-\frac{1}{4\pi^2}
\int_{-\infty}^{+\infty}\int_{-\infty}^{+\infty} f_{\eta}(x_1) g_{\eta}(x_2)F_5(x_1, x_2) \, \dd x_1\,\dd x_2+\cal E_1\,,
\end{multline}
where 
\begin{equation*}
F_5(x_1,x_2)\deq\frac{1}{N^2}\Big(f_{2-}(x_1^{+},x_2^{+})+f_3(x_1^{+},x_2^{+})\sum_{i,j}\cal C_4(H_{ij})+f_4(x_1^{+},x_2^{+})\sum_i\cal C_3(H_{ii})\Big)\,,
\end{equation*}
and the definitions of $f_{2-},f_3,f_4$ are given in \eqref{488} -- \eqref{f34}.

Plugging \eqref{518}-\eqref{521} and their complex conjugates into \eqref{510} shows
\begin{multline*} 
\int p_{E}(u,v) f_{-}(u)g_{+}(v) \,\mathrm{d}u\,\mathrm{d}v=\frac{1}{N^2}\bb E \big\langle \tr f_{\eta}(H) \big \rangle \big\langle \tr g_{\eta}(H) \big \rangle\\
=\frac{1}{4\pi^2}\int f_{\eta}(x_1)g_{\eta}(x_2) \Bigg(-\frac{4}{N^2(x_1-x_2)^2}+\frac{24}{N^4(x_1-x_2)^4\kappa^2}\\
+\frac{1}{N^2}\bigg(g_1(x_1,x_2)+g_2(x_1,x_2)\sum_{i,j}\cal C_4(H_{ij})
+g_3(x_1,x_2)\sum_i\cal C_3(H_{ii})\bigg)\Bigg) \, \dd x_1 \, \dd x_2
+\cal E_1\,,
\end{multline*}
where $g_1,...,g_3$ are defined as in \eqref{g_1} -- \eqref{g_3}. Changes of variables $u=N\varrho_E(x_1-E)$ and $v=N\varrho_E(x_2-E)$ give
\begin{equation} \label{explain2}
f_{\eta}(x_1)g_{\eta}(x_2)\, \dd x_1 \, \dd x_2=\varrho_E^{-2}f_{-}(u)g_{+}(v)\, \dd u \, \dd v \ \ \ \mbox{and}\ \ \ N(x_1-x_2)=\varrho_E^{-1}(u-v)\,,
\end{equation}
and together with $\kappa=2\pi\varrho_E$ we have the desired result.
\end{proof}

\section{Estimates on the Green function: proof of Proposition \ref{prop4.3}} \label{sec5} 
In this section we prove Proposition \ref{prop4.3} starting from the basic estimate of Lemma \ref{lem:4.2}. 

Recall $\alpha \deq -\log_N \eta \in [0,1-\tau]$. Throughout this section we use the quantity \begin{equation} \label{chi}
\chi\equiv \chi(\alpha)\deq 1/2\min\{\alpha,1-\alpha\}\,.
\end{equation}
Below (e.g.\ \eqref{reduce} and \eqref{442}), we shall see that each time we perform a cumulant expansion, we obtain an improvement of order $N^{-\chi}$. In order to simplify our argument we would like to have $\chi\ge c>0$. For this reason, in the remainder of this section we replace $\alpha$ by $\alpha'\deq (\alpha+1)/2 \in [1/2,1-\tau/2]$, and we have $\chi(\alpha')\geq\tau/4>0$. The replacement may be done without loss of generality, since $b(P)$ and $b_*(P)$ in Proposition \ref{prop4.3} are independent of $\alpha$, and $b_0(P)$ in \eqref{3.14} (as well as $b_1(P)$ in Lemma \ref{lem1} below) is an increasing function of $\alpha$, which allows us to use a larger $\alpha$.

In Sections \ref{sec:4.1} -- \ref{sec4.3} we focus on the case when $H$ is real symmetric, and we shall remark on the complex Hermitian case in Section \ref{sec4.4}.

\subsection{The first step} \label{sec:4.1} In order to illustrate our method, we first prove the following simpler estimate.

\begin{lemma} \label{lem1}
Under the conditions of Theorem \ref{thm_resolvent}, we have the bound
\begin{equation*}
\sum_{i_1,\ldots,i_n}a_{i_1,\ldots,i_n}P_{i_1,\ldots,i_n}  \prec N^{t+n+b_1(P)}
\end{equation*}
where $b_1(P)\deq \alpha\,{\nu}_2(P)-(1-\alpha)\nu_4(P)-(1-\alpha)\nu_5(P)/4$. Moreover, when $\nu_2(P)=\nu_4(P)=\nu_5(P)=0$, the bound is $\OO(N^{t+n})$ instead.
\end{lemma}

The above is a first improvement of Lemma \ref{lem:4.2}. Comparing to Lemma \ref{lem:4.2}, Lemma \ref{lem1} roughly states that in the estimates of $P$, we can bound the term $\bb E \ul{A^m}$ by $\OO(1)$ instead of $\OO_{\prec}(N^{\alpha(m-1)})$ for all $A \in \{G,G^*,F,F^*\}$. This motivates us to prove the following result, which trivially fills the gap between Lemmas \ref{lem:4.2} and \ref{lem1}. Indeed, Lemma \ref{lem:4.2} and Proposition \ref{lem4.4} immediately imply Lemma \ref{lem1}, by the previous observation.

\begin{proposition} \label{lem4.4}
Under the conditions of Theorem \ref{thm_resolvent}, we have
\begin{equation*}
\bb E \ul{A^m}=\OO_{m}(1)
\end{equation*}
for $A \in \{G,G^*,F,F^*\}$ and fixed $m \in \bb N_{+}$. 
\end{proposition}

The rest of Section \ref{sec:4.1} is devoted in proving Proposition \ref{lem4.4}. It suffices to consider the case $A=G$. Note that $\bb E\ul{G^m} \in \cal P(\{G\})$, and we would like to use lemma \ref{lem:cumulant_expansion} to get a Schwinger-Dyson equation for terms in $\cal P(\{G\})$.  

\begin{lemma} \label{lem5.3}
Let $P \in P^{(n,t)}(\{G\})$ for some $n \in \bb N$ and $t \in \bb R$. Let $(a_{i_1,\ldots,i_n})_{i_1,\ldots,i_n}$ be a family of complex numbers that is uniformly bounded in $i_1,\ldots,i_n$. Consider the term
\begin{equation}  \label{5.4}
\sum_{i_1,\ldots\,i_{n}} a_{i_1,\ldots,i_{n}}P_{i_1,\ldots,i_{n}} \,.
\end{equation} 
Suppose $\nu_1(P)+\nu_2(P)\ge 1$, then for any $D>0$, \eqref{5.4} equals to a finite (depends on $D$) sum of terms in the form 
\begin{equation} \label{5.5}
\sum_{i_1,\ldots\,i_{n'}} a'_{i_1,\ldots,i_{n'}}P'_{i_1,\ldots,i_{n'}}
\end{equation}
with an error $\OO(N^{-D})$, where $P'\in \cal P^{(n',t')}(\{G\})$, and $a'_{i_1,\ldots,i_{n'}}$ is a family of complex number uniformly bounded in $i_1,\ldots,i_{n'}$. For each $P'$ appears in the sum, we either have
\begin{equation}  \label{reduce}
t'+n'+b_0(P') \le t+n+b_0(P)-\chi\,,
\end{equation}	
or
\begin{equation}  \label{5.7}
t'+n'+b_0(P')= t+n+b_0(P)\mbox{ and }\ \nu_3(P')= \nu_3(P)+1\,.
\end{equation}
Moreover, each $P'$ satisfies 
\begin{equation} \label{4.10}
t'+n'+\nu_1(P')+\nu_2(P')\le t+n+\nu_1(P)+\nu_2(P)\ \mbox{ and }\ t'+n' \le t+n\,.
\end{equation}
\end{lemma}
\begin{proof} Note that $P \in \cal P(\{G\})$ is a monomial in the variables $\bb E X$, where $\bb E X$ is either $\bb E \ul{G^m}$, or
\begin{equation} \label{QG}
\bb E Q_{r,s} \deq \bb E (G^{\sigma_1})_{x_1  y_1}\cdots(G^{\sigma_r})_{x_r y_r}\langle \ul{G^{\delta_1}}\rangle \cdots \langle \ul{G^{\delta_s}}\rangle
\end{equation} 
for $(r,s) \in \bb N^2\backslash\{(0,0)\}$, and $\sigma_1,\ldots,\sigma_r,\delta_1,\ldots,\delta_s \in \bb N_{+}$. Since $\nu_1(P) +\nu_2(P) \ge 1$, in $P$ there is either a factor $\bb E \ul{G^m}$ for $m \ge 2$, or $\bb E Q_{r,s}$ for $\max\{\sigma_1,\dots,\sigma_r,\delta_1,\ldots,\delta_s\}\ge 2$. 

\emph{Case 1.} Suppose there is a factor $\bb E \ul{G^m}$ in $P$ for some $m \ge 2$.	We have
\begin{equation} \label{4.444}
\sum_{i_1,\ldots\,i_{n}} a_{i_1,\ldots,i_{n}}P_{i_1,\ldots,i_{n}}=\sum_{i_1,\ldots\,i_{n}} a_{i_1,\ldots,i_{n}}\cdot P_{i_1,\ldots,i_{n}}/\bb E\ul{G^m} \cdot\bb E \ul{G^m}
\end{equation}
It is easy to see that by the trivial bound $\norm{G}\le \eta^{-1}=N^{\alpha}$, we have
\begin{equation*}
\sum_{i_1,\ldots\,i_{n}} \abs{a_{i_1,\ldots,i_{n}}P_{i_1,\ldots,i_{n}}/\bb E\ul{G^m}}=\OO(N^{D'})
\end{equation*}
for some fixed $D' \equiv D'(t,P, m)>0$. By the resolvent identity
\begin{equation*}
z_1\bb E \ul{G^m}=-\bb E \ul{G^{m-1}}+\frac{1}{N}\sum_{i,j}\bb E {(G^m)_{ij}H_{ji}}
\end{equation*}
and Lemma \ref{lem:cumulant_expansion}, we see that for some fixed integer $l \in \bb N$, 
$$
z\bb E \ul{G^m}=-\bb E \ul{G^{m-1}}+\frac{1}{N^2}\sum_{a=1}^{m}\sum_{i,j}\big(-\bb E(G^a)_{ii}(G^{m+1-a})_{jj}-\bb E(G^a)_{ij}(G^{m+1-a})_{ij}\big)+\sum_{k=2}^{l}X_{k}+\OO(N^{-D-D'})\,,
$$
where 
\begin{equation} \label{Xk}
X_k\deq \frac{1}{N} \sum_{i,j}\frac{1}{k!}\cal C_{k+1}(H_{ji}) \bb E \frac{\partial^k (G^m)_{ij}}{\partial H^k_{ji}}\,,
\end{equation}
Note that for $a \in \{1,2,...,m\}$, we have
\begin{align*}
\sum_{i,j}\big(-\bb E(G^a)_{ii}(G^{m+1-a})_{jj}-\bb E(G^a)_{ij}(G^{m+1-a})_{ij}\big)=-\bb E \langle \ul{G^a}\rangle \langle\ul{G^{m+1-a}}\rangle-\bb E \ul{G^a}\,\bb E\ul{G^{m+1-a}}-\bb E \ul{G^{m+1}}\,,
\end{align*}
thus
\begin{equation} \label{Gm}
\bb E \ul{G^m}=\frac{1}{T}\bigg(\bb E \ul{G^{m-1}}+\sum_{a=1}^{m}\bb E \langle \ul{G^a}\rangle \langle\ul{G^{m+1-a}}\rangle+\sum_{a=2}^{m-1}\bb E \ul{G^a}\,\bb E\ul{G^{m+1-a}}+\frac{m}{N}\,\bb E \ul{G^{m+1}}-\sum_{k=2}^{L}X_{k}\bigg)+\OO(N^{-D-D'})\,,
\end{equation}
where we recall from Section \ref{sec4.1} that $T\deq -z_1-2\bb E \ul{G}$. 
By plugging \eqref{Gm} into \eqref{4.444}, we rewrite the left-hand side of \eqref{4.444} into a finite sum of terms in the form \eqref{5.5} with an error $\OO(N^{-D})$.
We see that for $2 \le a \le m-1$, the term $P/\bb E\ul{G^m}\cdot T^{-1}\bb E \ul{G^a}\,\bb E\ul{G^{m+1-a}}$ satisfies \eqref{5.7}. Moreover, we see that all other second cumulant terms we get satisfy \eqref{reduce}, and this leaves us with the estimate of higher cumulant terms. By Lemma \ref{lem} and the differential rule \eqref{3.15} we see that for $k \ge 2$, $T^{-1}X_k$ is a finite sum of the terms in the form
\begin{equation} \label{4.21}
\sum_{i,j}a_{i,j}P'_{i,j}\,,
\end{equation}
where $P' \in \cal P^{(2,-(k+3)/2)}(\cal \{G\})$. By \eqref{4.21} and exploring the differential in \eqref{Xk} carefully, we see that for $2 \le k \le 4$, 
\begin{equation*} 
\sum_{i_1,\ldots\,i_{n}} a_{i_1,\ldots,i_{n}}\cdot P_{i_1,\ldots,i_{n}}/\bb E\ul{G^m} \cdot T^{-1}X_k
\end{equation*}
is a finite sum of terms in the form \eqref{5.5}, and each of them satisfies \eqref{reduce}. For $k\ge 5$, we would like to apply Lemma \ref{useful} to compare it to the case $k=3$. More precisely, for $k \ge 5$, let 
\begin{equation} \label{5.24}
\sum_{i_1,\ldots\,i_{n'}} a'_{i_1,\ldots,i_{n'}}P'_{i_1,\ldots,i_{n'}}
\end{equation}
be a term in $\sum_{i_1,\ldots\,i_{n}} a_{i_1,\ldots,i_{n}}\cdot P_{i_1,\ldots,i_{n}}/\bb E\ul{G^m} \cdot T^{-1}X_k$. By the definition of $X_k$ in \eqref{Xk}, we see that there is a term $ \sum_{i,j} \tilde{a}_{ij} \bb E X$ in $T^{-1} X_3$ for some $X \in \cal M^2(\cal A)$ such that \eqref{5.24} appears in the sum 
\begin{equation*} 
\sum_{i_1,\ldots\,i_{n},i,j} a_{i_1,\ldots,i_{n}}\cdot P_{i_1,\ldots,i_{n}}/\bb E\ul{G^m} \cdot  \frac{3!\,C_{k+1}(H_{ij})}{k!\,C_4(H_{ij})} \,\tilde{a}_{ij}\, \bb E \frac{\partial^{k-3}X}{\partial H_{ij}^{k-3}}.
\end{equation*}
Now we compare the above and
\begin{equation*}
\sum_{i_1,\ldots\,i_{n},i,j} a_{i_1,\ldots,i_{n}}\cdot P_{i_1,\ldots,i_{n}}/\bb E\ul{G^m} \cdot \tilde{a}_{ij} \bb E X \eqd  \sum_{i_1,\ldots\,i_{n},i,j} a_{i_1,\ldots,i_{n}}\cdot \tilde{a}_{ij} \cdot \tilde{P}_{i_1,...,i_n,i,j}
\end{equation*}
using Lemma \ref{useful}. By Lemma \ref{lem}(i)  and $k \ge 5$ we see that $t'\leq t-1$. By setting $\lambda=\alpha$ in \eqref{5.3} and using \eqref{5.2}, we see that $b_0(P') \leq b_0(\tilde{P})+(1-\alpha)$. Thus $t'+b_0(P') \le \tilde{t}+ b_0(\tilde{P})$. By observing $n'=\tilde{n}=n+2$ and using our result of $\tilde{P}$, we see that $t'+n'+b_0(P') \le \tilde{t}+ \tilde{n}+b_0(\tilde{P}) \le t+n+b_0(P)-\chi$, which means $P'$ satisfies \eqref{reduce}. Thus we have shown that each $P'$ satisfies either \eqref{reduce} or \eqref{5.7}.

Similarly, one can show that, when plugging \eqref{Gm} into \eqref{4.444}, all the terms we get satisfy \eqref{4.10}. 

\emph{Case 2.} Suppose there is a factor $\bb E Q_{r,s}$ in $P$ for some $s \ge 1$. We have
\begin{equation} \label{4.4444}
\sum_{i_1,\ldots\,i_{n}} a_{i_1,\ldots,i_{n}}P_{i_1,\ldots,i_{n}}=\sum_{i_1,\ldots\,i_{n}} a_{i_1,\ldots,i_{n}}\cdot P_{i_1,\ldots,i_{n}}/\bb EQ_{r,s} \cdot\bb EQ_{r,s}\,,
\end{equation}
and
\begin{equation*}
\sum_{i_1,\ldots\,i_{n}} \abs{a_{i_1,\ldots,i_{n}}P_{i_1,\ldots,i_{n}}/\bb E Q_{r,s}}=\OO(N^{D'})
\end{equation*}
for some fixed $D'\deq D'(P,\bb E Q_{r,s})>0$. Let us abbreviate $Q_{r,s-1} \deq Q_{r,s}/\langle \ul{G^{\delta_s}}\rangle$, $Q_{r,s-2}^{(p)} \deq Q_{r,s-1}/\langle \ul{G^{\delta_p}}\rangle$ for $1\le p \le s-1$, and  $Q_{r-1,s-1}^{(q)} \deq Q_{r,s-1}/(G^{\sigma_q})_{x_q y_q}$ for $1 \le q \le r$. We have
\begin{equation} \label{4.17}
z_1\bb E Q_{r,s}=-\bb E Q_{r,s-1} \langle \ul{G^{\delta_s-1}} \rangle +\bb E \langle Q_{r,s-1} \rangle \ul{G^{\delta_s}H}=-\bb E Q_{r,s-1} \langle \ul{G^{\delta_s-1}} \rangle+\frac{1}{N}\sum_{i,j}\bb E \langle Q_{r,s-1} \rangle (G^{\delta_s})_{ij}H_{ji}\,,
\end{equation}
Analogously to \eqref{4.11}, by calculating the last sum in \eqref{4.17} using formula \eqref{eq:cumulant_expansion}, we have
\begin{equation} \label{Qrs}
\begin{aligned}
\bb E Q_{r,s}&=\frac{1}{T} \bigg(\bb E Q_{r,s-1} \langle \ul{G^{\delta_s-1}} \rangle+
\sum_{a=0}^{\delta_s-1} \bb E Q_{r,s-1} \langle \ul{G^{\delta_s-a}} \rangle \langle \ul{G^{a+1}} \rangle-\sum_{a=0}^{\delta_s-1} \bb E Q_{r,s-1} \bb E \langle \ul{G^{\delta_s-a}} \rangle \langle \ul{G^{a+1}} \rangle\\
&\ \ +\frac{\delta_s}{N}\bb E Q_{r,s-1} \langle \ul{G^{\delta_s+1}} \rangle + 2\sum_{a=1}^{\delta_s-1}Q_{r,s-1}\langle \ul{G^{\delta_s-a}}\rangle \bb E \ul{G^{a+1}}+\frac{2}{N^2}\sum_{p=1}^{s-1}\delta_p\bb E Q_{r,s-2}^{(p)}\bb E \ul{G^{\delta_p+\delta_s+1}}\\
&\ \ +\frac{2}{N^2}\sum_{p=1}^{s-1}\delta_p\bb E Q_{r,s-2}^{(p)}\langle \ul{G^{\delta_p+\delta_s+1}}\rangle+\frac{2}{N^2}\sum_{q=1}^{r}\sigma_q\bb E Q_{r-1,s-1}^{(q)}G_{x_q y_q}^{\sigma_q+\delta_s+1}-\sum_{k=2}^lY_k\bigg)+\OO(N^{-D-D'})\,,
\end{aligned}
\end{equation}
where
\begin{equation} \label{yk}
Y_k\deq \frac{1}{N} \sum_{i,j}\frac{1}{k!}\cal C_{k+1}(H_{ji}) \bb E \partial^k_{ij}(( Q_{r,s-1}-\bb E Q_{r,s-1})  (G^{\delta_s})_{ij})\,,
\end{equation} 
$T=-z_1-2\bb E \ul{G}$, and $l$ is some fixed integer. By plugging \eqref{Qrs} into \eqref{4.4444}, we rewrite the left-hand side of \eqref{4.4444} into a finite sum of terms in the form \eqref{5.5} with an error $\OO(N^{-D})$.
We see that for $1\le a\le \delta_s-1$ and $1 \le p \le s-1$,
\begin{equation*}
\sum_{i_1,\ldots\,i_{n}} a_{i_1,\ldots,i_{n}}\cdot P_{i_1,\ldots,i_{n}}/\bb EQ_{r,s} \cdot T^{-1}\bb E Q_{r,s-1}\langle\ul{G^{\delta_s-a}}\rangle \bb E\ul{G^{a+1}}
\end{equation*} 
and
\begin{equation*}
\sum_{i_1,\ldots\,i_{n}} a_{i_1,\ldots,i_{n}}\cdot P_{i_1,\ldots,i_{n}}/\bb EQ_{r,s} \cdot T^{-1}N^{-2}\delta_p\bb E Q_{r,s-2}^{(p)}\bb E \ul{G^{\delta_p+\delta_s+1}}
\end{equation*}
satisfy \eqref{5.7}.  Moreover, we see that all other second cumulant terms we get satisfy \eqref{reduce}, and this leaves us with the estimate of higher cumulant terms. By exploring the differential in \eqref{yk} carefully and using Lemma \ref{useful}, we see that for $2 \le k \le 3$, 
\begin{equation} \label{532}
\sum_{i_1,\ldots\,i_{n}} a_{i_1,\ldots,i_{n}}\cdot P_{i_1,\ldots,i_{n}}/\bb EQ_{r,s} \cdot T^{-1}Y_k
\end{equation}
is a finite sum of terms in the form \eqref{5.5}, and each of them satisfies \eqref{reduce}. Note that for $k=3$, each $P'$ in \eqref{532}  has two $i$ and $j$ as summation indices. For $k \ge 4$, we can apply Lemma \ref{useful} to compare it to the case $k=3$, and similar as in Case 1, we can show that $\sum_{i_1,\ldots\,i_{n}} a_{i_1,\ldots,i_{n}}\cdot P_{i_1,\ldots,i_{n}}/\bb EQ_{r,s} \cdot T^{-1}Y_k$
is a finite sum of terms in the form \eqref{5.5}, and each of them satisfies \eqref{reduce}. Similarly, one can show that all the terms we get satisfy \eqref{4.10}. This completes the proof for $s \ge 1$.

\emph{Case 3.} To deal with the special case $s=0$, we denote
\begin{equation} \label{4.23}
\bb EQ_r \deq \bb EQ_{r,0} = \bb E (G^{\sigma_1})_{x_1 y_1}\cdots(G^{\sigma_r})_{x_r y_r}\,.
\end{equation}
Suppose there is a term $\bb E Q_r$ in $P$, where $\max\{\sigma_1,\ldots,\sigma_r\} \ge 2$, and assume $\sigma_r=\max \sigma_q$. We have
\begin{equation*} 
\sum_{i_1,\ldots\,i_{n}} a_{i_1,\ldots,i_{n}}P_{i_1,\ldots,i_{n}}=\sum_{i_1,\ldots\,i_{n}} a_{i_1,\ldots,i_{n}}\cdot P_{i_1,\ldots,i_{n}}/\bb EQ_{r} \cdot\bb EQ_{r}\,,
\end{equation*}
and
\begin{equation*}
\sum_{i_1,\ldots\,i_{n}} \abs{a_{i_1,\ldots,i_{n}}P_{i_1,\ldots,i_{n}}/ Q_{r}}=\OO(N^{D'})
\end{equation*}
for some fixed $D'\deq D'(P, Q_{r})>0$. Let us denote $Q_{r-1}\deq Q_r/(G^{\sigma_r})_{x_r y_r}$, and $Q_{r-1}^{(q)}\deq Q_{r-1}/(G^{\sigma_q})_{x_q y_q}$ for $1\le q\le r-1$. As in Case 2, by formula \eqref{eq:cumulant_expansion} we have for some fixed integer $l$ that 
\begin{equation} \label{Qr}
\begin{aligned}
\bb E Q_{r}&=\frac{1}{U}\bigg(\bb E Q_{r-1}(G^{\sigma_r-1})_{x_r y_r}+\sum_{a=0}^{\sigma_r-1}\bb E Q_{r-1}(G^{\sigma_r-p})_{x_r y_r}\langle \ul{G^{a+1}} \rangle +\sum_{a=1}^{\sigma_r-1}\bb E Q_{r-1}(G^{\sigma_r-p})_{ x_r y_r}\bb E \ul{G^{a+1}}  \\
&\ \  +\frac{1}{N}\sum_{q=1}^{r-1}(\sigma_r+1)\sum_{a=0}^{\sigma_q-1}\bb E Q_{r-1}^{(q)}\Big((G^{\sigma_q-a})_{ x_q  y _r}(G^{\sigma_r+a+1})_{ x_r y_q}+(G^{\sigma_q-a})_{x_q x_r}(G^{\sigma_r+a+1})_{y_r y_q}\Big)\\
&\ \ +\frac{\sigma_r+1}{N} \bb E Q_{r-1}(G^{\sigma_r+1})_{x_r y_r}-\sum_{k=2}^{l}Z_k\bigg)+\OO(N^{-D-D'})\,,
\end{aligned}
\end{equation}
where $U=-z_1-\bb E\ul{G}$, and
\begin{equation*} 
Z_k\deq \frac{1}{N} \sum_{i,j}\frac{1}{k!}\cal C_{k+1}(H_{ji}) \bb E \partial^k_{ij}( Q_{r-1}  (G^{\sigma_r})_{ij})\,.
\end{equation*}
One can then proceed the proof similarly as in Case 2. We omit the details.
\end{proof}

Equipping with Lemma \ref{lem5.3}, we are now ready to prove Proposition \ref{lem4.4}. It suffices to show the case where $A=G$ and $m \geq 2$.

The proof relies on a tree structure that we now introduce. As in Section \ref{sec3}, when dealing with polynomials in the Green function, we distinguish between formal polynomials, which are algebraic expressions, and their values, which are expectations of random variables. The tree has a vertex set $\bb V$. Every vertex $v \in \bb V$ is labelled by a formal expression of the form
 \[
 \sum_{i_1,\ldots, i_n}a^{(v)}_{i_1,...,i_n}P^{(v)}_{i_1,...,i_n}\,,
 \] 
 where $P^{(v)} \in \cal P^{(n,t)}(\{G\}) \subset \cal P(\{G\})$ for some $n=n(v) \in \bb N$, $t=t(v) \in \bb R$, and $(a^{(v)}_{i_1\ldots i_n})_{1\le i_1,\ldots,i_n\le N}$ is a family of complex numbers that is uniformly bounded in $i_1,\ldots,i_n$. We denote $b_0(v)\deq b_0(P^{(v)})$, and $\nu_i(v)\deq \nu_i(P^{(v)})$ for $i=1,2,...,6$. Each vertex $v$ has a \emph{value}, a deterministic number obtained by formally evaluating all expectations and products in its definition. By a slight abuse of notation, as in Section \ref{sec3}, we use the same notation both for the formal monomial and its value.
 
The \emph{root} $r \in \bb V$ of the tree is given by
\begin{equation*}
r\deq \bb E \ul{G^m}\,,
\end{equation*} 
where $\bb E \ul{G^m} \in \cal P^{(n,t)}(\{G\})$ for $(n,t)=(0,0)$. The first step is to show that for each vertex $v$, we can assign a finite number of its \textit{descendants}, a set of vertices denoted by $\cal D(v)$, such that
\begin{equation} \label{yahaha}
v = \sum_{u \in \cal D(v)} u +\OO(1)\,.
\end{equation}
Here $|D(v)|\leq C(v)<+\infty$. Every vertex $v \in \bb V$ has a \emph{generation} $k(v)$, which is defined recursively by $k(r) = 0$ and $k(u) = k(v) + 1$ for all $u \in \cal D(v)$. In a second step, we show that $\nu_3(v)\leq d$ for some fixed $d(m,\alpha)<\infty$. In a third step, we use the result in step 2 to show that this tree has bounded depth, and the proof easily follows from the relation \eqref{yahaha}. The latter step means that after a bounded number of steps we always reach a vertex $v$ that satisfies $v = \OO(1)$. It is declared a leaf and has no descendants.

\emph{Step 1.}
Since $m \ge 2$, we have $\nu_1(r)+\nu_2(r) = m-1\ge 1$. Thus by setting $D=0$ in Lemma \ref{lem5.3}, we can write
\begin{equation} \label{tree}
r= \sum_{u \in \cal D(r)} u +\OO(1)\,,
\end{equation}
where each $u \in \cal D(r)$ lies in $\bb V$, and $|\cal D(r)|= \OO_u(1)<+\infty$. We also observe from Lemma \ref{lem5.3} that for each $u\in \cal D(r)$, 
\begin{equation*}
t(u)+n(u)+\nu_1(u)+\nu_2(u)\le t(r)+ n(r)+\nu_1(r)+\nu_2(r)=m-1\ \mbox{ and }\ t(u)+n(u) \le t(r)+n(r)=0\,.
\end{equation*} 
For a vertex $u\in D(r)$, if $\nu_1(u)=\nu_2(u)=0$, Lemma \ref{lem:4.2} will imply 
$$ 
u=\OO(N^{t(u)+n(u)+\nu_1(u)+\nu_2(u)})=\OO(N^{t(u)+n(u)})=\OO(1)\,,
$$
and we remove $u$ from the tree and classify it in the last term of \eqref{tree}. Thus every $u \in \cal D(r)$ also satisfies
\[
\nu_1(u)+\nu_2(u)\ge 1\,.
\]

In general, if we have a vertex $v$ in the tree satisfying
\begin{equation*} 
t(v)+n(v)+\nu_1(v)+\nu_2(v)\le m-1\,, \ t(v)+n(v) \le 0\,, \mbox{ and } \nu_1(v)+\nu_2(v)\ge 1\,,
\end{equation*}
the last condition enables us to apply Lemma \ref{lem5.3} again and get
\[
v=\sum_{u \in \cal D(v)} u+ \OO(1)\,,
\]
where each $u \in \cal D(v)$ lies in $\bb V$, and $|\cal D(v)|= \OO_u(1)<+\infty$. As in the root case, we have
\begin{equation}  \label{ooo}
t(u)+n(u)+\nu_1(u)+\nu_2(u)\le m-1\,, \ t(u)+n(u) \le 0\,, \mbox{ and } \nu_1(u)+\nu_2(u)\ge 1\,,
\end{equation}
for all $u \in \cal D(v)$.
By continuing this process, we create a locally finite tree with root $r$, where each vertex $u$ satisfies \eqref{ooo}.

\emph{Step 2.} Let $v$ be a vertex in our tree. From the construction we know
\begin{equation*}
t(v)+n(v)+\nu_1(v)+\nu_2(v)\le m-1\,,
\end{equation*}
thus Lemma \ref{lem:4.2} shows
\begin{equation} \label{yio}
v\prec N^{t(v)+n(v)+b_0(v)}\le N^{t(v)+n(v)+\alpha(\nu_1(v)+\nu_2(v))}\le N^{m-1-(1-\alpha)(\nu_1(v)+\nu_2(v))}\,.
\end{equation}
Since $v$ is a vertex of our tree, we must have \[
m-(1-\alpha)(\nu_1(v)+\nu_2(v))\ge0\,,
\] 
otherwise \eqref{yio} implies $v\prec N^{-1}$, which results $v=\OO(1)$. Thus $\nu_1(v)+\nu_2(v) \leq \ceil{m/(1-\alpha)}\eqd d(\alpha,m) <\infty$. Together with the observation $0 \leq \nu_3\leq\nu_1+\nu_2$, we have
\begin{equation} \label{vvv}
0\leq\nu_3(v) \leq d
\end{equation}
for all vertices $v$ in the tree.

\emph{Step 3.} We show the tree has a bounded depth, which easily concludes the proof of Proposition \ref{lem4.4}. Suppose conversely that there is a infinite sequence of vertices $(v_0,v_1,...)$ in the tree, such that $v_0=r$, and $v_{k+1}$ is a descendent of $v_k$ for all $k\in \bb N$. Lemma \ref{lem5.3} (in particular \eqref{reduce} and \eqref{5.7}) shows that for any $v$ and its descendant $u$ in this sequence, we have either
\begin{equation} \label{4.42}
t(u)+n(u)+b_0(u)\le t(v)+n(v)+b_0(v)-\chi\,,
\end{equation} 
or \begin{equation} \label{443}
t(u)+n(u)+b_0(u)= t(v)+n(v)+b_0(v)\ \mbox{ and }\ \nu_3(u)= \nu_3(v)+1\,.
\end{equation}
By \eqref{vvv} we see that in the sequence $(v_0,v_1,...)$, the case \eqref{443} can at most happen $d$ times in a row (otherwise we will have $\nu_3(v) \ge d+1$ for some vertex $v$). This means \eqref{4.42} will happen at least once in $d+1$ generations. Thus \eqref{4.42} and the first relation in \eqref{443} imply
\begin{equation} \label{kkk}
t(v)+n(v)+b_0(v) \leq t(r)+n(r)-1=-1
\end{equation}
whenever $k(v) \ge \ceil{(d+1)(b_0(r)+1)/\chi}=\ceil{(d+1)(\alpha(m-1)+1)/\chi} $, and we can conclude from \eqref{kkk} and Lemma \ref{lem:4.2} that
\begin{equation*}
v\prec N^{t(v)+n(v)+b_0(v)}\le N^{-1}\,,
\end{equation*}
which implies $v=\OO(1)$. Thus this tree has generation at most $\ceil{(d(\alpha,m)+1)(\alpha(m-1)+1)/\chi}$, and this completes the proof.

\subsection{Single matrix case} \label{sec4.2} In this section we prove Proposition \ref{prop4.3}(i) with the aid of Lemma \ref{lem1}. It suffices to show the case when $\cal A =\{G\}$. As in Section \ref{sec:4.1}, we would like to use Lemma \ref{lem:cumulant_expansion} to construct a tree whose root is the left-hand side of \eqref{4.39}. The off-spring process is summarized in the following Lemma. 
\begin{lemma} \label{lem4.6666}
Let $P \in P^{(n,t)}(\{G\})$ for some $n \in \bb N$ and $t \in \bb R$. Let $(a_{i_1,\ldots,i_n})_{i_1,\ldots,i_n}$ be a family of complex numbers that is uniformly bounded in $i_1,\ldots,i_n$. Consider the term
\begin{equation} \label{4.40}
\sum_{i_1,\ldots\,i_{n}} a_{i_1,\ldots,i_{n}}P_{i_1,\ldots,i_{n}} \,.
\end{equation} 
Suppose $(\nu_2(P),\nu_4(P),\nu_5(P),\nu_6(P)\ne (0,0,0,0))$, then for any $D>0$, \eqref{4.40} equals to a finite (depends on $D$) sum of terms in the form 
\begin{equation} \label{4.41}
\sum_{i_1,\ldots\,i_{n'}} a'_{i_1,\ldots,i_{n'}}P'_{i_1,\ldots,i_{n'}}
\end{equation}
with an error $\OO(N^{-D})$, where $P'\in \cal P^{(n',t')}(\{G\})$, and $a'_{i_1,\ldots,i_{n'}}$ is a family of complex number uniformly bounded in $i_1,\ldots,i_n'$. For each $P'$ appears in the sum, we have
\begin{equation}  \label{442}
t'+n'+b_1(P') \le t+n+b_1(P)-\chi\,,
\end{equation}	
and 
\begin{equation}  \label{IG.Kid}
t'+n'+b(P')\le t+n+b(P)\,,
\end{equation}
where $\chi$ is defined in \eqref{chi}.
\end{lemma}
\begin{proof}
\emph{Case 1.} Suppose $\nu_4(P) \ne 0$, then there is a factor $\bb E Q_{r,s}$ in $P$ for $s \ge 1$, where $\bb EQ_{r,s}$ is defined as in \eqref{QG}. We have
\begin{equation} \label{4.44}
\sum_{i_1,\ldots\,i_{n}} a_{i_1,\ldots,i_{n}}P_{i_1,\ldots,i_{n}}=\sum_{i_1,\ldots\,i_{n}} a_{i_1,\ldots,i_{n}}\cdot P_{i_1,\ldots,i_{n}}/\bb EQ_{r,s} \cdot\bb E Q_{r,s}
\end{equation}
and
\begin{equation*}
\sum_{i_1,\ldots\,i_{n}} \abs{a_{i_1,\ldots,i_{n}}P_{i_1,\ldots,i_{n}}/\bb EQ_{r,s}}=\OO(N^{D'})
\end{equation*}
for some fixed $D'\equiv D'(P,r,s)>0$. Recall that we have already obtained the expansion of $\bb E Q_{r,s}$ in \eqref{Qrs}. By plugging \eqref{Qrs} into \eqref{4.44}, we rewrite \eqref{4.40} into a finite sum of terms in the form \eqref{4.41} with an error $\OO(N^{-D})$. Note that we can now use the improved estimate Lemma \ref{lem1}. By Lemmas \ref{useful} and \ref{lem1}, we can show that, for each $P'$ we get, both \eqref{442} and \eqref{IG.Kid} are satisfied. This completes the proof for $\nu_4(P) \ne 0$.

\emph{Case 2.} Suppose $\nu_4(P) = 0$, then there is a factor $\bb EQ_{r}$ in $P$, where $\bb EQ_r$ is defined as in \eqref{4.23}. If $\nu_2(P)\ne 0$, then $\max\limits_{1\le q \le r}\sigma_q \ge 2$, and we can W.L.O.G. assume $\sigma_r\ge 2$. If $\nu_2=0$, then $(\nu_5 ,\nu_6)\ne (0,0) $, and we can assume $(G^{\sigma_r})_{x_r y_r}=G_{x_r y_r}$ for some $x_r \ne y_r$. In both situations, we have
\begin{equation} \label{4.47}
\sum_{i_1,\ldots\,i_{n}} a_{i_1,\ldots,i_{n}}P_{i_1,\ldots,i_{n}}=\sum_{i_1,\ldots\,i_{n}} a_{i_1,\ldots,i_{n}}\cdot P_{i_1,\ldots,i_{n}}/\bb EQ_{r} \cdot\bb E Q_{r}
\end{equation}
and
\begin{equation*}
\sum_{i_1,\ldots\,i_{n}} \abs{a_{i_1,\ldots,i_{n}}P_{i_1,\ldots,i_{n}}/\bb EQ_{r}}=\OO(N^{D'})
\end{equation*}
for some fixed $D'\equiv D'(P,r)>0$. We now plug the expansion of $\bb EQ_r$ that we obtained in \eqref{Qr} into \eqref{4.47}, which rewrites \eqref{4.40} into a finite sum of terms in the form \eqref{4.41} with an error $\OO(N^{-D})$. Note that we either have $\sigma_r \ge 2$, or $(G^{\sigma_r-1})_{x_r y_r}=\delta_{x_r y_r}$ for $x_r \ne y_r$, and in the second case the two summations over $x_r$ and $y_r$ will become a single summation. By exploring the terms in right-hand side of \eqref{Qr} carefully and using Lemma \ref{useful}, we see that for each $P'$ we get, both \eqref{442} and \eqref{IG.Kid} are satisfied. This completes the proof for $\nu_4(P)=0$.    
\end{proof}
\begin{proof} [Proof of Proposition \ref{prop4.3}(i)] 
The proof is similar to (also simpler than) that of Proposition \ref{lem4.4}, i.e.\ we proceed by generating a rooted tree. We have a vertex set $\bb V$. Every $v \in \bb V$ is labelled formally by
\[
\sum_{i_1,\ldots, i_n}a^{(v)}_{i_1,...,i_n}P^{(v)}_{i_1,...,i_n}\,,
\] 
where $P^{(v)} \in \cal P^{(n,t)}(\{G\}) \subset \cal P(\{G\})$ for some $n=n(v) \in \bb N$, $t=t(v) \in \bb R$, and $(a^{(v)}_{i_1\ldots i_n})_{1\le i_1,\ldots,i_n\le N}$ is a family of complex numbers that is uniformly bounded in $i_1,\ldots,i_n$. We denote $b_1(v)\deq b_1(P^{(v)})$, $b(v)\deq b(P^{(v)})$, and $\nu_i(v)\deq \nu_i(P^{(v)})$ for $i=1,2,...,6$.  Each vertex $v$ has a \emph{value}, a deterministic number obtained by formally evaluating the all expectations and products in its definition. By a slight abuse of notation, as in Section \ref{sec3}, we use the symbol $v$ both for the formal monomial and its value.

We construct a tree with root 
\begin{equation*}
r\deq \sum_{i_1,\ldots,i_n} a_{i_1,\ldots,i_n}P_{i_1,\ldots,i_n} 
\end{equation*} 
given at the left-hand side of \eqref{4.39}. By Lemma \ref{lem1} it suffices to consider the situation where $(\nu_2(r),\nu_4(r),\nu_5(r),\nu_6(r))\ne (0,0,0,0)$. Then by setting $D=-t(r)-n(r)-b(r)$ in Lemma \ref{lem4.6666}, we have
\begin{equation} \label{grudo}
r=\sum_{u \in \cal D(r)} u+\OO(N^{t+n+b(r)})\,,
\end{equation}
where $\cal D(r)$ is the set of descendants of $r$, and $|\cal D(r)|=\OO_r(1)<+\infty$. We also observe from Lemma \ref{lem4.6666} that
\[
t(u)+n(u)+b_1(u)\leq t(r)+n(r)+b_1(r)-\chi \mbox{\quad and \quad} t(u)+n(u)+b(u)\leq t(r)+n(r)+b(r)
\]
for all $u \in \cal D(r)$. For each vertex $u \in \cal D(r)$, if $(\nu_2(u),\nu_4(u),\nu_5(u),\nu_6(u))= (0,0,0,0)$, then by Lemma \ref{lem1} we have
\begin{equation*}
u=\OO(N^{t(u)+r(u)})=\OO(N^{t(u)+n(u)+b(u)})\le \OO(N^{t(r)+n(r)+b((r))})\,,
\end{equation*}
and we remove $u$ from the tree and classify it in the last term of \eqref{grudo}. Thus every $u \in \cal D(r)$ also satisfies
\[
(\nu_2(u),\nu_4(u),\nu_5(u),\nu_6(u))\ne (0,0,0,0)\,.
\]
We can repeat the above process and generate a locally finite tree, where each vertex $v$ satisfies\\ $(\nu_2(u),\nu_4(u),\nu_5(u),\nu_6(u))\ne (0,0,0,0)$. For each vertex $v$ and its descendent $u$, we have
\begin{equation} \label{kass}
t(u)+n(u)+b_1(u)\leq t(r)+n(r)+b_1(r)-\chi \mbox{\quad and \quad} t(u)+n(u)+b(u)\leq t(r)+n(r)+b(r)\,.
\end{equation}
By the first relation in \eqref{kass} and Lemma \ref{lem1}, we see that after at most $\ceil {(b_1(r)-b(r)+1 )/\chi} $ generations, all the vertices are bounded by $\OO(N^{t(r)+n(r)+b(r)})$, and this completes the proof.
\end{proof}
We end this section by proving Lemma \ref{lem4.3} mentioned above. This serves as a corollary of Proposition \ref{prop4.3}(i) and Lemma \ref{useful}. 
\begin{proof}[Proof of Lemma \ref{lem4.3}]
As usual, we use the resolvent identity $zG=HG-I$ and the formula \eqref{eq:cumulant_expansion} to get
\begin{equation*}
\bb E \ul{G}=\frac{1}{U}\bigg(1+\bb E \langle \ul{G}\rangle ^2+\frac{1}{N}\,\bb E \ul{G^{2}}-\sum_{k=2}^{L}X^{(0)}_{k}\bigg)+\OO(N^{-2})\,,
\end{equation*}
where $U=-z_1-\bb E \ul{G}$, $L \in \bb N$, and
\begin{equation*} 
X^{(0)}_k\deq \frac{1}{N} \sum_{i,j}\frac{1}{k!}\cal C_{k+1}(H_{ji}) \bb E \partial^k_{ji} G_{ij}\,.
\end{equation*}
By Proposition \ref{prop4.3}(i) one easily checks that $ \bb E \langle \ul{G}\rangle ^2$ and $X^{(0)}_{2}$ are both bounded by $\OO(N^{-3/2})$. Also, using Lemma \ref{useful} we can see $\{X_{k}^{(0)}\}_{k\ge 4}$ are uniformly bounded by $\OO(N^{-2})$. Thus
\begin{equation} \label{642}
(\bb E \ul{G})^2+z_1\bb E \ul{G}+1+\frac{1}{N}\bb E \ul{G^2}-X_3^{(0)}+\OO(N^{-3/2})=0\,.
\end{equation}	
We proceed the proof by first deducing a rough value of $\bb E \ul{G}$ from \eqref{642} and then use the result to compute $N^{-1}\bb E \ul{G^2}$; the latter value can be plugged back into \eqref{642} to have a more precise value of $\bb E \ul{G}$.

By Proposition \ref{prop4.3}(i) we have $X_3^{(0)}=\OO(N^{-1})$ and $N^{-1} \bb E \ul{G^2}=\OO(N^{-1})$, and this gives
\begin{equation*}
(\bb E \ul{G})^2+z_1\bb E \ul{G}+1+\OO(N^{-1})=0
\end{equation*}
Note that $m(z_1)$ is the unique solution of
$
x^2+z_1x+1=0
$
satisfying $\im m(z_1) \im z_1>0$. Let $\tilde{m}(z_1)$ be the other solution of  $x^2+z_1x+1=0$. An application of Lemma 5.5 in \cite{BK16} gives
\begin{equation} \label{4.64}
\min\{ |\bb E  \underline{G} -m(z_1)|,\,|\bb E  \underline{G} -\tilde{m}(z_1)|\}=\frac{\OO(N^{-1})}{\sqrt{|2-|E||}} =\OO(N^{-1})\,.
\end{equation}
Since $G=(H-z_1)^{-1}$, we know that $\sgn (\im \ul G)=\sgn (\im z_1)=1$. Also, we have $\im \tilde{m}(z_1) \le -c$ for some $c=c(\re z_1)=c(E)>0$. This shows $|\bb E  \underline{G} -\tilde{m}(z_1)|\ge c$.
Thus from \eqref{4.64} we have
\begin{equation} \label{645}
|\bb E  \underline{G} -m(z_1)|=\OO(N^{-1})\,.
\end{equation}
Now we carefully apply the differentials in $X_3^{(0)}$ and estimate the results by Proposition \ref{prop4.3}(i). This gives
\begin{equation*}
X_3^{(0)}=-\frac{1}{N}\sum_{i,j}\cal C_4(H_{ij})\bb E G_{ii}^2G_{jj}^2+\OO(N^{-2})\,,
\end{equation*}
and we have
\begin{equation} \label{459}
(\bb E \ul{G})^2+z_1\bb E \ul{G}+1+\frac{1}{N}\bb E \ul{G^2}+\frac{1}{N}\sum_{i,j}\cal C_4(H_{ij})\bb E G_{ii}^2G_{jj}^2+\OO(N^{-3/2})=0 \,.
\end{equation}	
By using \eqref{Gm} for $m=2$ and $D+D''=1$ we have
\begin{equation} \label{647}
\bb E \ul{G^2}=\frac{1}{T}\Big(\bb E \ul{G}+2\bb E \langle\ul{ G}\rangle \langle \ul{G^2} \rangle +\frac{2}{N}\bb E \ul{G^3}-\sum_{k=2}^{l}X_k\Big)+\OO(N^{-1})\,,
\end{equation}
where $l \in \bb N$. By Proposition \ref{prop4.3}(i) we see that all but the first term on the right-hand side of \eqref{647} are bounded by $\OO(N^{-1})$. Together with \eqref{645} we have
\begin{equation} \label{649}
\bb E \ul{G^2}=T^{-1}\bb E \ul{G}+\OO(N^{-1})=-\frac{1}{2}+\frac{z_1}{2\sqrt{z_1^2-4}}+\OO(N^{-1})\,,
\end{equation}
and this finishes the calculation of $\bb E \ul{G^2}$. 
Again by cumulant expansion \eqref{eq:cumulant_expansion} and Proposition \ref{prop4.3}(i), we can show that
\begin{equation} \label{650}
\frac{1}{N}\sum_{i,j}\cal C_4(H_{ij})\bb E G_{ii}^2G_{jj}^2=\frac{m(z_1)^4}{N}\sum_{i,j}\cal C_4(H_{ij})+\OO(N^{-2})\,.
\end{equation}
An elementary computation using \eqref{459}, \eqref{649} and \eqref{650} shows
\begin{equation*} 
\bb E \ul{G}=m(z_1)-\frac{1}{N\sqrt{z_1^2-4}}\Big(-\frac{1}{2}+\frac{z_1}{2\sqrt{z_1^2-4}}+m(z_1)^4\sum_{i,j}\cal C_4(H_{ij})\Big)+\OO(N^{-3/2})
\end{equation*}
as desired.
\end{proof}
\subsection{Two matrix case} \label{sec4.3} In this section we prove Proposition \ref{prop4.3}(ii). It suffices to show the case when $\cal A=\{G,F^*\}$. In this section, we use the fundamental parameter
\[
\tilde\chi\equiv\tilde\chi(\alpha,\gamma)\deq \min\{\alpha-\gamma,\alpha/2,(1-\alpha)/2\}\,.
\] 
As explained at the beginning of Section \ref{sec5}, throughout the section we replace $\alpha$ by $\alpha'=(1+\alpha)/2$, and this results in $\tilde{\chi}\geq \tau/4>0$.
We still follow our routine path of generating a rooted tree. 
\begin{lemma} \label{fly the ocean in a silver plane}
Let $P \in P^{(n,t)}(\{G,F^*\})$ for some $n \in \bb N$ and $t \in \bb R$. Let $(a_{i_1,\ldots,i_n})_{i_1,\ldots,i_n}$ be a family of complex numbers that is uniformly bounded in $i_1,\ldots,i_n$. Consider the term
\begin{equation} \label{4.400}
\sum_{i_1,\ldots\,i_{n}} a_{i_1,\ldots,i_{n}}P_{i_1,\ldots,i_{n}} \,.
\end{equation} 
Suppose $(\nu_2(P),\nu_4(P),\nu_5(P),\nu_6(P))\ne (0,0,0,0)$, then for any $D>0$, \eqref{4.400} equals to a finite (depends on $D$) sum of terms in the form 
\begin{equation} \label{4.4111}
\sum_{i_1,\ldots\,i_{n'}} a'_{i_1,\ldots,i_{n'}}P'_{i_1,\ldots,i_{n'}}
\end{equation}
with an error $\OO(N^{-D})$, where $P'\in \cal P^{(n',t')}(\{G\})$, and $a'_{i_1,\ldots,i_{n'}}$ is a family of complex number uniformly bounded in $i_1,\ldots,i_n'$. For each $P'$ appears in the sum, we have
\begin{equation}  \label{4422}
t'+n'+b_1(P') \le t+n+b_1(P)-\tilde\chi\,,
\end{equation}	
and 
\begin{equation}  \label{IG.Zz1tai}
t'+n'+b_*(P')\le t+n+b_*(P)\,,
\end{equation}
where $\tilde\chi\equiv\tilde\chi(\alpha,\gamma)\deq \min\{\alpha-\gamma,\alpha/2,(1-\alpha)/2\}=\tau/4>0$\,.
\end{lemma}
By a similar argument as in Section \ref{sec4.2}, Lemma \ref{fly the ocean in a silver plane} immediately implies Proposition \ref{prop4.3}(ii). The rest of this section is devoted into proving Lemma \ref{fly the ocean in a silver plane}.

As an analogue of the single matrix case, for $(r,s) \ne (0,0)$, we introduce a new set of quantity 
\begin{equation*}  
\bb E \tilde{Q}_{r,s} \deq \bb E (\prescript{1}{G}^{\sigma_1})_{x_1 y_1}\cdots(\prescript{r}{G}^{\sigma_r})_{x_r y_r}\langle \ul{\prescript{r+1}{G}^{\delta_1}}\rangle \cdots \langle \ul{\prescript{r+s}{G}^{\delta_s}}\rangle\,,
\end{equation*} 
where $\prescript{1}{G},\ldots,\prescript{r+s}{G} \in \{G,F^*\}$.
And for $s=0$, we define 
\begin{equation*} 
\bb E\tilde{Q}_r \deq \bb E\tilde{Q}_{r,0} = \bb E (\prescript{1}{G}^{\sigma_1})_{x_1 y_1}\cdots(\prescript{r}{G}^{\sigma_r})_{x_r y_r}\,.
\end{equation*}
Let $P$ be given as in Lemma \ref{fly the ocean in a silver plane}. We split the cases according to different values of the indices $(\nu_2(P),\nu_4(P),\nu_5(P),\nu_6(P)).$

\emph{Case 1.} Suppose $\nu_4(P) \ne 0$, then there is a factor $\bb E \tilde{Q}_{r,s}$ in $P$ for $s \ge 1$. We have
\begin{equation} \label{t4.44}
\sum_{i_1,\ldots\,i_{n}} a_{i_1,\ldots,i_{n}}P_{i_1,\ldots,i_{n}}=\sum_{i_1,\ldots\,i_{n}} a_{i_1,\ldots,i_{n}}\cdot P_{i_1,\ldots,i_{n}}/\bb E\tilde Q_{r,s} \cdot\bb E \tilde Q_{r,s}
\end{equation}
and
\begin{equation*}
\sum_{i_1,\ldots\,i_{n}} \abs{a_{i_1,\ldots,i_{n}}P_{i_1,\ldots,i_{n}}/\bb EQ_{r,s}}=\OO(N^{D'})
\end{equation*}
for some fixed $D'\equiv D'(P, r,s)>0$. Let us abbreviate $\tilde{Q}_{r,s-1} \deq \tilde{Q}_{r,s}/\langle \ul{\prescript{r+s}{G}^{\delta_s}}\rangle$, $\tilde{Q}_{r,s-2}^{(p)} \deq \tilde{Q}_{r,s-1}/\langle \ul{\prescript{r+p}{G}^{\delta_p}}\rangle$ for $1\le p \le s-1$, and  $\tilde{Q}_{r-1,s-1}^{(q)} \deq \tilde{Q}_{r,s-1}/(\prescript{q}{G}^{\sigma_q})_{x_q y_q}$ for $1 \le q \le r$. As in \eqref{Qrs}, we have for some fixed integer $l$ that
\begin{multline} \label{tQrs}
\bb E \tilde{Q}_{r,s}=\frac{1}{T(\prescript{r+s}{G})} \bigg(\bb E \tilde{Q}_{r,s-1} \langle \ul{\prescript{r+s}{G}^{\delta_s-1}} \rangle+
\sum_{a=0}^{\delta_s-1} \bb E \tilde{Q}_{r,s-1} \langle \ul{\prescript{r+s}{G}^{\delta_s-a}} \rangle \langle \ul{\prescript{r+s}{G}^{a+1}} \rangle\\
-\sum_{a=0}^{\delta_s-1} \bb E \tilde{Q}_{r,s-1} \bb E \langle \ul{\prescript{r+s}{G}^{\delta_s-a}} \rangle \langle \ul{\prescript{r+s}{G}^{a+1}} \rangle+\frac{\delta_s}{N}\bb E \tilde{Q}_{r,s-1} \langle \ul{\prescript{r+s}{G}^{\delta_s+1}} \rangle 
+ 2\sum_{a=1}^{\delta_s-1}\tilde{Q}_{r,s-1}\langle \ul{\prescript{r+s}{G}^{\delta_s-a}}\rangle \bb E \ul{\prescript{r+s}{G}^{a+1}}\\+\frac{2}{N^2}\sum_{p=1}^{s-1}\delta_p\bb E \tilde{Q}_{r,s-2}^{(p)}\bb E \ul{\prescript{r+p}{G}^{\delta_p+1}\cdot \prescript{r+s}{G}^{\delta_s}}
 +\frac{2}{N^2}\sum_{p=1}^{s-1}\delta_p\bb E Q_{r,s-2}^{(p)}\langle \ul{G^{\delta_p+\delta_s+1}}\rangle\\+\frac{2}{N^2}\sum_{q=1}^{r}\sigma_q\bb E \tilde{Q}_{r-1,s-1}^{(q)}(\prescript{q}{G}^{\sigma_q+1}\cdot \prescript{r+s}{G}^{\delta_s})_{x_q y_q}
-\sum_{k=2}^{l}\tilde{Y}_k+\OO(N^{-(D+D')}\bigg)\,, 
\end{multline}
where $\tilde{Y}_k $ is defined analogously to $Y_k$ in \eqref{Qrs},
and $T(\prescript{r+s}{G})\deq -z-2\bb E \ul{\prescript{r+s}{G}}$ for $\ul{\prescript{r+s}{G}}=(H-z)^{-1}$. Note that for $1 \le p \le s-1$, in \eqref{tQrs}
\begin{equation*} 
\frac{2\delta_p}{N^2}\bb E \tilde{Q}_{r,s-2}^{(p)}\bb E \ul{\prescript{r+p}{G}^{\delta_p+1}\cdot \prescript{r+s}{G}^{\delta_s}}
\end{equation*}
does not belong to $\cal P(\{G,F^{*}\})$ if $\prescript{r+p}{G} \ne \prescript{r+s}{G}$. In this situation we apply the resolvent identity
\begin{equation} \label{reso}
GF^{*}=\frac{G-F^*}{z_1-z_2^*}=\frac{1}{(z_1-z^*_2)N^{\gamma}}N^{\gamma}(G-F^*)
\end{equation}
repeatedly until every term is in $\cal P(\{G,F^{*}\})$. In \eqref{reso}, $1/((z_1-z^*_2)N^{\gamma})$ is treated as a complex number, and it is bounded by our assumption on $z_1$ and $z_2^*$. Similarly, for $1\le q \le r$ we repeatedly apply \eqref{reso} to
\begin{equation*}
\frac{2}{N^2}\sigma_q\bb E \tilde{Q}_{r-1,s-1}^{(q)}(\prescript{q}{G}^{\sigma_q+1}\cdot \prescript{r+s}{G}^{\delta_s})_{x_q y_q}
\end{equation*} 
if $\prescript{q}{G} \ne \prescript{r+s}{G}$. Thus every term we get is now in $\cal P(\{G,F^*\})$. By plugging the result into \eqref{t4.44}, we rewrite \eqref{4.400} into a finite sum of terms in the form \eqref{4.4111} with an error $\OO(N^{-D})$. By exploring the terms in \eqref{tQrs} carefully and using Lemma \ref{lem1}, we see that for each $P'$ we get, both \eqref{4422} and \eqref{IG.Zz1tai} are satisfied. This finishes the proof for $\nu_4(P) \ne 0$.

\emph{Case 2.} Suppose $\nu_4(P) = 0$, then there is a factor $\bb E\tilde{Q}_{r}$ in $P$. If $\nu_2(P)\ne 0$, then $\max\limits_{1\le q \le r}\sigma_q \ge 2$, and we can assume $\sigma_r\ge 2$. If $\nu_2=0$, then $(\nu_5,\nu_6) \ne (0,0)$, and we can assume $(\prescript{r}{G}^{\sigma_r})_{x_r y_r}=\prescript{r}{G}_{x_r y_r}$ for some $x_r \ne y_r$. In both situations, we have
\begin{equation*}
\sum_{i_1,\ldots\,i_{n}} a_{i_1,\ldots,i_{n}}P_{i_1,\ldots,i_{n}}=\sum_{i_1,\ldots\,i_{n}} a_{i_1,\ldots,i_{n}}\cdot P_{i_1,\ldots,i_{n}}/\bb E\tilde{Q}_{r} \cdot\bb E \tilde{Q}_{r} \ \ \mbox{and} \ \ \sum_{i_1,\ldots\,i_{n}} \abs{a_{i_1,\ldots,i_{n}}P_{i_1,\ldots,i_{n}}/\bb E\tilde{Q}_{r}}=\OO(N^{D'})
\end{equation*}
for some fixed $D'\equiv D'(t,P,r)>0$. Let us denote $\tilde{Q}_{r-1}\deq \tilde{Q}_r/(\prescript{r}{G}^{\sigma_r})_{ x_ry_r}$, and $\tilde{Q}_{r-1}^{(q)}\deq \tilde{Q}_{r-1}/(\prescript{q}{G}^{\sigma_q})_{x_q y_q}$ for $1\le q\le r-1$. As in \eqref{Qr}, for some fixed $l \ge 2$ we have

\begin{equation} \label{tQr}
\begin{aligned}
\bb E \tilde{Q}_{r}&=\frac{1}{U(\prescript{r}{G})}\bigg(\bb E \tilde{Q}_{r-1}(\prescript{r}{G}^{\sigma_r-1})_{x_r y_r}+\sum_{a=0}^{\sigma_r-1}\bb E \tilde{Q}_{r-1}(\prescript{r}{G}^{\sigma_r-p})_{x_ry_r}\langle \ul{\prescript{r}{G}^{a+1}} \rangle +\sum_{a=1}^{\sigma_r-1}\bb E \tilde{Q}_{r-1}(\prescript{r}{G}^{\sigma_r-p})_{x_ry_r}\bb E \ul{\prescript{r}{G}^{a+1}}  \\
&\ \  +\frac{1}{N}\sum_{q=1}^{r-1}(\sigma_r+1)\sum_{a=0}^{\sigma_q-1}\bb E \tilde{Q}_{r-1}^{(q)}\Big((\prescript{q}{G}^{\sigma_q-a})_{x_q y _r}(\prescript{q}{G}^{a+1}\cdot \prescript{r}{G}^{\sigma_r})_{x_r y_q}+(\prescript{q}{G}^{\sigma_q-a})_{x_q x _r}(\prescript{q}{G}^{a+1}\cdot \prescript{r}{G}^{\sigma_r})_{y_r y_q}\Big)\\
&\ \ +\frac{\sigma_r+1}{N} \bb E \tilde{Q}_{r-1}(\prescript{r}{G}^{\sigma_r+1})_{x_r y_r}-\sum_{k=2}^{l}\tilde{Z}_k+\OO(N^{-D-D'})\bigg)\,,
\end{aligned}
\end{equation}
where $\tilde{Z}_k$ is defined analogously to $Z_k$ in \eqref{Qr},
and $U(\prescript{r+s}{G})\deq -z-\bb E \ul{\prescript{r+s}{G}}$ for $\prescript{r+s}{G}=(H-z)^{-1}$. Again in \eqref{tQr}, for $1 \le q \le r-1$ and $0 \le a \le \sigma_q-1$, we apply \eqref{reso} to 
\begin{equation*}
\frac{1}{N}(\sigma_r+1)\bb E \tilde{Q}_{r-1}^{(q)}\Big((\prescript{q}{G}^{\sigma_q-a})_{x_q y _r}(\prescript{q}{G}^{a+1}\cdot \prescript{r}{G}^{\sigma_r})_{x_r y_q}+(\prescript{q}{G}^{\sigma_q-a})_{x_q x _r}(\prescript{q}{G}^{a+1}\cdot \prescript{r}{G}^{\sigma_r})_{y_r y_q}\Big)
\end{equation*}
if $\prescript{q}{G} \ne \prescript{r}{G}$. Thus every term we get now is in $\cal P(\{G,F^*\})$. In this way we rewrite \eqref{4.400} into a finite sum of terms in the form \eqref{4.4111} with an error $\OO(N^{-D})$. Note that we either have $\sigma_r \ge 2$, or $(G^{\sigma_r-1})_{x_r y_r}=\delta_{x_r y_r}$ for $x_r \ne y_r$, and in the second case the two summations over $x_r$ and $y_r$ will become a single summation. By exploring the terms in right-hand side of \eqref{tQr} carefully and applying Lemma \ref{lem1}, we see that for each $P'$ we get, both \eqref{4422} and \eqref{IG.Zz1tai} are satisfied. This completes the proof for $\nu_4(P)=0$.

\subsection{The complex Hermitian $H$} \label{sec4.4} In this section we explain the proof of Proposition \ref{prop4.3} when $H$ is complex Hermitian. By the arguments in Sections \ref{sec4.2} and \ref{sec4.3}, we see that it is enough to show the following Lemma.
\begin{lemma}
Lemmas \ref{lem1}-\ref{fly the ocean in a silver plane} remain true when $H$ is a complex Hermitian Wigner matrix.
\end{lemma}
\begin{proof}
We shall only remark on the proof of Lemma \ref{lem5.3} in the complex case, and the rest follows similarly. By examining the steps in the proof of Lemma \ref{lem5.3}, we see that the main steps shall remain the same in the complex case. The only difference is that when calculating $\bb E \ul{G^m}$, $\bb E Q_{r,s}$, and $\bb E Q_r$, we shall use the complex cumulant expansion formula Lemma \ref{lem:5.1}, and the results will be different from \eqref{Gm}, \eqref{Qrs}, and \eqref{Qr} respectively. 

By Lemmas \ref{lem:cumulant_expansion} and \ref{lem:5.1}, the complex Hermitian analogue of \eqref{Gm} is
\begin{equation*} 
\bb E \ul{G^m}=\frac{1}{T}\bigg(\bb E \ul{G^{m-1}}+\sum_{a=1}^{m}\bb E \langle \ul{G^a}\rangle \langle\ul{G^{m+1-a}}\rangle+\sum_{a=2}^{m-1}\bb E \ul{G^a}\,\bb E\ul{G^{m+1-a}}-\sum_{k=2}^{l}X^c_{k}+\OO(N^{-D-D'})\bigg)\,,
\end{equation*}
where $l$ is some fixed integer,
\begin{equation*} 
X^c_k\deq \frac{1}{N} \sum_{i,j}\sum_{p+q=k}\frac{1}{p!\,q!(1+\delta_{ij})^k}\cal C_{p,q+1}(H_{ij}) \bb E \frac{\partial^k (G^m)_{ij}}{\partial H^p_{ij}\,\partial H^q_{ji}}\,,
\end{equation*}
and we recall that $T\deq -z_1-2\bb E \ul{G}$. The $(1+\delta_{ij})^{-k}$ factor in \eqref{494} is because for $i=j$, we need to use Lemma \ref{lem:cumulant_expansion} for the real random variables $H_{ii}$. Similarly, for $\bb EQ_{r,s}$ defined in \eqref{QG}, the complex Hermitian analogue of \eqref{Qrs} is
\begin{equation*} 
\begin{aligned}
\bb E Q_{r,s}&=\frac{1}{T} \bigg(\bb E Q_{r,s-1} \langle \ul{G^{\delta_s-1}} \rangle+
\sum_{a=0}^{\delta_s-1} \bb E Q_{r,s-1} \langle \ul{G^{\delta_s-a}} \rangle \langle \ul{G^{a+1}} \rangle-\sum_{a=0}^{\delta_s-1} \bb E Q_{r,s-1} \bb E \langle \ul{G^{\delta_s-a}} \rangle \langle \ul{G^{a+1}} \rangle\\
&\ \ + 2\sum_{a=1}^{\delta_s-1}Q_{r,s-1}\langle \ul{G^{\delta_s-a}}\rangle \bb E \ul{G^{a+1}}+\frac{1}{N^2}\sum_{p=1}^{s-1}\delta_p\bb E Q_{r,s-2}^{(p)}\bb E \ul{G^{\delta_p+\delta_s+1}}\\
&\ \ +\frac{1}{N^2}\sum_{p=1}^{s-1}\delta_p\bb E Q_{r,s-2}^{(p)}\langle \ul{G^{\delta_p+\delta_s+1}}\rangle+\frac{1}{N^2}\sum_{q=1}^{r}\sigma_q\bb E Q_{r-1,s-1}^{(q)}G_{x_q y_q}^{\sigma_q+\delta_s+1}-\sum_{k=2}^lY^c_k+\OO(N^{-D-D'})\bigg)\,,
\end{aligned}
\end{equation*} 
where 
\begin{equation*} 
Y^c_k\deq \frac{1}{N} \sum_{i,j}\sum_{p+q=k}\frac{1}{p!\,q!(1+\delta_{ij})^k}\cal C_{p,q+1}(H_{ij}) \bb E \frac{\partial^k(( Q_{r,s-1}-\bb E Q_{r,s-1})  (G^{\delta_s})_{ij})}{\partial H^p_{ij}\, \partial H^q_{ji}}\,.
\end{equation*}
 Moreover, for $Q_r$ defined in \eqref{4.23}, the complex Hermitian analogue of \eqref{Qrs} is
\begin{equation*} 
\begin{aligned}
\bb E Q_{r}&=\frac{1}{U}\bigg(\bb E Q_{r-1}(G^{\sigma_r-1})_{x_r y_r}+\sum_{a=0}^{\sigma_r-1}\bb E Q_{r-1}(G^{\sigma_r-p})_{x_r y_r}\langle \ul{G^{a+1}} \rangle +\sum_{a=1}^{\sigma_r-1}\bb E Q_{r-1}(G^{\sigma_r-p})_{x_ry_r}\bb E \ul{G^{a+1}}  \\
&\ \  +\frac{1}{N}\sum_{q=1}^{r-1}(\sigma_r+1)\sum_{a=0}^{\sigma_q-1}\bb E Q_{r-1}^{(q)}(G^{\sigma_q-a})_{x_q y _r}(G^{\sigma_r+a+1})_{x_r y_q}-\sum_{k=2}^{K}Z^c_k+\OO(N^{-D-D'})\bigg)\,,
\end{aligned}
\end{equation*}
where
\begin{equation*}
Z_k^c\deq \frac{1}{N} \sum_{i,j}\frac{1}{p!\,q!(1+\delta_{ij})^k}\cal C_{p,q+1}(H_{ij}) \sum_{p+q=k} \bb E \frac{\partial^k( Q_{r-1}  (G^{\sigma_r})_{ij})}{\partial H^p_{ij}\,\partial H^q_{ji}}\,.
\end{equation*} Note in the above complex cumulant expansions, we have fewer second cumulant terms comparing to the real case, while we need to track more carefully of the higher cumulants terms. By applying the differential rule \eqref{4.9}, one easily checks that the argument in the real case also applies here, and the proof follows easily.   
\end{proof}		 

\section{The non-conjugate case} \label{sec7}
In this section we prove the following result, where the spectral parameters $z_1$ and $z_2^*$ in Theorem \ref{thm_resolvent} are replaced with $z_1$ and $z_2$, both having positive imaginary parts.

\begin{proposition} \label{prop6.1}
Under the conditions of Theorem \ref{thm_resolvent}, we have the following results. For the real symmetric case $(\beta=1)$, we have
	\begin{equation} \label{real7}
	\bb E \langle \ul{G}(z_1)\rangle \langle \ul{G}(z_2)\rangle= \frac{1}{N^2}\Big(f_{2-}(z_1,z_2)+f_3(z_1,z_2)\sum_{i,j}\cal C_4(H_{ij})+f_4(z_1,z_2)\sum_i\cal C_3(H_{ii})\Big)
	+\cal E_1\,,
	\end{equation}
and for the complex Hermitian case $(\beta=2)$, we have
	\begin{equation} \label{ankai}
	\bb E \langle \ul{G}(z_1)\rangle \langle \ul{G}(z_2)\rangle= \frac{1}{N^2}\Big(\,\frac{1}{2}f_{2-}(z_1,z_2)+f_3(z_1,z_2)\sum_{i,j}\cal C_{2,2}(H_{ij})+f_4(z_1,z_2)\sum_i\cal C_3(H_{ii})\Big)+
	\cal E_1\,.
	\end{equation}
	Here $f_{2-}(z_1,z_2)$, $f_3(z_1,z_2)$ and $f_4(z_1,z_2)$ are bounded terms defined in \eqref{488} and \eqref{f34}.
\end{proposition}
In the rest of this section we denote $G=G(z_1)$ and $F=G(z_2)$. We start with some notations in addition to what was introduced in Section \ref{sec3}. Let $\cal A=\{G,F\}$, we define the following series of operators  $\tilde\nu_1,\ldots,\tilde\nu_6\col \cal P(\cal A)\to \N$ for each $P \in \cal P$ by 
\begin{enumerate}
\item $\tilde\nu_{1} (P)$ - the sum of the degrees $m+n-1$ of $ \ul {G^mF^n} $ in $P$.
\item $\tilde\nu_2(P)$ - the sum of the degrees $m+n-1$ of  $\langle \ul {G^mF^n} \rangle$ and $(G^mF^n)_{x y}$ in $P$.
\item $\tilde\nu_3(P)$ - the total number of $\ul{G^mF^n}$ in $P$ for $m+n \ge 2$.
\item $\tilde\nu_4(P)$ - the total number of $\langle \ul{G^mF^n} \rangle$ in $P$.
\item $\tilde\nu_5(P)$ - the number of $i_k$ in $\{i_1,i_2,\ldots\}$ such that the total number of appearance of $i_k$ in all factors $(G^mF^n)_{x y}$ of $P$ is odd.
\item $\tilde\nu_6(P)$ - the number of $i_k$ in $\{i_1,i_2,\ldots\}$ such that the total number of appearance of $i_k$ in all factors $(G^mF^n)_{x y}$ of $P$ is even, and $i_k$ appears in at least one $(G^mF^n)_{x y}$ for $x \ne y$.
\end{enumerate}
The following lemma is an analogue of Proposition \ref{prop4.3}.
\begin{lemma}  \label{lem6.2}
Let $(a_{i_1,\ldots,i_n})_{i_1,\ldots,i_n}$ be a family of complex numbers that is uniformly bounded in $i_1,\ldots,i_n$, and let  $\cal A=\{G,F\}$. Fix $P \in P^{(n,t)}(\cal A)$ for some $n \in \bb N$ and $t \in \bb R$. We have
\begin{equation*} 
\sum_{i_1,\ldots,i_n}a_{i_1,\ldots,i_n}P_{i_1,\ldots,i_n} = \OO(N^{t+n+\tilde b(P)})\,,
\end{equation*}
where $b(P) \deq -\tilde \nu_4(P)-(\tilde \nu_5(P)+\tilde \nu_6(P))/2$.
\end{lemma}
The proof of Lemma \ref{lem6.2} is very close to that of Proposition \ref{prop4.3}(i). Instead of considering matrices $G^m$, we consider $G^mF^n$ instead, and we never apply the resolvent identity $GF=(G-F)/(z_1-z_2)$ to separate $G$ and $F$. We omit the details.
\subsection{The real case} \label{sec6.1} In this section we prove \eqref{real7}. The proof is very similar to that of Theorem \ref{thm_resolvent}, namely we will repeatedly use the cumulant expansion formula Lemma \ref{lem:cumulant_expansion} and estimate the results by Lemma \ref{lem6.2}. As in \eqref{417}, we can apply Lemma \ref{lem:cumulant_expansion} and get
\begin{equation} \label{617}
\bb E \langle \ul{G} \rangle \langle \ul{F} \rangle=\frac{1}{T}\bigg(\bb E \langle \ul{G} \rangle^2 \langle \ul{F} \rangle+\frac{1}{N}\bb E \langle\ul{G^2}\rangle \langle \ul{F}\rangle+\frac{2}{N^2}\bb E \ul{GF^{2}}- W^{(12)}_2-W^{(12)}_3\bigg)+\OO(N^{-3})\,,
\end{equation}
where $T=-z_1-2\bb E \ul{G}$, and
\begin{equation} \label{W12k}
W^{(12)}_k\deq \frac{1}{N} \sum_{i,j}\frac{1}{k!}\cal C_{k+1}(H_{ji}) \bb E \frac{\partial^k (G_{ij}\langle\ul{F}\rangle)}{\partial H^k_{ji}}\,.
\end{equation}
From Lemma \ref{lem6.2} we see that 
\begin{equation} \label{666laotieshuangji66666}
\frac{1}{T} \bb E \langle \ul{G} \rangle^2 \langle \ul{F} \rangle =\OO(N^{-3})  \ \ \ \mbox{and} \ \ \ \frac{1}{TN}\bb E \langle \ul{G^2}\rangle \langle \ul{F} \rangle =\OO(N^{-3})\,,
\end{equation}
thus it remains to estimate the third, fourth, and fifth term on the right-hand side of \eqref{617}.

We first look at $2/(TN^2) \bb E \ul{GF^2}$, and the proof is similar to how we calculated $2/(TN^2) \bb E \ul{GF^{*2}}$ in Lemma \ref{lem4.5}. By Lemma \ref{lem4.3} and resolvent identitywe have
\begin{equation} \label{610}
\frac{2}{TN^2} \bb E \ul{GF^2}=\frac{2}{N^2}\frac{1}{-z_1-2\bb E \ul{G}}\bigg(\frac{\bb E \ul{G}-\bb E\ul{F}}{(z_1-z_2)^2}-\frac{\bb E \ul{F^2}}{(z_1-z_2)}\bigg)=\frac{1}{N^2}f_{2-}(z_1,z_2)+\cal E_1
\end{equation}
The computation of $-T^{-1}W^{(12)}_2$ is similar to that of $-T^{-1}W^{(1)}_2$ in Lemma \ref{lem4.6}, and we have
\begin{equation}
-\frac{1}{T}W^{(12)}_2=\frac{1}{N^2}f_4(z_1,z_2)\sum_{i}\cal C_3(H_{ii})+\cal E_1\,.
\end{equation}
The computation of $-T^{-1}W^{(12)}_3$ is similar to that of $-T^{-1}W^{(1)}_3$ in Lemma \ref{lem4.7}, and we have
\begin{equation}  \label{8.8}
-\frac{1}{T}W^{(12)}_3=\frac{1}{N^2}f_3(z_1,z_2)\sum_{i,j}\cal C_4(H_{ij})+\cal E_1\,.
\end{equation}
By inserting \eqref{W12k}-\eqref{8.8} into \eqref{617} we complete the proof.

\subsection{The complex case} \label{sec6.2} In this section we prove \eqref{ankai}. The proof is very similar to that of Theorem \ref{thm_resolvent}, namely we will repeatedly use the cumulant expansion formula Lemmas \ref{lem:cumulant_expansion} and \ref{lem:5.1}, and estimate the results by Lemma \ref{lem6.2}. As in \eqref{498}, we can apply Lemmas \ref{lem:cumulant_expansion} and \ref{lem:5.1} to get
\begin{equation} \label{614}
\bb E \langle \ul{G} \rangle \langle \ul{F} \rangle=\frac{1}{T}\bigg(\bb E \langle \ul{G} \rangle^2 \langle \ul{F} \rangle+\frac{1}{N^2}\bb E \ul{GF^{2}}- W^{(13)}_2-W^{(13)}_3\bigg)+\OO(N^{-3})\,,
\end{equation}
where $T=-z_1-2\bb E \ul{G}$, and
\begin{equation*} 
W^{(13)}_k\deq \frac{1}{N} \sum_{i,j}\sum_{p+q=k}\frac{1}{p!\,q!(1+\delta_{ij})^k}\cal C_{p,q+1}(H_{ij}) \bb E \frac{\partial^k (G_{ij}\langle\ul{F}\rangle)}{\partial H^p_{ji} \partial H^q_{ji}}\,.
\end{equation*}
Lemma \ref{lem6.2} implies
\begin{equation} \label{666laotieshuangji6666666666}
\frac{1}{T} \bb E \langle \ul{G} \rangle^2 \langle \ul{F} \rangle =\OO(N^{-3})  \,,
\end{equation}
thus it remains to estimate the third, fourth, and fifth term on the right-hand side of \eqref{614}. As in Section \ref{sec6.1}, we have
\begin{equation}
\frac{1}{TN^2}\bb E \ul{GF^{2}}=\frac{1}{2N^2}f_{2-}(z_1,z_2)+\cal E_1\,,
\end{equation}
\begin{equation}
-\frac{1}{T}W^{(13)}_2= \frac{1}{N^2}f_4(z_1,z_2)\sum_{i}\cal C_3(H_{ii})+\cal E_1\,,
\end{equation}
and
\begin{equation} \label{622}
-\frac{1}{T}W^{(13)}_3=\frac{1}{N^2}f_3(z_1,z_2)\sum_{i,j}\cal C_{2,2}(H_{ij})+\cal E_1\,.
\end{equation}
Recall the definitions of $f_{2-}(z_1,z_2),f_3(z_1,z_2),$ and $f_4(z_1,z_2)$ in \eqref{488} -- \eqref{f34}. Inserting \eqref{666laotieshuangji6666666666} -- \eqref{622} into \eqref{614} completes the proof.

\appendix 
\section{Comparison to Gustavsson's theorem} \label{sec:Gustavsson}
In this section we compare our result to the ones in \cite{G2005} and \cite{O2010}. Roughly, the conclusion is that the question addressed in \cite{G2005,O2010} is independent from the one addressed in our work, and neither implies the other.

In Gustavsson's work \cite{G2005}, the fluctuation of a single eigenvalue was first established to be Gaussian for GUE, and it was also showed that the joint limit distribution of eigenvalues is a Gaussian process, provided that the eigenvalue are separated by a mesoscopic distance. In \cite{O2010}, the result was extended to GOE and, by moment matching, to a class of Wigner matrices.

For the rest of the section let us adopt the assumptions in Theorem \ref{mainthm}, and in addition we assume $H$ is GOE. Recall the notations $\alpha \deq -\log_N \eta$ and $\gamma \deq -\log_N \omega$. We would like to compute $\int p_{E}(u,v) f_{+}(u)g_{-}(v) \,\mathrm{d}u\,\mathrm{d}v$. By \eqref{n5.3} we have
\begin{equation} \label{A.1}
\int p_{E}(u,v) f_{-}(u)g_{+}(v) \,\mathrm{d}u\,\mathrm{d}v=\sum_{i,j}\frac{1}{N^2\eta^2}\bb E \Big\langle f\Big(\frac{(\lambda_i-E)\varrho_E-\omega}{\eta}\Big) \Big \rangle \Big\langle g\Big(\frac{(\lambda_j-E)\varrho_E+\omega}{\eta}\Big) \Big \rangle\,.
\end{equation}
Let $\lambda_1\leq\cdots\leq\lambda_N$ be the eigenvalues of $H$, and for $k=1,2,...,N$, let the quantile $\gamma_k$ be the typical location of $\lambda_k$, i.e.\ it satisfies $k/N=\int_{-2}^{\gamma_k} \varrho_x \dd x$. By our assumptions $E \in [-2+\tau,2-\tau]$, $f,g \in C^{\infty}_c(\bb R)$, and a standard eigenvalue rigidity result (e.g.\ Theorem 2.2,\cite{EYY3}), we see that the sum in \eqref{A.1} can be reduced to
\begin{multline*} 
\int p_{E}(u,v) f_{-}(u)g_{+}(v) \,\mathrm{d}u\,\mathrm{d}v
\\
=\sum_{i\in A,j \in B}\frac{1}{N^2\eta^2}\bb E \Big\langle f\Big(\frac{(\lambda_i-E)\varrho_E-\omega}{\eta}\Big) \Big \rangle \Big\langle g\Big(\frac{(\lambda_j-E)\varrho_E+\omega}{\eta}\Big) \Big \rangle+\OO(N^{-10})\,,
\end{multline*}
where $A \deq \{i\col |(\gamma_i-E)\varrho_E-\omega| \leq \eta\log N\}$, $B \deq \{j\col |(\gamma_j-E)\varrho_E+\omega| \leq \eta\log N\}$. Clearly for $i \in A$ and $j\in B$, we have $i-j \asymp N\omega= N^{1-\gamma}$, and Theorem 5 of \cite{O2010} shows that
\begin{equation} \label{A.3}
\bigg(\frac{\pi \varrho_{\gamma_i}N(\lambda_i-\gamma_i)}{\sqrt{\log N}}, \frac{\pi \varrho_{\gamma_j}N(\lambda_j-\gamma_j)}{\sqrt{\log N}}\bigg)
\end{equation}
converges weakly to the Gaussian random vector $N(\f \mu,\Sigma)$ with $$\f \mu=
\begin{pmatrix}
0 \\ 0
\end{pmatrix}
\ \ \  \mbox{ and } \ \ \ \Sigma=\bigg(\begin{array}{cc}
1 & \gamma \\
\gamma & 1 \\
\end{array}\bigg)\,,$$ 
as $N \to \infty$. Also, by $|\gamma_i-E|=\OO(\omega)$ we see that $\varrho_E=\varrho_{\gamma_i}+\OO(\omega)$. Hence the random vector
\begin{equation} \label{A.4}
\bigg(\frac{\pi \varrho_{E}N(\lambda_i-\gamma_i)}{\sqrt{\log N}}, \frac{\pi \varrho_{E}N(\lambda_j-\gamma_j)}{\sqrt{\log N}}\bigg)
\end{equation}
has the same weak limit as the one in \eqref{A.3}. With the rescaling $u_k\deq (\lambda_k-E)N\varrho_E$ for $k=1,2,...,N$, we have
\begin{equation} \label{A.5}
\int p_{E}(u,v) f_{-}(u)g_{+}(v) \,\mathrm{d}u\,\mathrm{d}v=\sum_{i\in A,j \in B}\frac{1}{N^2\eta^2}\bb E \Big\langle f\Big(\frac{u_i-N\omega}{N\eta}\Big) \Big \rangle \Big\langle g\Big(\frac{u_j+N\omega}{N\eta}\Big) \Big \rangle+\OO(N^{-10})\,,
\end{equation}
Now let us compute the right-hand side of \eqref{A.5} by using the weak limit of \eqref{A.4}, i.e.\ that the random vector $(u_i,u_j)$ has asymptotically (as $N \to \infty$) the same law as the Gaussian vector $(\tilde u_i, \tilde u_j) \eqdist N(\f \mu', \Sigma')$, where
$$
\f \mu'=
\begin{pmatrix}
(\gamma_i-E)N\varrho_E \\ (\gamma_j-E)N\varrho_E
\end{pmatrix}
\eqd
\begin{pmatrix}
m_i \\ m_j
\end{pmatrix}
\ \ \mbox{\ and\ }\ \ \ \Sigma'=\frac{\log N}{\pi^2}\bigg(\begin{array}{cc}
1 & \gamma \\
\gamma & 1 \\
\end{array}\bigg)\,.
$$
By Taylor expansion, we have 
\begin{multline} \label{AA}
\sum_{i\in A,j \in B}\frac{1}{N^2\eta^2}\bb E \Big\langle f\Big(\frac{u_i-N\omega}{N\eta}\Big) \Big \rangle \Big\langle g\Big(\frac{u_j+N\omega}{N\eta}\Big) \Big \rangle\\
=\hspace{-0.2cm}\sum_{i\in A,j \in B}\frac{1}{N^2\eta^2}\bb E \bigg\langle f\Big(\frac{m_i-N\omega}{N\eta}\Big)+\frac{1}{N\eta}f'\Big(\frac{m_i-N\omega}{N\eta}\Big)(u_i-m_i)
+\frac{1}{2N^2\eta^2}f''\Big(\frac{m_i-N\omega}{N\eta}\Big)(u_i-m_i)^2+\cdots \bigg \rangle\\
\cdot \bigg\langle g\Big(\frac{m_j+N\omega}{N\eta}\Big)+\frac{1}{N\eta}g'\Big(\frac{m_j+N\omega}{N\eta}\Big)(u_j-m_j)
+\frac{1}{2N^2\eta^2}g''\Big(\frac{m_j+N\omega}{N\eta}\Big)(u_j-m_j)^2+\cdots \bigg\rangle\,.
\end{multline}
The first non-vanishing term is the one proportional to $f'g'$. We now compute it asymptotically using Gustavsson's theorem, by replacing $(u_i,u_j)$ with $(\tilde u_i, \tilde u_j)$. The result is
\begin{align} \label{A.6}
&\mspace{-20mu}\sum_{i\in A,j \in B}\frac{1}{N^4\eta^4}\bb E \bigg\langle f'\Big(\frac{m_i-N\omega}{N\eta}\Big)(\tilde u_i-m_i) \bigg \rangle \bigg\langle g'\Big(\frac{m_j+N\omega}{N\eta}\Big)(\tilde u_j-m_j) \bigg \rangle \notag \\
&=\sum_{i\in A,j \in B}\frac{\log N}{N^4\eta^4\pi^2} f'\Big(\frac{m_i-N\omega}{N\eta}\Big) g'\Big(\frac{m_j+N\omega}{N\eta}\Big)
\notag \\
&=\sum_{j \in B}\frac{\log N}{N^3\eta^3\pi^2}g'\Big(\frac{m_j+N\omega}{N\eta}\Big)\int(1+\OO(\omega))f'(x)\, \dd x+\OO(N^{-10})
\notag \\
&=\frac{\log N}{N^2\eta^2\pi^2}\int(1+\OO(\omega))f'(x)g'(y) \,\dd x\, \dd y+\OO(N^{-10})\prec N^{2\alpha-\gamma-2}\,,
\end{align}
where in the second and third step we used the fact that $m_{k+1}-m_k=1+\OO(\omega)$ for $k \in A\cup B$, and the Riemann sum of a smooth compactly supported function converges to its integral to any polynomial order (by the Poisson summation formula). Similarly, we can show that all other terms in \eqref{AA} are bounded by $\OO_{\prec}(N^{4\alpha-4})$. Thus, under the assumption that we can use Gustavsson's theorem and the error from the replacement of $(u_i, u_j)$ with the Gaussian vector $(\tilde u_i, \tilde u_j)$ is negligible, we get, together with \eqref{A.5},
\begin{equation*}
\int p_{E}(u,v) f_{-}(u)g_{+}(v) \prec N^{2\alpha-\gamma-2}+N^{4\alpha-4}\,.
\end{equation*}
We find that the right-hand side is much smaller than $N^{2\gamma-2}$ for $\alpha>3/4$ and $\gamma>2\alpha/3$. However, by Theorem \ref{mainthm} we know that the leading term in $\int p_{E}(u,v) f_{-}(u)g_{+}(v)$ is of order $N^{2\gamma-2}$. We conclude that the use of Gustavsson's theorem leads to a wrong result in at least the regime $\alpha>3/4$ and $\gamma>2\alpha/3$. This argument shows that the computation of $\int p_{E}(u,v) f_{-}(u)g_{+}(v)$ using Gustavsson's theorem leads to the wrong result in a rather flagrant fashion: the obtained quantity is much smaller than than the true value. In fact, we believe that more generally Gustavsson's result cannot be used to recover the correct mesoscopic density-density correlations in any regime: the error made between \eqref{AA} and \eqref{A.6} is never affordable.

The issue lies in the fact that, in order to compute sums of the form \eqref{AA}, even to leading order, one needs much stronger control on the joint distribution of the individual eigenvalues than is provided in \cite{G2005} and \cite{O2010}. This is a manifestation of the fact, mentioned in the introduction, that eigenvalue density-density correlations are much weaker than their location-location correlations. Going from the latter to the former consequently requires very precise asymptotics for the latter.

\section{Comparison to results on macroscopic linear statistics} \label{sec:linstat}

In this appendix we give the short calculation that shows how our main result, Theorem \ref{mainthm}, recovers the well-known covariance formula for macroscopic linear statistics of Wigner matrices \cite{LP,KKP96}. This corresponds to the extreme case $\omega \asymp \eta \asymp 1$ in Theorem \ref{mainthm}.

Let for shortness $f^{\eta}(x)=\eta^{-1}f(((x-E)\varrho_E-\omega)/\eta)$ and $g^{\eta}(x)=\eta^{-1}g(((x-E)\varrho_E+\omega)/\eta)$. It follows from \eqref{2.6}, \eqref{FFF} and \eqref{n5.3} that there exists $c=c(\tau)>0$ such that
\begin{equation*} 
\begin{aligned}
&\,\frac{1}{N^2}\cov (\tr f^\eta(H);g^{\eta}(H)) =\int p_{E}(u,v) f_{+}(u)g_{-}(v)  \,\mathrm{d}u\,\mathrm{d}v\\
=&\,\int \Big(-\frac{1}{\pi^2(u-v)^2}+\frac{F_1(u,v)}{N^2\kappa_E^2}+\frac{F_2(u,v)}{N^2\kappa_E^2}\sum_{i,j}\cal C_4(H_{ij})\Big) f_{+}(u)g_{-}(v)  \,\mathrm{d}u\,\mathrm{d}v+\OO(N^{-2-c})+\OO\Big({\frac{1}{N^4\omega^4}}\Big)\\
=&\,\frac{1}{\pi^2}\int \bigg(-\frac{1}{N^2(x_1-x_2)^2}\cdot\frac{4-x_1x_2}{\sqrt{4-x_1^2}\sqrt{4-x_2^2}}+\frac{(x_1^2-2)(x_2^2-2)}{2N^2\sqrt{(4-x_1^2)(4-x_2^2)}}\sum_{i,j}\cal C_4(H_{ij})\bigg) f^{\eta}(x_1)g^{\eta}(x_2)\\
&+\OO\Big(N^{-2-c}+{\frac{1}{(u-v)^4}}\Big)\,.
\end{aligned}
\end{equation*}
Here $F_1,F_2$ are defined as in \eqref{FFF}. We see that when $\omega \gg N^{-1/2}$ the integral RHS of the above equation dominates, and it matches Equation (VI.81) in \cite{KKP96}. 

\section{Remark on the distribution of the diagonal entries} \label{sec:diag_general}

Theorems \ref{mainthm} and \ref{thm_resolvent} can be extended to a more general case where we relax the assumption on the variance of $H_{ii}$ in Definition \ref{def:wigner}. Here we summarize the appropriate generalizations of Theorems \ref{thm_resolvent} and \ref{thm:main}, whose proofs are simple modifications of the arguments given in the previous sections. This generalization illustrates the great sensitivity of our results to the precise distribution of the matrix entries.

Let $\tilde{H}$ be a Wigner matrix where the condition $\bb E |\sqrt{N}H_{ii}|^2=2/\beta$ is replaced by $\bb E|\sqrt{N}\tilde{H}_{ii}|^2=\zeta_i$, and $\max_i \zeta_i  = \OO(1)$. Let $\tilde{G}(z)\deq (\tilde{H}-z)^{-1}$, and let $\tilde{\Upsilon}_{E,\beta} $ be the analogue of $\Upsilon_{E,\beta}$ defined in \eqref{2.6}--\eqref{2.7}.

Theorem \ref{thm_resolvent} is generalized as follows. For the real symmetric case ($\beta=1$), we have 
\begin{multline*}
\cov \big(\tilde{\ul{G}}(z_1), \tilde{\ul{G}}(z_2^*) \big)
\\
=\cov \big(\ul{G}(z_1), \ul{G}(z_2^*) \big)+ \bigg(\frac{-2+2\ii\im m(E)^2}{N^4(z_1-z_2^*)^2}+\frac{m(z_1)m(z_2^*)}{N^3\sqrt{z_1^2-4}\sqrt{z_2^{*2}-4}}\bigg)\sum_i(\zeta_i-2)+\cal E_1\,,
\end{multline*}
and for the complex Hermitian case ($\beta=2$), we have
\begin{multline*}
\cov \big(\tilde{\ul{G}}(z_1), \tilde{\ul{G}}(z_2^*) \big)
\\
=\cov \big(\ul{G}(z_1), \ul{G}(z_2^*) \big)+ \bigg(\frac{-1+\ii\im m(E)^2}{N^4(z_1-z_2^*)^2}+\frac{m(z_1)m(z_2^*)}{N^3\sqrt{z_1^2-4}\sqrt{z_2^{*2}-4}}\bigg)\sum_i(\zeta_i-1)+\cal E_1\,.
\end{multline*}

Theorem \ref{thm:main} is generalized as follows.
We have
\begin{equation*}
\tilde{\Upsilon}_{E,1}(u,v)=\Upsilon_{E,1}(u,v)+\bigg(\frac{F_4(u,v)}{N^3\kappa_E^2}-\frac{1}{N^2\pi^2(u-v)^2}\bigg)\sum_i(\zeta_i-2)
+\cal E
\end{equation*}
as well as
\begin{equation*}
\tilde{\Upsilon}_{E,2}(u,v)=\Upsilon_{E,2}(u,v)+\bigg(\frac{F_4(u,v)}{N^3\kappa_E^2}-\frac{1}{2N^2\pi^2(u-v)^2}\bigg)\sum_i(\zeta_i-1)
+\cal E\,.
\end{equation*}
where $$F_4(u,v)=g_4\Big(\frac{u-E}{N\varrho_{E}},\frac{v-E}{N\varrho_{E}}\Big)\ \ \ \mbox{ and }\ \ \ g_4(x_1,x_2)=\frac{x_1x_2}{\sqrt{4-x_1^2}\sqrt{4-x_2^2}
}\,.$$

{\small

\bibliography{bibliography} 

\bibliographystyle{amsplain}
}

\paragraph{Acknowledgements}
We gratefully acknowledge funding from the European Research Council (ERC) under the European Union's Horizon 2020 research and innovation programme (grant agreement No.\ 715539\_RandMat) and from the Swiss National Science Foundation through the SwissMAP grant.

\bigskip

\noindent
Yukun He, University of Z\"urich, Institute for Mathematics, Winterthurerstrasse 190, 8057 Z\"urich, Switzerland.
Email: \href{mailto:yukun.he@math.uzh.ch}{yukun.he@math.uzh.ch}.
\\[1em]
Antti Knowles, 
University of Geneva, Section of Mathematics, 2-4 Rue du Li\`evre, 1211 Gen\`eve 4, Switzerland. 
Email:
\href{mailto:antti.knowles@unige.ch}{antti.knowles@unige.ch}.

\end{document}